\documentclass{article}
\usepackage{titlesec}
\usepackage{titling}
\titleformat{\chapter}{\normalfont\huge}{\thechapter.}{20pt}{\huge\it}
\renewcommand{\paragraph}{\subsection}
\usepackage[utf8]{inputenc}
\usepackage[T1]{fontenc}
\usepackage{amsmath,amssymb,amsthm,mathrsfs,calligra,amsfonts}
\usepackage{enumerate}
\usepackage[framemethod=tikz]{mdframed}
\usepackage{array}
\usepackage{lipsum}
\usepackage{mathtools}
\usepackage{scrextend}
\usepackage{tikz-cd}
\usepackage{cite}
\usepackage{yhmath}
\usepackage{pdfrender,xcolor}
\usepackage{lmodern}
\usepackage{thmtools}
\usepackage{geometry}
\usepackage[parfill]{parskip}    
\usepackage{graphicx}
\usepackage{epstopdf}
\usepackage[colorlinks]{hyperref}
\usepackage[noabbrev,capitalize]{cleveref}
\crefformat{section}{\S#2#1#3}
\theoremstyle{plain}
\newtheorem{theorem}{Theorem}[subsection]   
\newtheorem{corollary}[theorem]{Corollary}
\newtheorem{lemma}[theorem]{Lemma}

\newtheorem{proposition}[theorem]{Proposition}
\theoremstyle{definition}
\newtheorem{definition}[theorem]{Definition}

\newtheorem{example}[theorem]{Example}
\newtheorem{conjecture}[theorem]{Conjecture}

\theoremstyle{remark}
\newtheorem{remark}[theorem]{Remark}
\titleformat*{\subsection}{\itshape\mdseries}

\newcommand{\Em}{\widehat{\mathscr{E}}_{\Omega_\mr{X}}}
\newcommand{\Db}{{\mathrm{\bf{D}}}^{\mr{b}}}

\newcommand{\mr}{\mathrm}

\newcommand{\Aa}{\mr{A}_{\infty}}
\newcommand{\Z}{\mathbf Z}
\newcommand{\Ww}{\widehat{\mathscr{W}}_{\mr{X}}}

\newcommand{\Q}{\mathbf Q}
\newcommand{\C}{\mathbf C}

\newcommand{\MS}{\mr{MS}}
\newcommand{\Cc}{\mathcal C}
\newcommand{\brak}{[\![\hbar]\!]}
\newcommand{\cbrak}{(\!(\hbar)\!)}

\newcommand{\2}{\vspace{2mm}}

\newcommand{\E}{\mathscr E}

\newcommand{\F}{\mathscr{F}}

\newcommand*{\shom}{\mathscr{H}\text{\kern -3pt {\it{om}}}\,}
\newcommand*{\ext}{\mathscr{E}\text{\kern -1.5pt {\it{xt}}}\,}
\newcommand*{\tor}{\mathscr{T}\text{\kern -3pt {\it{or}}}\,}

\makeatletter
 	
\newcommand{\xdashrightarrow}[2][]{\ext@arrow 0359\rightarrowfill@@{#1}{#2}}
\newcommand{\xdashleftarrow}[2][]{\ext@arrow 3095\leftarrowfill@@{#1}{#2}}
\newcommand{\xdashleftrightarrow}[2][]{\ext@arrow 3359\leftrightarrowfill@@{#1}{#2}}
\def\rightarrowfill@@{\arrowfill@@\relax\relbar\rightarrow}
\def\leftarrowfill@@{\arrowfill@@\leftarrow\relbar\relax}
\def\leftrightarrowfill@@{\arrowfill@@\leftarrow\relbar\rightarrow}
\def\arrowfill@@#1#2#3#4{%
  $\m@th\thickmuskip0mu\medmuskip\thickmuskip\thinmuskip\thickmuskip
   \relax#4#1
   \xleaders\hbox{$#4#2$}\hfill
   #3$%
}
\makeatother

\title{\vspace{-1cm}Differential graded categories in holomorphic symplectic geometry}
\author{Borislav Mladenov}

\AtEndDocument{\bigskip{\footnotesize%
  \text{Institute of Mathematics, Academia Sinica, Taipei, Taiwan} \\
  \href{mailto:mladenov@as.edu.tw}{\textit{mladenov@as.edu.tw}}
  }}
\predate{}
\date{}
\postdate{}
\begin{document}
\maketitle
\nocite{*}
\begin{abstract}
Let $(\mr{X},\sigma)$ be a holomorphic symplectic manifold. We study the differential graded category of canonical Lagrangian $\mr{D}$-branes $\mathcal{D}_\mr{Lag}(\mr{X},\sigma)$ along with its deformation quantisation, spanned by quantised orientations, $\mathcal{DQ}(\mr{X},\sigma)$, and the virtual de Rham category $\mathcal{DR}^{\mr{vir}}(\mr{X},\sigma)$.\\
$\hspace*{4mm}$We prove the formality of these dg categories when localised at a countable collection of orientable compact K\"{a}hler Lagrangian submanifolds with pairwise clean intersections. \\
$\hspace*{4mm}$Along the way, we define Kaledin classes of minimal $\Aa$-categories and show that they are the obstructions to formality. In addition, we obtain a formality criterion for flat weakly proper Calabi-Yau dg categories.
\end{abstract}
\tableofcontents

\section{Introduction}

Let $(\mr{X},\sigma)$ be a holomorphic symplectic manifold. The main goal of this paper is the study of various differential graded categories associated to $\mr{X}$ and its Lagrangian submanifolds.

\subsection{Lagrangian \texorpdfstring{$\mr{D}$}--branes}

Let $\mr{L}\subset \mr{X}$ be a compact K\"{a}hler Lagrangian submanifold equipped with a choice of a square root $\mr{K}_\mr{L}^{1/2}$ of its canonical bundle. Recall that there is a spectral sequence $$\mr{E}_2=\mr{H}(\mr{L}/\C)\Rightarrow \mr{Ext}_{\mathscr{O}_\mr{X}}\left(\mr{K}_\mr{L}^{1/2},\mr{K}_\mr{L}^{1/2}\right).$$
It was shown in \cite{MR4678893} that it collapses on the second page and moreover that the differential graded algebra $$\mr{RHom}_{\mathscr{O}_\mr{X}}\left(\mr{K}_\mr{L}^{1/2},\mr{K}_\mr{L}^{1/2}\right)$$ is formal and quasi-isomorphic to the de Rham algebra of $\mr{L}$.\\
These results were motivated by mirror symmetry and the latter suggests that we should be working with categories rather than single objects. In the present article, we shall upgrade the above to the level of differential graded categories.\\
Let $\mathcal{D}_\mr{Lag}(\mr{X},\sigma)$ be the full differential graded subcategory of $\Db_\mr{dg}(\mr{X})$ spanned by square-roots of the canonical bundles of Lagrangian submanifolds in $\mr{X}$. Thus:
\begin{itemize}
\item Objects in $\mathcal{D}_\mr{Lag}(\mr{X},\sigma)$ are choices of square-roots $\mr{K}_\mr{L}^{1/2}$ where $\mr{L}$ is a Lagrangian submanifold in $\mr{X}$;
\item For a pair of Lagrangian submanifolds $\mr{L}$ and $\mr{M}$ and two objects associated with these Lagrangians $\mr{K}_\mr{L}^{1/2}$ and $\mr{K}_\mr{M}^{1/2}$, the morphism spaces are given by the complexes $$\mathcal{D}_\mr{Lag}\left(\mr{K}_\mr{L}^{1/2},\mr{K}_\mr{M}^{1/2}\right)=\mr{RHom}_{\mathscr{O}_\mr{X}}\left(\mr{K}_\mr{L}^{1/2},\mr{K}_\mr{M}^{1/2}\right).$$
\end{itemize}
We introduce a local version of this category. Let $\mathfrak{L}$ be a collection of orientable Lagrangian submanifolds in $\mr{X}$. We denote by $\mathcal{D}_\mr{Lag}(\mathfrak{L})$ the full subcategory of $\mathcal{D}_\mr{Lag}(\mr{X},\sigma)$ spanned by objects supported in $\mathfrak{L}$. Our first result goes as follows:
\2
\begin{theorem}[\cref{main1proof}]\label{main1}
Let $(\mr{X},\sigma)$ be a holomorphic symplectic manifold and let $\mathfrak{L}$ be a (countable) collection of orientable compact K\"{a}hler Lagrangian submanifolds with clean pairwise intersections. Then the differential graded category $\mathcal{D}_\mr{Lag}(\mathfrak{L})$ is formal.
\end{theorem}
We call an $\mathfrak{L}$ as in the theorem a Solomon-Verbitsky collection. The proof of this theorem is via deformation quantisation and a categorical generalisation of the above mentioned spectral sequence, more details on the method of proof will be given below.
 
\subsection{A (local) holomorphic Fukaya category via deformation quantisation}

As already mentioned, the proof of \cref{main1} is using deformation quantisation techniques. By \cite{10.1155/S1073792804132819} there exists a (unique) canonical deformation quantisation $\C\cbrak$-algebroid $\widehat{\mathscr{W}}_\mr{X}$. We are going to consider a differential graded category of $\widehat{\mathscr{W}}_\mr{X}$-modules.\\
Let $\mathcal{DQ}(\mr{X},\sigma)$ be the full differential graded subcategory of $\Db_{\mr{dg}}(\widehat{\mathscr{W}}_\mr{X})$ consisting of quantised orientation modules. As in the D-brane situation, for a collection $\mathfrak{L}$ of Lagrangian submanifolds, we write $\mathcal{DQ}_\mathfrak{L}(\mr{X},\sigma)$ be the full subcategory spanned by objects supported in $\mathfrak{L}$. Our second result is as follows, cf \cite[Conjecture~0.3.4]{MR4678893}:
\2
\begin{theorem}[\cref{lags}]\label{main2}
Let $(\mr{X},\sigma)$ be a holomorphic symplectic manifold and let $\mathfrak{L}$ be a (countable) collection of orientable compact K\"{a}hler Lagrangian submanifolds with clean pairwise intersections. Then the differential graded category $\mathcal{DQ}_\mathfrak{L}(\mr{X},\sigma)$ is formal.
\end{theorem}
If each $\mr{L}$ in $\mathfrak{L}$ is equipped with a choice of (classical) orientation data $\mr{K}_\mr{L}^{1/2}$, then there is a canonical full subcategory $\mathcal{DQ}^{\mr{s}}_\mathfrak{L}(\mr{X},\sigma)$ of $\mathcal{DQ}_\mathfrak{L}(\mr{X},\sigma)$, containing a single object for each Lagrangian in $\mathfrak{L}$, defined as follows; by \cite{MR2331247} and \cite{saf}, for each $\lambda	\in \C$, one has a unique quantised orientation $\mathscr{D}^\lambda_\mr{L}$ such that there is an isomorphism $$\mu_{\Lambda_\mr{L}}\left(\mathscr{D}^\lambda_\mr{L}\right)|_{\mr{L}\times 1}\cong \mr{K}_\mr{L}^{1/2}\otimes \C_\lambda$$ of twisted $\mr{D}$-modules with $\mr{exp}(2\pi i \lambda)$ monodromy automorphism. By the formality results of \cite{MR4678893}, $\mathcal{DQ}^{\mr{s}}_\mathfrak{L}(\mr{X},\sigma)$ is independent of the choice of quantised orientations.
\2 
\begin{corollary}\label{main2a}
Let $(\mr{X},\sigma)$ be a holomorphic symplectic manifold and let $\mathfrak{L}$ be a (countable) collection of orientable compact K\"{a}hler Lagrangian submanifolds with clean pairwise intersections. Then the differential graded category $\mathcal{DQ}^\mr{s}_\mathfrak{L}(\mr{X},\sigma)$ is formal.
\end{corollary}
In fact, the morphism spaces of $\mathcal{DQ}(\mr{X},\sigma)$ are mixed Hodge complexes. Following the philosophy that "purity implies formality", we expect the next to be true:
\begin{conjecture}
Let $(\mr{X},\sigma)$ be a holomorphic symplectic manifold and let $\mathfrak{L}$ be a (countable) collection of orientable compact Lagrangian submanifolds such that the weight filtration of $\mr{H}\mathcal{DQ}_\mathfrak{L}(\mr{X},\sigma)$ is pure. Then the differential graded category $\mathcal{DQ}_\mathfrak{L}(\mr{X},\sigma)$ is formal.
\end{conjecture}
Suppose that $\mathfrak{L}$ consists of compact Lagrangian submanifolds. Then, we show that the $\C\cbrak$-linear dg category  $\mathcal{DQ}_\mathfrak{L}(\mr{X},\sigma)$ admits a lift to a $\C\brak$-linear dg category $\widetilde{\mathcal{DQ}}_\mathfrak{L}(\mr{X},\sigma)$ which will be used to relate the deformation quantisation model to the $\mr{D}$-brane model via its classical limit $\hbar\to 0$. 
 \begin{proposition}[\cref{lift}]
Let $\mathfrak{L}$ be a collection of compact Lagrangian submanifolds in $(\mr{X},\sigma)$, then there exists a $\C\brak$-linear differential graded $\widetilde{\mathcal{DQ}}_\mathfrak{L}(\mr{X},\sigma)$ such that $$\mr{loc}\left(\widetilde{\mathcal{DQ}}_\mathfrak{L}(\mr{X},\sigma)\right)=\mathcal{DQ}_\mathfrak{L}(\mr{X},\sigma).$$ 
\end{proposition}
When $\mathfrak{L}$ is Solomon-Verbitsky, we obtain \cref{hfree}, allowing us to control the Hochschild (co)homology of $\widetilde{\mathcal{DQ}}_\mathfrak{L}(\mr{X},\sigma)$ and thus apply our general results on formality of dg categories in families to it.
\subsection{The virtual de Rham category}

In the case of a single Lagrangian submanifold, we have seen that de Rham cohomology and, somewhat more precisely, the spectral sequence which starts with de Rham cohomology and converges to the cohomology of $\mr{RHom}$, play a key role relating the coherent complex (counting massless states of open strings) and the deformation quantisation complex. Following this philosophy, we shall define a virtual de Rham differential graded category playing the role of de Rham cohomology in the case of categories.\\
Let $\mr{L}$ and $\mr{M}$ be two Lagrangian submanifolds and suppose given two smooth half-twisted $\mr{D}$-modules $\mathscr{L}$ and $\mathscr{M}$ on $\mr{L}$ and $\mr{M}$, respectively. Following Behrend and Fantechi, we define the virtual de Rham complex $$\mathscr{DR}^{\mr{vir}}(\mathscr{L},\mathscr{M})\coloneq \left(\ext^{\bullet}_{\mathscr{O}_\mr{X}}(\mathscr{L},\mathscr{M}),\mr{d}_\mr{BF}\right).$$
 In \cref{msres}, we introduce the Malgrange-Serre functor $\mr{MS}$ which provides soft resolutions of coherent sheaves on a complex manifold. These are particularly well-suited for our purposes as they are functorial and multiplicative when the underlying coherent sheaves are equipped with bilinear pairings.\\
For a contact manifold $\mr{Y}$ and a Legendrian $\Lambda$, the microlocalisation functor $$\mu_{\Lambda}:\mr{Mod}_{\Lambda,\mr{rh}}\left(\widehat{\mr{E}}_\mr{Y}\right)\to \gamma_*\mr{Mod}_\mr{rh}\left(\mr{D}_{\tilde{\Lambda}}^{\sqrt{v}}\right)$$ allows us in \cref{vdrsec} to extend this construction to regular holonomic $\widehat{\mathscr{W}}_\mr{X}$-modules admitting lifts along $\mu_{\Lambda}$ to $\widehat{\mr{E}}_{\mr{Y}}$-modules for a (local) contactification $\mr{Y}\to \mr{X}$. We note that in general the forgetful functor induced by the inclusion $\Ww\xhookrightarrow{} \rho_*\widehat{\mr{E}}_\mr{Y}$ is only locally essentially surjective. \\
Putting it all together, we obtain a differential graded category $$\mathcal{DR}^{\mr{vir}}(\mr{X},\sigma)$$ whose objects are quantised orientation modules and, for two such modules $\mathscr{D}_\mr{L}$ and $\mathscr{D}_\mr{M}$, the morphism complex is $$\Gamma\left(\mr{X},\mr{MS}\left(\mathscr{DR}^\mr{vir}\left(\mathscr{D}_\mr{L},\mathscr{D}_\mr{M}\right)\right)\right).$$
We then have a formality result for the local version along a collection $\mathfrak{L}$ which is a dg version of the K\"{a}hler formality of \cite{Deligne1975} :
\2
\begin{theorem}[\cref{drvir}]\label{main3}
Let $(\mr{X},\sigma)$ be a holomorphic symplectic manifold and let $\mathfrak{L}$ be a (countable) collection of orientable compact K\"{a}hler  Lagrangian submanifolds with clean pairwise intersections. Then the differential graded category $\mathcal{DR}^\mr{vir}_\mathfrak{L}(\mr{X},\sigma)$ is formal.
\end{theorem}

In fact, \cref{main2} is a corollary of \cref{main3} since, for Solomon-Verbitsky collections, we have:
\2
\begin{proposition}[\cref{dqdr}]
Let $\mathfrak{L}$ be a Solomon-Verbitsky collection. There is a quasi-isomorphism $$\mathcal{DQ}_\mathfrak{L}(\mr{X},\sigma)\cong \mr{Ind}_{\C\cbrak/\C}\left(\mathcal{DR}^\mr{vir}_\mathfrak{L}(\mr{X},\sigma)\right).$$
\end{proposition}
In general, the virtual de Rham complexes are closely related to the perverse sheaf of vanishing cycles and we conjecture that there is a simply-graded spectral sequence whose first page is the de Rham complex which converges to the cohomology of the (shifted) perverse sheaf. 
\subsection{Formality criteria for differential graded categories}
Proving \cref{main1} requires a generalisation of the results of Lunts on formality of $\Aa$ and differential graded algebras over general rings to $\Aa$ and differential graded categories.\\
For a minimal $\Aa$-category $\Cc$ over a $\Q$-algebra $\mr{R}$, we first extend the construction of Lunts' Kaledin class $\mr{k}_\Cc$ to categories and prove:
\2
\begin{theorem}[\cref{kal}]
Let $\mathcal{C}$ be a minimal $\Aa$-category over $\mr{R}$ and assume $\mr{R}$ is a $\Q$-algebra. The following are equivalent:
\begin{enumerate}
\item The category $\mathcal{C}$ is formal.
\item The category $\mathcal{C}$ is $\mr{A}_n$-formal for all $n\ge 1$.
\item For all $n\ge 3$, the truncated Kaledin class $\mr{k}^{\le n}_\mathcal{C}$ vanishes.
\item The Kaledin class $\mr{k}_\mathcal{C}$ vanishes.
\end{enumerate}
\end{theorem}
As an easy corollary, mimicking Lunts' case of algebras, we have:
\2
\begin{corollary}[\cref{a2generic}]
Suppose $(\Cc,m)$ is a finite minimal $\Aa$-category which is flat and proper over an integral domain $\mr{R}$ with generic point $\eta$. Assume the $\mr{R}$-module $\mr{HH}^2_c(\Cc,m_2)$ is torsion-free. Then $\Cc$ is formal if and only if $\Cc_\eta$ is formal.
\end{corollary}

There is a differential graded version of this, just as in the algebra case, see \cref{a2dg}. The difficulty of working with categories, rather than algebras, is that the Kaledin class is never living in the compactly supported cohomology if the category is not finite, hence we lose the nice base-change properties of the compactly supported Hochschild cohomology. We can still salvage the following which is good enough for our applications:
\2
\begin{theorem}[\cref{a2cy}]
Let $\Cc$ be a flat weakly proper Calabi-Yau differential graded category over an intergral domain $\mr{R}$ with generic point $\eta$. Assume the $\mr{R}$-module $\mr{HH}_\bullet(\mr{H}\Cc,m_2)$ is projective, where $m_2$ is the composition in $\Cc$. Then $\Cc$ is formal if and only if $\Cc_\eta=\mr{Ind}_{k(\eta)/\mr{R}}(\Cc)$ is formal.
\end{theorem}
\subsection{\textbf{Context}}

\subsubsection{Holomorphic Floer theory}
There has been a duality philosophy for some time now in the framework of holomorphic symplectic manifolds. The papers \cite{2004IJGMM..01...49K}, \cite{Kapustin2005AbranesAN} are first works on the subject.
The paper \cite{2004IJGMM..01...49K} puts forward a conjecture that the Fukaya category of a hyperkähler variety $(\mr{X},\mr{I,J,K})$ with symplectic form $\omega_\mr{J}$ should be quasi-equivalent to a non-commutative deformation of the derived category on the holomorphic symplectic manifold $(\mr{X},\mr{I},\sigma_\mr{I} = \omega_\mr{J}+i\omega_\mr{K})$. In fact, his calculations show that this should be a formal deformation in the non-commutative Poisson bivector direction.\\ More generally, in the subsequent paper \cite{Kapustin2005AbranesAN}, he upgrades the above to a duality between $\mr{A}$ and $\mr{B}$ branes on $(\mr{X},\omega_\mr{J})$ and $(\mr{X},\mr{I},\sigma_\mr{I} = \omega_\mr{J}+i\omega_\mr{K})$, respectively. The following conjecture is a precise formulation of the above discussion.
\2
\begin{conjecture}\label{swdconj}
 Let $(\mr{X},\mr{I,J,K},g)$ be a hyperkähler manifold. There is a quasi-equivalence $$\mathcal{DF}(\mr{X},\omega_\mr{J}) \simeq \mr{Ind}_{\mr{Nov}/\C\cbrak}\left(\mathbf{D}_\mr{dg,h}(\widehat{\mathscr{W}}_\mr{X})\right),$$ between the Fukaya category $\mathcal{DF}(\mr{X},\omega_\mr{J})$ and the differential graded category of holonomic $\widehat{\mathscr{W}}_\mr{X}$-modules associated to the holomorphic symplectic manifold $(\mr{X},\mr{I},\sigma_\mr{I} = \omega_\mr{J}+i\omega_\mr{K})$.
\end{conjecture} 
More recently, Kontsevich and Soibelman \cite{hft} have introduced a generalised Riemann-Hilbert correspondence for a holomorphic symplectic manifold $(\mr{X},\sigma)$. Their global conjecture is an equivalence, after extension of scalars, between the $\mr{A}$-side, also called Betti, which is a global Fukaya category $\mathcal{F}_\mr{glob}$ defined over the limit of rings analytic functions on punctured unit discs, and the $\mr{B}$-side, also known as de Rham, which is the category of global holonomic deformation quantisation modules $\mathcal{H}ol_\mr{glob}$. Their theory includes Lagrangian subvarieties and the theory of quantum wave functions is supposed to produce deformation quantisation modules supported of these singular Lagrangian subvarieties. Kontsevich and Soibelman have local versions of this correspondence which is associated with a neighbourhood of a fixed Lagrangian and seems better understood.
\subsubsection{Symplectic geometry and Solomon-Verbitsky}In \cite{solomon_verbitsky_2019} Solomon and Verbitsky consider the (local) Fukaya category $\widehat{\mathcal{A}}_\mathfrak{L}$ of a collection of $\mr{I}$-holomorphic graded spin Lagrangians $\mathfrak{L}$ in a hyperkähler variety $(\mr{X},\mr{I,J,K},g)$, equipped with the symplectic form $\omega_\mr{J}=g(\mr{J}\cdot,\cdot)$. Their main result shows that, generically, finite energy holomorphic curves bounding $\mr{I}$-holomorphic Lagrangian submanifolds must be constant.\\
It follows that for two such Lagrangian $\mr{L}$ and $\mr{M}$ which intersect cleanly, the Floer coboundary operator $\mu_1(\mr{L,M})$ of $\mr{CF}(\mr{L,M})$ coincides with the de Rham differential, hence the spectral sequence $$\mr{H}^{*}(\mr{L}/\C)\otimes \mr{Nov} \Rightarrow \mr{HF}^{*}(\mr{L,L})$$ collapses on the second page. In addition, the Floer composition $\mu_2(\mr{L})$ of $\mr{CF}(\mr{L,L})$ is the wedge product of differential forms up to sign and $\mu_k(\mr{L,L})=0$ for $k \ge 3$. This implies, as in \cite{Deligne1975}, that the Floer $\Aa$-algebra $\mr{CF}(\mr{L},\mr{L})$ is formal. The de Rham versions of these results were obtained in \cite{MR4678893}.\\
These results motivate them to state the following conjecture attributed to Ivan Smith:
\2
\begin{conjecture}\label{ivanformality}
For a collection of compact spin $\mr{I}$-holomorphic Lagrangian submanifolds with clean pairwise intersections, the $\Aa$-category $\widehat{\mathcal{A}}_\mathfrak{L}$ is a formal.
\end{conjecture}
Then our \cref{main2} is the de Rham analogue of this conjecture under the generalised Riemann-Hilbert correspondence of Kontsevich and Soibelman \cite{hft}.\\
The category $\widehat{\mathcal{A}}_\mathcal{L}$ is known to be (intrinsically) formal by \cite{MR3486414} for the Slodowy slice to a nilpotent matrix with two equal Jordan blocks and the (finite) Seidel-Smith collection of distinguished Lagrangian submanifolds.\\
In relation to the work of Solomon and Verbitsky, we introduce also a local version of \cref{swdconj} which should be much more accessible.\\
Note that, as explained above, given a Solomon-Verbitsky collection $\mathfrak{L}$, the full differential graded subcategory $\mathcal{DQ}_\mathfrak{L}^{\mr{s}}$ of $\mathcal{DQ}_\mathfrak{L}$, spanned by a choice of quantised orientation for each orientable Lagrangian submanifold, equipped with orientation data, in $\mathfrak{L}$, is independent of these choices as the morphism complexes are invariant upon chaining (quantised) orientations. Since the Solomon-Verbitsky category $\widehat{\mathcal{A}}_\mathfrak{L}$ has a single object associated to each Lagrangian in $\mathfrak{L}$, we should expect in line with \cref{main2a} and \cref{ivanformality}: 
\2
 \begin{conjecture}\label{svdq}
  Let $\mathfrak{L}$ be a Solomon-Verbitsky collection of Lagrangian submanifolds in $(\mr{X},\sigma)$. There is a quasi-isomorphism $$\widehat{\mathcal{A}}_\mathfrak{L}\cong \mr{Ind}_{\mr{Nov}/\C\cbrak}\left(\mathcal{DQ}_\mathfrak{L}^{\mr{s}}\right).$$
 \end{conjecture}

\subsection{Plan}
In \cref{sec1} we begin by reviewing the general theory of $\Aa$-categories and cocategories and their relationship via the bar construction, following \cite{lefvrehasegawa2003sur}. Next, we recall several notions of Hochschild (co)homology for $\Aa$-categories as well as (weak) Calabi-Yau structures on differential graded categories after \cite{MR3911626}. This culminates in the last subsection which contains the theory of Kaledin classes of $\Aa$-categories and their applications to formality.\\
The next section \cref{sec2} is mostly a refresher on deformation quantisation after \cite{MR3012169}, $\mr{d}$-critical loci and perverse sheaves \cite{bbdjs}, \cite{saf}, as these play crucial roles in later parts of the paper. It contains a few new definitions, most importantly the notion of quantised orientations for Lagrangian submanifolds.\\
Then in \cref{sec3}, we introduce the Malgrange-Serre resolutions, recall the constructible complexes of Behrend and Fantechi \cite{MR2641169} and define the virtual de Rham (sheaf) complex for Lagrangian intersections. We conclude by relating it to the perverse sheaf of \cite{bbdjs} in the clean intersection case.\\
The last section \cref{sec4} begins with the definitions the various differential graded categories associated with a holomorphic symplectic manifold and its Lagrangian submanifolds which occupy the first three subsections. In the last subsection, we conclude with our main results on formality of these dg categories.

\subsection{Acknowledgements}

I wish to thank Jake Solomon and Richard Thomas for asking several questions, to which this paper owes its existence, and related discussions. 

\section{\texorpdfstring{ $\Aa$-}- categories}\label{sec1}

\subsection{Graded categories}

Fix a ring $\mr{R}$. We denote by $\mr{Mod}_{\mr{gr}}(\mr R)$ the category of $\Z$-graded $\mr{R}$-modules whose objects are modules $\mr{M}$ over $\mr{R}$ equipped with a decomposition $$\mr{M}=\bigoplus_{n \in \Z} \mr{M}_n.$$ An $\mr{R}$-linear map $f:\mr{M} \to \mr{N}$ is said to be of degree $m$ if $$f(\mr{M}_n) \subseteq \mr{N}_{n+m}.$$ Morphisms in the category of graded $\mr{R}$-modules are $\mr{R}$-linear maps of degree $0$, also called (co)chain maps.\\
A differential on a graded module $\mr{M}$ is a degree $1$ linear map $\mr{d}_\mr{M}:\mr{M}\to \mr{M}$ such that $\mr{d}_\mr{M}^2=0$. We call the pair $(\mr{M},\mr{d}_\mr{M})$ a differential graded module. A morphism of differential graded modules is a degree $0$ linear morphism commuting with the differential. We denote by $\mr{Mod}_\mr{dg}(\mr{R})$ the category of differential graded $\mr{R}$-modules.
\2
\begin{definition}
A graded category over $\mr{R}$ is a category $\mathcal{C}$ enriched over $\mr{Mod}_{\mr{gr}}(\mr{R})$.
\end{definition}
More explicitly, this means for any two objects $\mr{X,Y} \in \mathcal{C}$, we have a decomposition of the space of morphisms $$\mathcal{C}(\mr{X},\mr{Y}) = \bigoplus_{n \in Z}\mathcal{C}^{n}(\mr{X},\mr{Y}).$$
Moreover, the composition of morphisms is of degree zero.\\
A graded functor $\mr{F}:\mathcal{C}_1\to \mathcal{C}_2$ is a functor which induces a degree $0$ morphism $$\mr{F}_{\mr{X},\mr{Y}}:\mathcal{C}_1(\mr{X},\mr{Y})\to \mathcal{C}_2(\mr{F}(\mr{X}),\mr{F}(\mr{Y}))$$ for any $\mr{X},\mr{Y} \in \mr{Ob}(\mathcal{C}_1)$.\\
For a graded category $\mathcal{C}$, the opposite graded category $\mathcal{C}^{\mr{op}}$ has the same objects as $\mathcal{C}$, the morphism spaces are $$\mathcal{C}^{\mr{op}}(\mr{X,Y})=\mathcal{C}(\mr{Y,X})$$ with composition $$\mathcal{C}^{\mr{op},n}(\mr{Y,Z})\otimes \mathcal{C}^{\mr{op},m}(\mr{X,Y})\to \mathcal{C}^{\mr{op},n+m}(\mr{X,Z})$$ such that $$g\otimes f \mapsto (-1)^{nm}f\circ g.$$
\subsection{Differential graded categories}
\begin{definition}
A differential graded category over $\mr{R}$ is a category $\mathcal{C}$  enriched over $\mr{Mod}_\mr{dg}(\mr{R})$.
\end{definition}

Hence the morphism spaces of a dg category $\mathcal{C}$ are endowed with a structure of differential graded modules over $\mr{R}$, that is, $$\mathcal{C}(\mr{X,Y})=\bigoplus_{n\in \Z}\mathcal{C}^{n}(\mr{X,Y})$$ with a differential $\mr{d}_\mr{X,Y}$ of degree $1$ and the compositions $$\mathcal{C}(\mr{Y,Z})\otimes \mathcal{C}(\mr{X,Y})\to \mathcal{C}(\mr{X,Z})$$ are morphisms of dg $\mr{R}$-modules.\\
A differential graded (dg) functor is a functor $$\mr{F}:\mathcal{C}_1\to \mathcal{C}_2$$ such that $$\mr{F}_{\mr{X},\mr{Y}}:\mathcal{C}_1(\mr{X},\mr{Y})\to \mathcal{C}_2(\mr{F}(\mr{X}),\mr{F}(\mr{Y}))$$ is a morphism of dg $\mr{R}$-modules, i.e. it has degree $0$ and commutes with the differentials, for any $\mr{X},\mr{Y} \in \mr{Ob}(\mathcal{C}_1)$. A dg functor $$\mr{F}:\mathcal{C}_1\to \mathcal{C}_2$$ is quasi-isomorphism if it induces a bijection on objects $$\mr{Ob}(\mathcal{C}_1)\cong \mr{Ob}(\mathcal{C}_2)$$ and, for any two $\mr{X,Y}\in \mr{Ob}(\mathcal{C}_1)$, a quasi-isomorphism $$\mr{F}_{\mr{X},\mr{Y}}:\mathcal{C}_1(\mr{X},\mr{Y})\to \mathcal{C}_2(\mr{F}(\mr{X}),\mr{F}(\mr{Y})).$$
The opposite differential graded category of $\mathcal{C}$ is defined as the opposite graded category $\mathcal{C}^\mr{op}$ with the same differential as $\mathcal{C}$.

\subsection{\texorpdfstring{ $\Aa$-}- categories}

\begin{definition}
Let $n \in \mathbf{N} \cup \{\infty \}$. An $\mr A_n$-category $\mathcal{C}$ is a consists of class of objects $\mr{Ob}(\mathcal{C})$ and graded $\mr{R}$-modules of morphisms $\mathcal{C}(\mr{X,Y})$ for all $\mr{X,Y}\in \mr{Ob}(\mathcal{C})$ such that for all $n\ge i\ge 1$ and all $\mr{X}_0,\cdots,\mr{X}_i \in \mr{Ob}(\mathcal{C})$ there are $\mr{R}$-linear morphisms $$m_i:\mathcal{C}(\mr{X}_{i-1},\mr{X}_{i})\otimes\cdots\otimes \mathcal{C}(\mr{X}_0,\mr{X}_1)\to \mathcal{C}(\mr{X}_0,\mr{X}_i)[2-i]$$ satisfying \begin{equation}\label{defeq1cat} \tag{$*_{m}$}
 \sum_{j+k+l=m} (-1)^{jk+l}m_{j+1+l}(\mr{id}^{\otimes j}\otimes m_{k}\otimes \mr{id}^{\otimes l})=0
 \end{equation}
 for all $m\le n$.
\end{definition}
\2
\begin{example}
\begin{enumerate}
\item An $\mr{A}_n$-category $\mathcal{C}$ with one object $\{\bullet\}$ is the data of an $\mr{A}_n$-algebra structure on the endomorphism space $\mathcal{C}(\bullet,\bullet)$.
\item A graded category is an $\Aa$-category with $m_i=0$ for all $i\not=2$.
\item A (possibly non-unital) dg category is an $\Aa$-category with $m_i=0$ for all $i\ge 3$.
\end{enumerate}
\end{example}
\2
\begin{remark}
Note that an $\mr{A}_n$-category needn't be a category for its composition may not be associative and there may be no identities.
\end{remark}
\2
\begin{example}
 Let $\Cc$ be an $\Aa$-category. Then the first relation is $$(*_1) \quad m_1m_1 =0,$$ i.e. $m_1$ is a differential. The second relation is $$(*_2) \quad m_1m_2 = m_2(m_1\otimes \mr{id} + \mr{id} \otimes m_1),$$ meaning that $m_1$ is a derivation for the composition $m_2$. The third equation shows $m_2$ is associative up to the homotopy $m_3$: $$(*_3) \quad m_2(m_2\otimes \mr{id} - \mr{id}\otimes m_2) = m_1m_3+m_3(m_1\otimes \mr{id}^{\otimes 2} + \mr{id}\otimes m_1 \otimes \mr{id} + \mr{id}^{\otimes 2}\otimes m_1).$$ In particular, any minimal $\mr{A}_n$-category is associative. The cohomology of an $\mr{A}_n$-category is a graded associative category.
\end{example}
\2 
\begin{definition}
Let $\mathcal{C}$ be an $\mr{A}_n$ category. Its cohomology category $\mr{H}\mathcal{C}$ is the graded category with the same objects and morphisms $$\mr{H}\mathcal{C}(\mr{X,Y}) = \mr{H}^{\bullet}(\mathcal{C}(\mr{X,Y}))$$ for any two objects $\mr{X,Y}$.\\
The homotopy category of $\mathcal{C}$, $\mr{H}^0\mathcal{C}$, is the degree $0$ part of $\mr{H}\mathcal{C}$.
\end{definition}
\2
\begin{remark}
Note that these auxiliary categories are functorially associated to $\mathcal{C}$.
\end{remark}
\2
\begin{definition}
Let $\mathscr{C}_1$ and $\mathcal{C}_2$ be two $\mr{A}_n$-categories. An $\mr{A}_n$-functor $\mr{F}:\mathcal{C}_1\to \mathcal{C}_2$ is a tuple $(\mr{F},\mr{F}_1,\cdots,\mr{F}_n)$, where $\mr{F}:\mr{Ob}(\mathcal{C}_1)\to \mr{Ob}(\mathcal{C}_2)$ and for all $i\le n$ and all $\mr{X}_0,\cdots,\mr{X}_i \in \mr{Ob}(\mathcal{C}_1)$ $$\mr{F}_i:\mathcal{C}_1(\mr{X}_{i-1},\mr{X}_{i})\otimes\cdots\otimes \mathcal{C}_1(\mr{X}_0,\mr{X}_1)\to \mathcal{C}_2(\mr{F}\mr{X}_0,\mr{F}\mr{X}_i)$$ are morphisms of degree $1-i$ such that for all $m\le n$ we have \begin{equation}\label{defeq2cat}\tag{$**_m$}
 \sum_{j+k+l=m} (-1)^{jk+l}\mr{F}_{j+1+l}(\mr{id}^{\otimes j}\otimes m_{k}\otimes \mr{id}^{\otimes l}) =  \sum_{i_{1}+\cdots +i_{r}=m}(-1)^{s(i_1,\cdots,i_r)}m_{r}(\mr{F}_{i_1}\otimes \cdots \otimes \mr{F}_{i_r}),                                                                                                                                                                                                                                                                                                                                                      \end{equation}
 where we set $$s(i_1,\cdots,i_r) = \sum_{2\le u\le r}\big((1-i_u)\sum_{1\le v \le u} i_v\big).$$
The composition of $\mr{F} : \mathcal{C}_1 \to \mathcal{C}_2$ and $\mr{G} : \mathcal{C}_2 \to \mathcal{C}_3$ is defined by $$(\mr{G} \circ \mr{F})_n = \sum_r \sum_{i_1 +\cdots+i_r=n} (-1)^{s(i_1,\cdots,i_r)}\mr{G}_{r}(\mr{F}_{i_{1}}\otimes \cdots \otimes \mr{F}_{i_{r}}).$$
\end{definition}
\2 
\begin{remark}
\begin{enumerate}
\item We abuse notation in the above definition, writing $m_i$ for the structure morphisms of both $\mathcal{C}_1$ and $\mathcal{C}_2$.
\item We shall say that $\mr{F}$ is strict if the morphisms $\mr{F}_i=0$ for all $i\ge 2$. Hence the morphism $\mr{F}_1$ commutes with all multiplications.
\end{enumerate}
\end{remark}
\2
\begin{example}
 Let $F : \Cc_1 \to \Cc_2$ be an $\Aa$-functor. Then $$(**_1) \quad F_1m_1=m_1F_1,$$ that is, $F_1$ is a morphism of complexes, i.e. $\mr{A}_1$-morphisms are just morphisms of complexes together with a map of objects. The second relation is $$(**_2) \quad F_1m_2 = m_2(F_1\otimes F_1) + m_1F_2+F_2(m_1\otimes \mr{id} +\mr{id}\otimes m_1),$$ measuring the compatibility of $F_1$ with the compositions of $\Cc_1$ and $\Cc_2$.
\end{example}
\2
\begin{definition}
An $\mr{A}_n$ functor $\mr{F}:\mathcal{C}_1\to \mathcal{C}_2$ is an quasi-equivalence if $$\mr{F}_1:\mathcal{C}_1(\mr{X},\mr{Y})\to \mathcal{C}_2(\mr{F}\mr{X},\mr{F}\mr{Y})$$ is a quasi-isomorphism for all $\mr{X,Y}\in \mr{Ob}(\mathcal{C}_1)$ and $\mr{H}^0\mr{F}:\mr{H}^0\mathcal{C}_1\to \mr{H}^0\mathcal{C}_2$ is an equivalence of categories. \\
We say $\mathcal{C}_1$ and $\mathcal{C}_2$ are quasi-equivalent if there exist $\mr{A}_n$-categories $\mathcal{A}_1,\cdots \mathcal{A}_m$ and quasi-equivalences $\Cc_1 \leftarrow \mathcal{A}_1 \rightarrow \cdots \leftarrow \mathcal{A}_m \rightarrow \Cc_2$.
\end{definition}
\2
\begin{definition}Let $n$ be a positive integer or $\infty$.
\begin{enumerate}
 \item An $\mr{A}_n$-category $\Cc$ is called minimal if $m_1=0$.
 \item A minimal model for an $\mr{A}_n$-category $\Cc$ is a minimal $\mr{A}_n$-category $\Cc_0$ together with a quasi-equivalence $\Cc_0 \to \Cc$.
 \end{enumerate}
\end{definition}
Recall the following classical theorem of Kadeishvilli:
\2
\begin{theorem}[Kadeishvili \cite{MR580645}]
 Let $\mr{A}$ be an $\Aa$-algebra over $\mr{R}$ such that $\mr{HA}$ is a projective $\mr{R}$-module. For any choice of a quasi-isomorphism $f_1 : \mr{HA} \to \mr{A}$ of complexes of $\mr{R}$-modules, there exists a minimal $\Aa$-structure on $\mr{HA}$, with $m^{\mr{HA}}_2$ being induced by $m_2$, and an $\Aa$-quasi-isomorphism $f : \mr{HA} \to \mr{A}$ lifting $f_1$.
\end{theorem}
The homotopy transfer theorem extends this theorem on $\Aa$-algebras to $\Aa$-categories.
\2
\begin{theorem}[Markl \cite{MR2287133}]
Suppose that $\mathcal{C}$ is an $\Aa$-category and for all objects $\mr{X,Y}$ in $\mathcal{C}$, we have a diagram
\[
\begin{tikzcd}
\mathcal{C}_0(\mr{X,Y}) \arrow[r, bend left, "\mr{F}_1"] & \mathcal{C}(\mr{X,Y}) \arrow[l, bend left, "\mr{G}_1"] \arrow[loop right, "\mr{h}_1"]
\end{tikzcd}
\]
where $\mathcal{C}_0(\mr{X,Y})$ are complexes, $\mr{F_1,G_1}$ are morphisms of complexes, and $\mr{h}_1$ is a degree $-1$ map such that $\mr{dh_1+h_1d}=\mr{F_1G_1}-\mr{id}$.\\
Then, there is an $\Aa$-category $\mathcal{C}_0$ with objects $\mr{Ob}(\mathcal{C})$, whose $\mr{A}_1$-structure is given by the complexes $\mathcal{C}_0(\mr{X,Y})$. Furthermore, there are $\Aa$-functors $\mr{F}:\mathcal{C}_0\to \mathcal{C}$ and $\mr{G}:\mathcal{C}\to \mathcal{C}_0$ which are identity on objects and lift $\mr{F}_1$ and $\mr{G}_1$, respectively, such that there is a homotopy $\mr{h}: \mr{F}\circ\mr{G}\simeq \mr{id}_\mathcal{C}$ lifting $\mr{h}_1$.
\end{theorem}
\begin{proof}
We are going to sketch an inductive definition of the $\Aa$-structure on $\Cc_0$ and the functor $\mr{F}$, referring to \cite{MR2287133} for details and formulae for $\mr{G}$ and $\mr{h}$. We start with $\mr{F}$ and then define the structure on $\Cc_0$ so that $\mr{F}$ is an $\Aa$ functor. Let $$\mr{F}_i=\sum_r\sum_{i_1,\cdots,i_r} (-1)^{s(i_1,\cdots,i_r)}\mr{h}_1\circ m_{\Cc,r}\circ \left(\mr{F}_{i_1}\otimes\cdots\otimes \mr{F}_{i_r}\right).$$ For the $\Aa$-structure maps, we have $$m_{\Cc_0,i}=\sum_r\sum_{i_1,\cdots,i_r} (-1)^{s(i_1,\cdots,i_r)}\mr{G}_1\circ m_{\Cc,r}\circ \left(\mr{F}_{i_1}\otimes\cdots\otimes \mr{F}_{i_r}\right),$$ where we are summing over all $r>1$ and $i_1+\cdots+i_r=i$. It is not difficult to check that these satisfy the $\Aa$-relations. Let us just mention that the inductive formulae for $\mr{G}$ and $\mr{h}$ are somewhat more complicated and planar binary trees provide a cleaner approach to the proof.
\end{proof}
We record the following immediate corollaries for future use:
\2
\begin{corollary}
Let $\mathcal{C}$ be an $\Aa$-category such that $\mr{H}\mathcal{C}(\mr{X,Y})$ are projective $\mr{R}$-modules. Then for any quasi-isomorphism $\mr{F}_1:\mr{H}\mathcal{C}(\mr{X,Y}) \to \mathcal{C}(\mr{X,Y})$, there is a minimal $\Aa$-structure on $\mr{H}\mathcal{C}$ and a quasi-isomorphism $\mr{F}:\mr{H}\mathcal{C}\to \mathcal{C}$ lifting $\mr{F}_1$.
\end{corollary}
\2
\begin{corollary}
Let $\mathcal{C}$ be a dg category such that $\mr{H}\mathcal{C}(\mr{X,Y})$ are projective $\mr{R}$-modules. Then $\mathcal{C}$ has a minimal $\Aa$-model.
\end{corollary}

\subsection{Cocategories}

\begin{definition}
A cocategory $\mathcal{C}$ over $\mr{R}$ is given by a class of objects $\mr{Ob}(\mathcal{C})$ together with $\mr{R}$-linear comultiplication $\Delta_\mathcal{C} : \mathcal{C}\to \mathcal{C}\otimes \mathcal{C}$ which is coassociative, that is, $\left(\Delta_\mathcal{C}\otimes 1_\mathcal{C}\right)\circ \Delta_\mathcal{C} = \left(1_\mathcal{C}\otimes \Delta_\mathcal{C}\right)\circ \Delta_\mathcal{C}$.
\end{definition}
\2
\begin{definition}
Let $\mathcal{C}_1$ and $\mathcal{C}_2$ be $\mr{R}$-linear cocategories. A cofunctor $\mr{F}:\mathcal{C}_1\to\mathcal{C}_2$ is a pair consisting of a map of objects $\mr{F}:\mr{Ob}(\mathcal{C}_1)\to\mr{Ob}(\mathcal{C}_2)$ and an $\mr{R}$-linear degree $0$ morphism  $$\mr{F}:\mathcal{C}_1(\mr{X},\mr{Y})\to \mathcal{C}_2(\mr{FX},\mr{FY})$$ such that $$\left(\mr{F}\otimes \mr{F}\right)\circ \Delta_{\mathcal{C}_1}=\Delta_{\mathcal{C}_2}\circ \mr{F}.$$
\end{definition}

The main example for us is the reduced tensor cocategory associated to a graded category (or more generally graph as we don't need the compositions) over $\mr{R}$. Namely, let $\mathcal{C}$ be an $\mr{R}$-graph. Let $\overline{\mr{T}}\mathcal{C}$ be the cocategory with the same objects as $\mathcal{C}$ and morphisms $$\overline{\mr{T}}\mathcal{C}(\mr{X,Y})=\bigoplus_{n\ge 1}\bigoplus_{\mr{X}=\mr{X}_0,\cdots,\mr{X}_n=\mr{Y}}\mathcal{C}(\mr{X}_{n-1},\mr{X}_n)\otimes \cdots\otimes \mathcal{C}(\mr{X}_0,\mr{X}_1).$$ The comultiplication is given by $$\Delta:\overline{\mr{T}}\mathcal{C}(\mr{X,Y})\rightarrow\bigoplus_{\mr{Z}\in \mr{Ob}(\mathcal{C})}\overline{\mr T}\mathcal{C}(\mr{Z},\mr{Y})\otimes \overline{\mr T}\mathcal{C}(\mr{Z},\mr{X})$$ which is determined on pure tensors by $$f_1\otimes f_2\otimes \cdots\otimes f_n\mapsto \sum_{i=1}^{n}\left(f_n\otimes \cdots \otimes f_{i+1}\right)\otimes \left(f_i\otimes \cdots \otimes f_1\right).$$
In addition to the internal grading arising from $\mathcal{C}$, $\overline{\mr T}\mathcal{C}$ has external grading by weight coming from the fact that it is cofree cocategory. The weights are given by tensor length. Hence we obtain an increasing filtration by sub cocategories $$\mr{W}_1\overline{\mr T}\mathcal{C} \subseteq \mr{W}_2\overline{\mr T}\mathcal{C}\subseteq \cdots \subset \mr{W}_n\overline{\mr T}\mathcal{C}\subseteq \cdots \subseteq \overline{\mr T}\mathcal{C},$$ where $\mr{W}_n\overline{\mr T}\mathcal{C}$ has the same objects as $\overline{\mr T}\mathcal{C}$ and morphisms of weight $\le n$.
\2
\begin{definition}
Let $\mr{F,G}:\mathcal{C}_1\to \mathcal{C}_2$ be cofunctors. An $\left(\mr{F,G}\right)$-coderivation of degree $p$ is given by a map of objects $\mr{D}:\mr{Ob}(\mathcal{C}_1)\to \mr{Ob}(\mathcal{C}_2)$ and a collection of $\mr{R}$-linear morphisms of degree $p$ $$\mr{D}:\mathcal{C}_1(\mr{X,Y})\to \mathcal{C}_2(\mr{D}\mr{X},\mr{D}\mr{Y})$$ 
such that $$\Delta_{\mathcal{C}_2}\circ \mr{D}=\left(\mr{G}\otimes \mr{D}+\mr{D}\otimes \mr{F}\right)\circ \Delta_{\mathcal{C}_1}.$$
\end{definition}
We denote the graded $\mr{R}$-module of $(\mr{F,G})$-coderivations by $\mr{coDer}(\mr{F,G})$. Any increasing filtration $\mr{W}_\bullet\mathcal{C}_1$ on $\mathcal{C}_1$ induces a decreasing filtration $\mr{W}_\bullet\mr{coDer}(\mr{F,G})$ with $\mr{W}_n\mr{coDer}(\mr{F,G})$ consisting of all coderivations vanishing on $\mr{W}_{n-1}\mathcal{C}_1$.

\subsection{The bar construction.}

Let $\mathcal{C}$ be a graded $\mr{R}$-linear category endowed with morphisms $$m_{i}: \mathcal{C}^{\otimes i} \to \mathcal{C}.$$ For $i \ge 1$ we have a bijection
\begin{align*}
\mr{Hom}(\mathcal{C}^{\otimes i},\mathcal{C}) &\to \mr{Hom}((\mathcal{C}[1])^{\otimes i},\mathcal{C}[1])\\
 m_i &\mapsto \mr{d}_i=(-1)^{i-1+\mr{deg}m_i}s\circ m_i\circ (s^{-1})^{\otimes i},
\end{align*}
 where $s:\mathcal{C} \to \mathcal{C}[1]$ is the canonical degree $-1$ morphism. Remark that in our case $m_i$ are of degree $2-i$, so the corresponding $\mr{d}_i$ have degree $1$. The morphisms $\mr{d}_i$ define a unique morphism $$\overline{\mr{T}}(\mathcal{C}[1]) \to \mathcal{C}[1],$$ which by the universal property of the reduced tensor cocategory (in the category of cocomplete (also known as conilpotent) cocategories) corresponds to a unique degree $1$ coderivation $$\mr{d} : \overline{\mr{T}}(\mathcal{C}[1]) \to \overline{\mr{T}}(\mathcal{C}[1]).$$

\begin{lemma}[Lefèvre-Hasegawa \cite{lefvrehasegawa2003sur}]
 The morphisms $m_i$ define an $\Aa$-category structure on $\mathcal{C}$ iff $\mr{d}$ is a codifferential, i.e. $\mr{d}^2=0$.
\end{lemma}
\2
\begin{definition}
 The bar construction of an $\Aa$-category $\mathcal{C}$ is the differential graded cocategory $\mathcal{B}(\mathcal{C})\coloneqq(\overline{\mr{T}}(\mathcal{C}[1]),\mr d)$.
\end{definition}
\2
Let $\mathcal{C}, \mathcal{D}$ be graded categories. For $i\ge 1$ we have a bijection
\begin{align*}
\mr{Hom}(\mathcal{C}^{\otimes i},\mathcal{D}) &\to \mr{Hom}((\mathcal{C}[1])^{\otimes i},\mathcal{D}[1])\\
 f_i &\mapsto \mr{F}_i=(-1)^{i-1+\mr{deg}f_i}s_{\mathcal{D}}\circ f_i\circ (s_{\mathcal{C}}^{-1})^{\otimes i}.
\end{align*}
If $f_i$ are of degree $1-i$, the maps $\mr{F}_i$ define a degree $0$ cofunctors of graded cocategories $$\mr{F}:\mathcal{B}(\mathcal{C}) \to \mathcal{B}(\mathcal{D}).$$
\begin{lemma}[Lefèvre-Hasegawa \cite{lefvrehasegawa2003sur}]
 Let $\mathcal{C}, \mathcal{D}$ be $\Aa$-categories and let $f_i \in \mr{Hom}(\mathcal{C}^{\otimes i},\mr{D})$ be of degree $1-i$. The morphisms $f_i$ define an $\Aa$ functor if and only if $F$ is compatible with the codifferentials, i.e. we have a bijection
 $$\mr{Hom}_{\mr{A}_{\infty}}(\mathcal{C},\mathcal{D}) \xrightarrow{\sim} \mr{Hom}_{\mr{dgcoCat}}(\mathcal{B}(\mathcal{C}),\mathcal{B}(\mathcal{D})).$$
\end{lemma}
\subsection{Deformations of linear categories}

\begin{definition}
\begin{enumerate}
\item An $\mr{R}$-linear category $\Cc$ is flat if for all $\mr{X,Y}\in \mr{Ob}(\Cc)$, the $\mr{R}$-modules $\Cc(\mr{X,Y})$ are projective.
\item An $\Aa$-category $\Cc$ over $\mr{R}$ is flat if its cohomology $\mr{H}\Cc$ is flat.
\end{enumerate}
\end{definition}
\2
\begin{remark}
We note that $\Cc$ is flat iff $\Cc^{\mr op}$ is flat.
\end{remark}
Let $\varphi : \mr{R}\to \mr{S}$ be a map of rings. We have an induced functor $$\mr{Res}_\varphi : \mr{Mod}_\mr{S}\to \mr{Mod}_\mr{R}.$$ For an $\mr{S}$-linear category $\Cc$, we let $\mr{Res}_\varphi(\Cc)$ be the $\mr{R}$-linear category with the same objects and morphisms $\mr{Res}_\varphi(\Cc)(\mr{X,Y})=\mr{Res}_\varphi(\Cc(\mr{X,Y}))$. In the opposite direction, we have $$\mr{Ind}_\varphi\coloneqq-\otimes_{\mr{R}}\mr{S} :\mr{Mod}_\mr{R}\to \mr{Mod}_\mr{S}$$ and for an $\mr{R}$-linear category $\Cc$, we let $\mr{Ind}_\varphi(\Cc)$ be the $\mr{S}$-linear category with the same objects as $\Cc$ and $$\mr{Ind}_\varphi(\Cc)(\mr{X,Y})=\Cc(\mr{X,Y})\otimes_\mr{R}\mr{S}.$$
It is an easy check that $\mr{Ind}_\varphi$ is left adjoint to $\mr{Res}_\varphi$. Following \cite{MR2238922}, we shall define deformations of linear categories as lifts along these functors.
\2
\begin{definition}
Let $\varphi:\mr{R}\to\mr{S}$ be a map of rings. Let $\Cc$ be an $\mr S$-linear category. An $\mr{R}$-linear deformation of $\Cc$ is a flat $\mr R$-linear category $\mathcal{D}$ equipped with an $\mr{R}$-linear functor $\mathcal{D}\to \mr{Res}_\varphi(\Cc)$ which induces an $\mr{S}$-linear equivalence $\mr{Ind}_\varphi(\mathcal{D})\to \Cc$. 
\end{definition}
\2
\begin{remark}
In sufficiently nice circumstances, an $\mr{R}$-linear deformation induces an abelian deformation of the corresponding categories of modules. We refer to \cite{MR2238922} for the definition of deformations of abelian categories and details.
\end{remark}
\subsection{Hochschild cohomology}

Suppose $\mathcal{C}_1$ and $\mathcal{C}_2$ are dg cocategories. That is, they are equipped with degree $1$ coderivations $\mr{d}_{\mathcal{C}_i}\in \mr{coDer}(1_{\mathcal{C}_i})$. Then, for any two dg cofunctors $\mr{F,G}:\mathcal{C}_1\to\mathcal{C}_2$, the space $\mr{coDer}(\mr{F,G})$ is a dg $\mr{R}$-module with differential $$\mr{D}\mapsto \mr{d}_{\mathcal{C}_2}\circ \mr{D}-(-1)^{\mr{deg}\mr{D}}\mr{D}\circ \mr{d}_{\mathcal{C}_1}.$$

\begin{definition}
Let $\mathcal{C}$ be an $\Aa$-category. The Hochschild cohomology complex of $\mathcal{C}$ is $$\mr{CC}^\bullet(\mathcal{C})\coloneqq \mr{coDer}\left(1_{\overline{\mr{T}}\left(\mathcal{C}[1]\right)}\right),$$ with differential $\mr{d}_\mr{Hoch}=[\mr{d}_{\overline{\mr{T}}\left(\mathcal{C}[1]\right)},-]$. Its cohomology is called the Hochschild cohomology of $\Cc$ and is denoted by $\mr{HH}^\bullet(\mathcal{C})\coloneqq \mr{H}^{\bullet-1}(\mr{CC}(\Cc))$.
\end{definition} 
\2
\begin{remark}
We see immediately from the universal property of the reduced tensor cocategory that $$\mr{CC}^{n}(\mathcal{C})\simeq \mr{Hom}^{n}\left(\mathcal{B}\left(\mathcal{C}\right),\mathcal{C}[1]\right)$$
\end{remark}
The weight filtration on $\mr{CC}(\mathcal{C})$ given by tensor length corresponds to the weight filtration on the space of coderivation under the above identification. Since the differential of $\mr{CC}(\mathcal{C})$ lies in $\mr{W}_1\mr{CC}(\mathcal{C})$, we see that the filtration descends to cohomology that we still denote by $\mr{W}_\bullet\mr{HH}^{\bullet}(\mathcal{C})$ and call the (cohomological) weight filtration.\\ We shall later relate the second piece $\mr{W}_2\mr{HH}^\bullet(\mathcal{C})$ to $\Aa$-isotopies which will be useful for applications.
\2
\begin{definition}
Let $(\mathcal{C},m)$ and $(\Cc,m')$ be two $\Aa$-category structures on $\Cc$. An $\Aa$-isotopy between $m$ and $m'$ is an $\Aa$-functor $\mr{F}:(\Cc,m)\to (\Cc,m')$ such that $\mr{F}=\mr{id}_{\mr{Ob}(\Cc)}$ and for all $\mr{X,Y}\in \mr{Ob}(\Cc)$ we have  $\mr{F}_1=1_{\Cc(\mr{X,Y})}:\Cc(\mr{X,Y})\to \Cc(\mr{X,Y})$.
\end{definition}
\2
\begin{remark}
\begin{enumerate}
\item The set of $\Aa$-isotopies of $\Cc$ forms an $\Aa$-subcategory of $\mr{Fun}(\Cc,\Cc)$.
\item Note that an $\Aa$-isotopy is equivalent to an element $\mr{F}_{+}\in \mr{W}_2\mr{coDer}^{0}\left(1_{\mathcal{B}(\Cc)}\right)$ given by the "coefficients" $\mr{F}_i$, $i\ge 2$, of $\mr{F}$.
\end{enumerate} 
\end{remark}
\2
We observe that $\mr{CC}(\Cc)$ is a complete pre-Lie algebra with complete proper pre-Lie subalgebra given by $\mr{CC}(\Cc)_+=\mr{W}_2\mr{CC}(\Cc)$. This subalgebra is pro-nilpotent In this set up we have an exponential map:
\2
\begin{definition}
For any $\mr{c}\in \mr{W}_2\mr{CC}(\Cc)$, the exponential of $\mr{c}$ is defined as $$\mr{exp}(\mr{c})=\sum_{i \ge 0}\frac{\mr{c}^i}{i!}.$$ 
\end{definition}
\2
As $\mr{c}\in \mr{W}_2\mr{CC}(\Cc)$, we get $\mr{c}^i\in \mr{W}_{i+1}\mr{CC}(\Cc)$, hence the series converges with respect to the topology induced by the filtration $\mr{W}$.
\2
\begin{proposition}
For any $\mr{c}$ in $\mr{W}_2\mr{CC}^0(\Cc)$, the exponential $\mr{exp}(\mr{c})$ defines an $\Aa$-isotopy $\Cc\to \Cc$ iff $\mr{d}_\mr{Hoch}(\mr{c})=0$, i.e. $[\mr{c}]$ is a Hochschild class in $\mr{W}_2\mr{HH}^1(\Cc)$.
\end{proposition}
\2
\begin{definition}
Let $\mathcal{C}$ be an $\Aa$-category. The Hochschild homology complex of $\mathcal{C}$ is $$\mr{CC}_\bullet(\mathcal{C})=\mathcal{C}\otimes \mathcal{B}(\mathcal{C})$$ with differential \begin{align*}\partial(f_0,f_1,\cdots,f_n)&=\sum_{j+k\le n}(-1)^{\epsilon_k}\mr{d}_j(f_{n-k+1}\otimes \cdots f_n \otimes f_0\otimes  f_1\otimes \cdots \otimes f_{j-k-1})\otimes f_{j-k}\otimes \cdots\otimes  f_{n-k}\\&+ \sum_{j+k\le n} (-1)^{\lambda_k}f_0\otimes f_1\otimes \cdots\otimes f_k\otimes \mr{d}_{j}(f_{k+1}\otimes \cdots \otimes f_{k+j})\otimes f_{k+j+1}\otimes \cdots\otimes f_n,\end{align*} where we let $$\epsilon_k=\sum_{l\le k-1}|f_{n-l}|(s_n-|f_{n-l}|) \text{ and } \lambda_k = \sum_{l\le k}|f_l| \text{ for } s_n=\sum_{l\le n}|f_l|.$$
The Hochschild homology of $\mathcal{C}$ is denoted by $\mr{HH}_\bullet(\mathcal{C})$.
\end{definition}

\subsection{Bi-graded and compactly supported Hochschild cohomology}
We shall need variants of Hochschild cohomology. These will be mostly useful in the case of a (finite) graded category $\Cc$. We extend the following definition of Lunts from algebras to categories:
\begin{definition}
A minimal flat $\Aa$-category $\Cc$ is called finitely defined if only finitely many of its higher multiplications $m_i$ are non-zero.
\end{definition}
\2
\begin{definition}
Let $\Cc$ be a finitely defined $\Aa$-category. Its compactly supported Hochschild cohomology complex is $$\mr{CC}_c^{d}(\Cc)=\bigoplus_{p\ge 0}\prod_{\mr{X}_0,\cdots,\mr{X}_p} \mr{Hom}_\mr{R}^d\left(\Cc(\mr{X}_{p-1},\mr{X}_p)[1]\otimes \cdots\otimes\Cc(\mr{X}_0,\mr{X}_1)[1]),\Cc(\mr{X}_0,\mr{X}_p)[1]\right).$$
\end{definition}
\2
\begin{remark}
This is a subcomplex of the Hochschild complex. The reason usually one considers this subcomplex is that direct sums, unlike products, base-change. Hence, the version for categories could play that role only if $\Cc$ is finite.
\end{remark}
Let now $\Cc$ be a graded category, define $$\mr{CC}^{p,q}(\Cc) = \prod_{\mr{X}_0,\cdots,\mr{X}_p} \mr{Hom}_\mr{R}^{p+q}\left(\Cc(\mr{X}_{p-1},\mr{X}_p)\otimes \cdots\otimes\Cc(\mr{X}_0,\mr{X}_1)),\Cc(\mr{X}_0,\mr{X}_p)\right).$$ Then have a decomposition of the compactly supported Hochschild complex \begin{equation}\label{cdecomp}
\mr{CC}^{\bullet}_c(\Cc)=\bigoplus_{q\in \Z} \mr{CC}^{\bullet,q}(\Cc).\end{equation}
\begin{remark}
Observe that the decompostion \eqref{cdecomp} is a splitting of complexes which holds precisely because $\Cc$ has no higher multiplications. For a general $\Aa$-category, we only get a spectral sequence for the cohomologies instead.
\end{remark}
The next results are straightforward generalisations of algebra versions of Lunts. 
\2
\begin{proposition}\label{hochbasechange}
 Let $\Cc$ be a finitely defined $\Aa$-category over $\mr{R}$. Assume that $\Cc$ is finite and flat of finite type over $\mr{R}$. Suppose that $\varphi:\mr{R} \to \mr{S}$ is a morphism of commutative rings. 
 \begin{enumerate}
  \item $\mr{CC}^{\bullet}_c(\mr{Ind}_\varphi(\Cc)) = \mr{Ind}_\varphi\left(\mr{CC}^{\bullet}_c(\Cc)\right)$.
  \item Assuming $\varphi$ flat, $\mr{HH}_c^{\bullet}(\mr{Ind}_\varphi(\Cc)) = \mr{Ind}_\varphi(\mr{HH}^{\bullet}_c(\Cc))$.
 \end{enumerate}
\end{proposition}
\begin{proof}
 Clearly it is enough to prove the first assertion. Since $\Cc$ is finite and its morphism spaces are finite projective over $\mr{R}$, so are their tensor products. The claim now follows immediately from the definition of $\mr{CC}^{\bullet}_c(\Cc)$.
\end{proof}
\2
\begin{corollary}
In the situation of \cref{hochbasechange} assume $\Cc$ is a graded $\mr R$-linear category. Then the conclusions remain true for the bi-graded Hochschild complex and its cohomology.
\end{corollary}
\2
\begin{proposition}\label{extraformalityres}
 Suppose that $\mr{R}$ is Noetherian. Consider a graded category $\Cc$ over $\mr{R}$ as in \cref{hochbasechange} and assume and that for all $p,q \in \mathbf{Z}$ the $\mr{R}$-module $\mr{HH}^{p,q}(\Cc)$ is projective. Then, for any morphism of commutative rings $\varphi:\mr{R} \to \mr{S}$, we have $$\mr{HH}^{p,q}(\mr{Ind}_\varphi(\Cc)) = \mr{Ind}_\varphi(\mr{HH}^{p,q}(\Cc)).$$
\end{proposition}
\begin{proof}
 This follows from the next lemma.
\end{proof}
\begin{lemma}[Lunts \cite{MR2578584}]
 Let $\mr{R}$ be Noetherian and assume $(\mr{K},\mr{d})$ is a bounded below complex of finite projective $\mr{R}$-modules such that each $\mr{R}$-module $\mr{H}^p(\mr{K})$ is projective. Then, for each $p$, $\mr{Im}(\mr{d}^p)$ is projective over $\mr{R}$ and hence $\mr{K}$ is homotopy equivalent to its cohomology $\oplus_p \mr{H}^{p}(\mr{K})[-p]$.
\end{lemma}
\2
\begin{corollary}
Let $\Cc$ and $\varphi$ be as in \cref{extraformalityres}. Then we have $$\mr{HH}^{n}_c(\mr{Ind}_\varphi(\Cc)) = \mr{Ind}_\varphi(\mr{HH}^{n}_c(\Cc)).$$
\end{corollary}

\subsection{Clabi-Yau structures}
Recall that a dg category $\Cc$ over $\mr{R}$ is called proper if for all objects $\mr{X,Y}$, the space of morphisms $\Cc(\mr{X,Y})$ is a perfect complex of $\mr{R}$-modules. We let $\Cc^{\mr{op},\vee}$ be the right dual of $\Cc$ as a module over $\Cc\otimes \Cc^\mr{op}$, i.e. $$\Cc^{\mr{op},\vee} : (\mr{X,Y}) \mapsto \Cc(\mr{Y,X})^\vee.$$ We have a canonical equivalence \begin{equation}\label{hhdual}\mr{RHom}_{\mr{R}}(\Cc\otimes_{\Cc\otimes \Cc^{\mr{op}}}\Cc,\mr{R})\simeq \mr{RHom}_{\Cc\otimes \Cc^{\mr op}}\left(\Cc,\Cc^{\mr{op},\vee}\right).\end{equation}
We denote the cohomology of the complex on the left by $\mr{HH}_\bullet^{\vee}(\Cc)$. 
\2
\begin{remark}
If $\mr{R}$ is a field, so injective over itself, we get just dual of the Hochschild homology. In general, the situation is more complicated, although in sufficiently nice circumstances, e.g. when $\mr{R}$ has global dimension $1$, we can control this via spectral sequences.
\end{remark}
\2
\begin{definition}
Let $\Cc$ be a proper dg category. A weakly proper $n$-Calabi-Yau structure on $\Cc$ is an element in $\mr{HH}^{\vee}_{-n}(\Cc)$ such that the corresponding morphism $\Cc\to \Cc^{\mr{op},\vee}[-n]$ is an isomorphism.
\end{definition}
\2
\begin{proposition}
Let $\Cc$ be a weakly proper $n$-Calabi-Yau dg category. There is a canonical identification $$\mr{HH}^{\vee}_{\bullet}(\Cc)\simeq \mr{HH}^{\bullet+n}(\Cc).$$
\end{proposition}
\begin{proof}
This follows immediately from \eqref{hhdual} and the definition.
\end{proof}
\begin{corollary}\label{hhinj}
Let $\Cc$ be a  weakly proper Calabi-Yau dg category over $\mr{R}$ such that $\mr{HH}_\bullet(\Cc)$ is projective. Then for any flat $\mr{R}\xhookrightarrow{}\mr{Q}$, we have an injection $$\mr{HH}^{\bullet}(\Cc)\otimes_{\mr{R}} \mr{Q}\xhookrightarrow{} \mr{HH}^{\bullet}(\Cc\otimes_{\mr{R}}\mr{Q}).$$
\end{corollary}
\begin{proof}
For any projective $\mr{R}$-module, we have $\mr{M}^\vee\otimes_{\mr{R}} \mr{Q} \xhookrightarrow{} \left(\mr{M}\otimes_\mr{R}\mr{Q}\right)^\vee$. Hence the result follows at once from the above proposition, as $\mr{HH}_\bullet^\vee(\Cc)=\mr{HH}_\bullet(\Cc)^\vee$ by our assumptions, and flat base change for Hochschild homology.
\end{proof}

\subsection{Kaledin classes and formality criteria}
In this section, we shall generalise two important results on formality of $\Aa$-algebras due to Kaledin and Lunts. 
\2
\begin{theorem}[Lunts \cite{MR2578584}]
 Let $\mr{A}$ be a minimal $\Aa$-algebra over $\mr{R}$ which is projective as an $\mr{R}$-module. Then $\mr{A}$ is formal if and only if it is $\mr{A}_n$-formal for all $n$.
\end{theorem}
Kaledin \cite{MR2372207} shows that $\mr{A}_n$-formality is obstructed by a Hochschild cohomology class, called the Kaledin class, and that the next holds: 
\2
\begin{theorem}[Kaledin-Lunts \cite{MR2372207},\cite{MR2578584}]\label{formalityalg}
 Let $\mr{R}$ be an integral domain with field of fractions $k(\eta)$. Consider a minimal $\Aa$-algebra $\mr{A}$ over $\mr{R}$ which is a finite projective $\mr{R}$-module. Assume that the $\mr{R}$-module $\mr{HH}^{2}_{\mr{c}}(\mr{A}(2))$ is torsion-free. If $\mr{A}_\eta = k(\eta) \otimes_{\mr{R}} \mr{A}$ is formal, then $\mr{A}$ is formal. In particular, $\mr{A}_{\mathfrak{p}}$ is formal for all $\mathfrak{p} \in \mr{Spec(R)}$.
\end{theorem}
Consider now an $\Aa$-category $\mathcal{C}$ and its associated dg cocategory $\mathcal{B}(\mathcal{C})$. The codifferential of $\mathcal{B}(\mathcal{C})$ is an element of degree $1$ in $\mr{coDer}\left(1_{\mathcal{B}(\mathcal{C})}\right)$. Hence we may consider its weight decomposition $$\mr{d}_{\mathcal{B}(\mathcal{C})} = \mr{d}_{\mathcal{B}(\mathcal{C}),1}+\mr{d}_{\mathcal{B}(\mathcal{C}),2}+\cdots.$$
\begin{definition}
Let $\mathcal{C}$ be a minimal $\Aa$-category. The Kaledin class of $\mathcal{C}$ is defined as $$\mr{k}_\mathcal{C}=\mr{d}_{\mathcal{B}(\mathcal{C}),3}+2\mr{d}_{\mathcal{B}(\mathcal{C}),4}+\cdots.$$
\end{definition}

Clearly $\mr{k}_\mathcal{C}$ is a degree $1$ element of $\mr{W}_2\mr{coDer}\left(1_{\mathcal{B}(\mathcal{C})}\right)$. This gives for all $n,m$ $$[\mr{d}_{\mathcal{B}(\mathcal{C}),n},\mr{d}_{\mathcal{B}(\mathcal{C}),m}]=[\mr{d}_{\mathcal{B}(\mathcal{C}),m},\mr{d}_{\mathcal{B}(\mathcal{C}),n}].$$ Recalling $[\mr{d}_{\mathcal{B}(\mathcal{C})},\mr{d}_{\mathcal{B}(\mathcal{C})}]=0$, we get $$[\mr{d}_{\mathcal{B}(\mathcal{C})},\mr{k}_\mathcal{C}]=0.$$
Hence $\mr{k}_\mathcal{C}$ is a cocycle and we get a well defined class in $\mr{W}_2\mr{HH}^{2}(\mathcal{C})$. We shall abuse notation and denote by $\mr{k}_\Cc$ this Hochschild cohomology class.
\2
\begin{remark}
Similarly there are Kaledin classes associated to $\mr{A}_n$-categories. Any $\Aa$-category $\mathcal{C}$ can be considered as $\mr{A}_n$-category and the corresponding Kaledin class is the truncation of $\mr{k}_\mathcal{C}$ at weight $n$, denoted $\mr{k}_\mathcal{C}^{\le n}$. It is a cocycle in $\mr{W}_2\mr{coDer}(1_{\mathcal{B}(\mathcal{C})})/\mr{W}_{n+1}\mr{coDer}(1_{\mathcal{B}(\mathcal{C})})$.
\end{remark}
\2
\begin{lemma}
Let $m$ and $m'$ be two minimal $\Aa$-structures on $\Cc$. Let $F:(\Cc,m)\to (\Cc,m')$ be an $\Aa$-isomorphism with $F=1_{\mr{Ob}(\Cc)}$. Then it induces an isomorphism $\mr{HH}(\Cc,m)\to \mr{HH}(\Cc,m')$ which is compatible with the weigh filtrations. Moreover, $\mr{HH}(\mr{F})\left(\mr{k}_{(\Cc,m)}\right)=\mr{k}_{(\Cc,m')}$.
\end{lemma}
\begin{proof}
We denote by $\mr{F}=(\mr{F}_1,\mr{F}_2,\cdots)$ the cofunctor between the bar constructions too. Let $\mr{F}_*$ and $\mr{F}^*$ be post and pre-composition with the cofunctor $\mr{F}$. Note that these operations preserve coderivations. The quasi-isomorphism on the cochain level is defined as $$\mr{CC}^\bullet(\mr{F})\coloneqq\mr{F}^{-1,*}\circ \mr{F}_*:\mr{CC}^\bullet(\Cc,m)\to \mr{CC}^\bullet(\Cc,m').$$ It respects the weight filtrations since so do the maps $\mr{F}^{-1,*}$ and $\mr{F}_*$.\\ Similarly to the construction of the Kaledin class, we let $\partial\mr{F}$ be the $(\mr{F,F})$-coderivation determined by $(0,\mr{F}_2,2\mr{F}_3,\cdots)$. By inspecting the construction via the universal property of the reduced tensor coalgebra, this is precisely the map we obtain by changing all higher components of the cofunctor $\mr{F}$ via the same procedure. Hence $\partial\mr{F}\circ \mr{F}^{-1}$ is a degree $0$ coderivation and a calculation shows that $$\mr{CC}^\bullet(\mr{F})\left(\mr{k}_{(\Cc,m)}\right)-\mr{k}_{(\Cc,m')}=[\mr{d}_{\mathcal{B}(\Cc,m')},\partial\mr{F}\circ \mr{F}^{-1}].$$
\end{proof}
\begin{theorem}\label{kal}
Let $\mathcal{C}$ be a minimal $\Aa$-category over $\mr{R}$ and assume $\mr{R}$ is a $\Q$-algebra. The following are equivalent:
\begin{enumerate}
\item The category $\mathcal{C}$ is formal.
\item The category $\mathcal{C}$ is $\mr{A}_n$-formal for all $n\ge 1$.
\item For all $n\ge 3$, the truncated Kaledin class $\mr{k}^{\le n}_\mathcal{C}$ vanishes.
\item The Kaledin class $\mr{k}_\mathcal{C}$ vanishes.
\end{enumerate}
\end{theorem}

\begin{proof}
The implication $1 \Rightarrow 2$ is trivial. Since the (truncated) Kaledin class is invariant under $\Aa$-isomorphisms, we get immediately that $2 \Rightarrow 3$. By definition we have $3 \Rightarrow 4$.\\
We turn to $4\Rightarrow 1$. Suppose now $\mr{k}_\Cc=0$. We are going to construct an isomorphism $\mr{H}^0\Cc \simeq \Cc$. Since the Kaledin class vanishes, we may pick $\tau^{(2)} \in \mr{W}_2\mr{coDer}^0(1_{\mathcal{B}(\Cc)})$ such that $[\mr{d}_{\mathcal{B}(\Cc)},\tau^{(2)}]=\mr{k}_\Cc$. In particular, for $n=3$, we get $$[\mr{d}_{\mathcal{B}(\Cc),2}, \tau_2^{(2)}]=\mr{d}_{\mathcal{B}(\Cc),3}.$$ Hence the map $$\mr{d}_{\mathcal{B}\left(\Cc,\tau_2^{(2)}\right)}\coloneqq \mr{exp}(-\tau_2^{(2)})\circ \mr{d}_{\mathcal{B}(\Cc)}\circ \mr{exp}(\tau_2^{(2)})$$ is a degree $1$ coderivation on $\overline{\mr T}(\Cc[1])$ and we have an isomorphism of dg cocategories $$\mr{exp}(\tau_2^{(2)}):\left(\overline{\mr{T}}(\Cc[1]),\mr{d}_{\mathcal{B}\left(\Cc,\tau_2^{(2)}\right)}\right)\simeq \left(\mathcal{B}(\Cc),\mr{d}_{\mathcal{B}(\Cc)}\right).$$ We have $\mr{d}_{\mathcal{B}\left(\Cc,\tau_2^{(2)}\right),3}=\mr{d}_{\mathcal{B}(\Cc),3}-[\mr{d}_{\mathcal{B}(\Cc),2},\tau_2^{(2)}]=0$. Let us denote the dg cocategory on the left by $\mathcal{B}\left(\Cc,{\tau_2^{(2)}}\right)$. We may now apply the same procedure to $\mathcal{B}\left(\Cc,{\tau_2^{(2)}}\right)$ to find $\tau_3^{(3)}$ such that $$\mr{exp}(\tau_3^{(3)}):\left(\overline{\mr{T}}(\Cc[1]),\mr{d}_{\mathcal{B}\left(\Cc,\tau_3^{(3)}\right)}\right)\simeq \mathcal{B}\left(\Cc,{\tau_2^{(2)}}\right).$$ For the dg cocategory $\mathcal{B}\left(\Cc,{\tau_3^{(3)}}\right)$, we have $$\mr{d}_{\mathcal{B}\left(\Cc,\tau_3^{(3)}\right),3}=\mr{d}_{\mathcal{B}\left(\Cc,\tau_3^{(3)}\right),4}=0.$$ By induction, we get a sequence $\tau_{n}^{(n)}$ for all $n\ge 2$ such that the codifferential of $\mathcal{B}\left(\Cc,{\tau_n^{(n)}}\right)$ vanishes in weights $\le n+1$. The infinite composition $$\mr{exp}\left(\tau_2^{(2)}\right)\circ \mr{exp}\left(\tau_3^{(3)}\right)\circ \cdots \circ \mr{exp}\left(\tau_n^{(n)}\right)\circ \cdots : \mr{H}^0\Cc \to \Cc$$ is the required isomorphism which is well-defined as $\tau_n^{(n)}$ is of weight $n$, hence its exponential is identity on weights $<n$, thus for any fixed weight, we have a finite composition. 
\end{proof}
\begin{corollary}\label{a2kaledin}
Let $(\Cc,m)$ be a minimal $\Aa$-category such that $m_3=\cdots=m_n=0$. Then $\Cc$ is $\mr{A}_{n+1}$-formal if and only if $\mr{k}_{\Cc}^{\le n+1}=[m_{n+1}]=0$ in $\mr{HH}^2(\Cc,m_2)$.
\end{corollary}
\begin{proof}
This an immediate consequence of the preceding theorem and the assumptions on the vanishing of the higher multiplications.
\end{proof}
\begin{corollary}\label{a2generic}
Suppose $(\Cc,m)$ is a finite minimal $\Aa$-category which is flat and proper over an integral domain $\mr{R}$ with generic point $\eta$. Assume the $\mr{R}$-module $\mr{HH}^2_c(\Cc,m_2)$ is torsion-free. Then $\Cc$ is formal iff $\Cc_\eta$ is formal.
\end{corollary}
\begin{proof}
If $\Cc$ is formal, then so is its $k(\eta)$ base-change $\Cc_\eta$. Suppose $\Cc_\eta$ is formal. It suffices to prove $\Cc$ is $\mr{A}_n$-formal for all $n\ge 1$. We proceed by induction base case being trivially true. Assume $\Cc$ is $\mr{A}_n$-formal. Then, we may assume we are in the situation of \cref{a2kaledin}. Hence $\Cc$ is $\mr{A}_{n+1}$-formal iff $$\mr{k}_{\Cc}^{\le n+1}=[m_{n+1}]=0$$ in $\mr{HH}^{2}_c(\Cc,m_2)\subset\mr{HH}^2(\Cc,m_2)$ Since $\mr{HH}_c^2(\Cc,m_2)$ commutes with flat base-change, the formality of $\Cc_\eta$ implies that $\left(\mr{k}_{\Cc}^{\le n+1}\right)_\eta=\mr{k}_{\Cc_\eta}^{\le n+1}=0$, so $\mr{k}_{\Cc}^{\le n+1}$ is torsion. By assumption $\mr{HH}^2_c(\Cc,m_2)$ is torsion-free, hence $\mr{k}_{\Cc}^{\le n+1}=0$ and $\Cc$ is $\mr{A}_{n+1}$-formal.
\end{proof}
We shall formulate a dg version of \cref{a2generic} for future use.
\2
\begin{corollary}\label{a2dg}
Let $\Cc$ be a flat, proper finite dg category over an integral domain $\mr{R}$ with generic point $\eta$. Assume the $\mr{R}$-module $\mr{HH}^2_c(\mr{H}\Cc,m_2)$ is torsion-free, where $m_2$ is the composition in $\Cc$. Then $\Cc$ is formal iff $\Cc_\eta$ is formal.
\end{corollary}
Finally, as finite categories are very restrictive (and are essentially just algebras), we have the following formality criterion which drops the finiteness requirement at the expense of introducing other constraints.
\2
\begin{theorem}\label{a2cy}
Let $\Cc$ be a flat weakly proper $n$-Calabi-Yau differential graded category over an intergral domain $\mr{R}$ with generic point $\eta$. Assume the $\mr{R}$-module $\mr{HH}_\bullet(\mr{H}\Cc,m_2)$ is projective, where $m_2$ is the composition in $\Cc$. Then $\Cc$ is formal if and only if $\Cc_\eta=\mr{Ind}_{k(\eta)/\mr{R}}(\Cc)$ is formal.
\end{theorem}
\begin{proof}
By the Calabi-Yau property, we have an isomorphism of $\mr{R}$-modules $$\mr{HH}_\bullet^\vee\left(\mr{H}\Cc\right)\simeq\mr{HH}^{\bullet+n}\left(\mr{H}\Cc\right).$$ We can calculate the left side via the spectral sequence $$\mr{E}_2^{p,q}=\mr{Ext}_{\mr{R}}^p\left(\mr{HH}_q\left(\mr{H}\Cc\right),\mr{R}\right)\Rightarrow \mr{Ext}^{p+q}_{\mr{R}}\left(\mr{CC}_\bullet\left(\mr{H}\Cc\right),\mr{R}\right).$$ Since $\mr{HH}_\bullet(\mr{H}\Cc,m_2)$ is projective over $\mr{R}$, the spectral sequence collapses on the second page. Hence $\mr{HH}^\bullet\left(\mr{H}\Cc\right)$ is the $\mr{R}$-dual of $\mr{HH}_\bullet\left(\mr{H}\Cc\right)$ and is, therefore, torsion-free.\\
Next, we claim that the natural map 
\begin{equation}\label{injmap}
\mr{HH}^\bullet\left(\mr{H}\Cc\right)\otimes_\mr{R} k(\eta) \to \mr{HH}^\bullet(\mr{H}\Cc_\eta)
\end{equation}is injective. Indeed, $k(\eta)$ is flat over $\mr{R}$, so \cref{hhinj} implies the claim. Now, by \cref{a2kaledin}, we may work by induction on $n$, then injectivity of the map \eqref{injmap} and formality of $\Cc_\eta$ imply that the obstruction classes map to $0$ and hence are torsion in the $\mr{R}$-module $\mr{HH}^2\left(\mr{H}\Cc\right)$ which is torsion-free, so must vanish.
\end{proof}

\section{Deformation quantisation and perverse sheaves}\label{sec2}

\subsection{Holomorphic contact and symplectic manifolds}
This section introduces the main geometric objects of interest, mainly fixing notation and terminology. We refer to \cite{saf} for more careful and detailed exposition. We start with the complex analogues of real symplectic geometry.
\2
\begin{definition}
A holomorphic symplectic manifold is a pair $(\mr{X},\sigma)$, where $\mr{X}$ is a complex manifold and $\sigma$ is non-degenerate holomorphic $2$-form. 
\end{definition}
\2
\begin{example}
\begin{enumerate}
\item The most basic example of a holomorphic symplectic manifold is given by cotangent bundles of complex manifolds. 
\item There is a holomorphic Darboux theorem, meaning that the local models of holomorphic symplectic manifolds are open subsets of cotangent bundles.
\end{enumerate}
\end{example}
\2
\begin{definition}
A complex Lagrangian subvariety of $(\mr{X},\sigma)$ is a locally closed subvariety $\mr{L}$ of dimension $\mr{dim}\mr{X}/2$ such that the restriction of the symplectic form to the regular locus vanishes, i.e. $\sigma|_{\mr{L}_{\mr{sm}}}=0$. When $\mr{L}$ is smooth, we call it a submanifold.
\end{definition}
Similarly, there is a holomorphic version of contact geometry.
\2
\begin{definition}
A holomorphic contact manifold is a triple $(\mr{Y},\mathscr{L},\alpha)$, where $\mr{Y}$ is a complex manifold of dimension $2n+1$, $\mathscr{L}$ is a line bundle on it and $\alpha \in \Gamma(\mr{Y},\Omega_\mr{Y}\otimes \mathscr{L})$ is such that $\alpha\wedge(\mr{d}\alpha)^n$ is non-degenerate.
\end{definition}
We call $\mathscr{L}$ the contact line bundle and $\alpha$ - the contact form. A contactomorphism $\phi :(\mr{Y}_1,\alpha_1) \to (\mr{Y}_2,\alpha_2)$ is a local biholomorphism such that $\phi^*\alpha_2=\alpha_1$. 
\2
\begin{example}
\begin{enumerate}
\item Here the most simple example is the projectivisation of the cotangent bundle of a complex manifold. 
\item Projective spaces of odd dimension are contact.
\item There is a contact version of the Darboux's theorem asserting that the local models of contact manifolds are given by open subsets of projectivised cotangent bundles.
\end{enumerate}
\end{example}
The analogue of Lagrangian subvarieties is as follows:
\2
\begin{definition}
A complex Legendrian subvariety of $(\mr{Y},\mathscr{L},\alpha)$ is a locally closed subvariety $\Lambda$ of dimension $n$ such that $\alpha|_{\Lambda_\mr{sm}}=0$.
\end{definition}
Next, we describe how to go back and forth between the symplectic and contact geometries.
\2
\begin{definition}
Let $(\mr{Y},\mathscr{L},\alpha)$ be a contact manifold. Its canonical symplectisation is the principal $\C^*$-bundle $\gamma: \tilde{\mr{Y}}\to \mr{Y}$ associated to the line bundle $\mathscr{L}^{-1}$ equipped with the $2$-form $\sigma = -\mr{d}\tilde{\alpha}$, where $\tilde{\alpha}$ is the lift of $\alpha$ to $\tilde{\mr{Y}}$.
\end{definition}
\2
\begin{remark}
\begin{enumerate}
\item In fact the non-degeneracy of $\sigma$ in the above is equivalent to the non-degeneracy condition on $\alpha$. In any case, we get $\mathscr{L}^{\otimes -(n+1)}\cong \mr{K}_\mr{Y}$ and so the Lie derivative along the Euler vector field $\mr{Lie}_{eu}\sigma=\sigma$, i.e. the symplectic form is homogeneous of weight $1$.
\item We could have defined contactomorphisms as isomorphisms of the associated principal $\C^*$-bundles.
\end{enumerate}
\end{remark}
Note that any contactomorphism lifts to a homogeneous symplectomorphism of the canonical symplectisations.
\2
\begin{definition}
A contactification of a symplectic manifold $(\mr{X},\sigma)$ is a $\C$-principal bundle $$\rho:\mr{Y}\to \mr{X}$$ with a connection whose curvature is $-\sigma$. 
\end{definition}
A contactification is naturally a contact manifold with contact form given by the connection $1$-form. We will also need to contactify Lagrangians in $\mr{X}$. 
\2
\begin{definition}
Let $\mr{L}$ be a Lagrangian subvariety in $\mr{X}$. A contactification of $\mr{L}$ is a contactification $\rho:\mr{Y}\to \mr{X}$ of $\mr{X}$ and a Legendrian (for the connection $1$-form) $\Lambda \subset\mr{Y}$ such that $\rho|_\Lambda:\Lambda	\to \mr{L}$ is a homeomorphism and induces an isomorphism of the smooth loci.
\end{definition}
Given a symplectic manifold $\mr{X}$ with $2$-form $\sigma$, its contactifications form a $\C$-gerbe $\mathbf{\mathfrak{C}}_\sigma$, whose objects are pairs $(\rho:\mr{V}\to \mr{U},\alpha)$, where $\mr{U}\subset \mr{X}$ is open, $\rho:\mr{V}\to \mr{U}$ is a principal $\C$-bundle and $\alpha$ is a vertical $1$-form such that $\mr{d}\alpha=\rho^*\sigma$. Morphisms in $\mathbf{\mathfrak{C}}_\sigma$ are morphisms of principal bundles compatible with the $1$-forms, i.e. the domain $1$-form is the pull-back of the codomain $1$-form.\\
The following will be important later:
\2
\begin{proposition}[Gunningham-Safronov \cite{saf}]
Let $(\mr{X},\sigma)$ be a symplectic manifold and $\mr{L}\subset \mr{X}$ be a Lagrangian subvariety. There exists a canonical section of $\mathbf{\mathfrak{C}}_\sigma|_\mr{L}$.
\end{proposition}
In other words, there exists a unique contactification of $\mr{X}$ in a neighbourhood of $\mr{L}$ together with a Legendrian lift of $\mr{L}$ which contactifies it.
\subsection{Microdifferential modules}
Let $\mr{X}$ be a complex manifold. The cotangent bundle $\Omega_\mr{X}$ supports a filtered sheaf of $\C$-algebras $\widehat{\mathscr{E}}_{\Omega_\mr{X}}$ of formal microdifferential operators. Fix $(x_1,\cdots,x_n)$ local coordinates on $\mr{X}$, and write $(x_1,\cdots,x_n,\xi_1,\cdots,\xi_n)$ for the induced coordinates on $\Omega_\mr{X}$. Let $\mathscr{O}_{\Omega_\mr{X}}(m)$ be the sheaf of homogeneous functions in the fibre coordinates on $\Omega_\mr{X}$ of degree $m$, i.e. $$\left(\sum \xi_j\partial/\partial\xi_j - m\right)f(x,\xi)=0.$$
 We define the sheaf of formal microdifferential operators of order $\le m$ by $$\widehat{\mathscr{E}}_{\Omega_\mr{X}}(m) = \prod_{j \in \mathbf{N}} \mathscr{O}_{\Omega_\mr{X}}(m-j).$$ In order to get a sheaf globally on $\Omega_\mr{X}$, glue these sheaves on overlaps using the transformation rule for total symbols of differential operators. Taking the limit over $m \in \mathbf{Z}$, we get the sheaf of formal microdifferential operators on $\Omega_\mr{X}$: $$\widehat{\mathscr{E}}_{\Omega_\mr{X}} = \varinjlim_{m \in \mathbf{Z}}\widehat{\mathscr{E}}_{\Omega_\mr{X}}(m).$$
Let $\sigma_m: \widehat{\mathscr{E}}_{\Omega_\mr{X}}(m) \to \mathscr{O}_{\Omega_\mr{X}}(m)$ be the principal symbol. There is an isomorphism $$\mr{gr}\widehat{\mathscr{E}}_{\Omega_\mr{X}}\simeq \bigoplus_{m\in \Z}\mathscr{O}_{\Omega_\mr{X}}(m).$$
There are products $\widehat{\mathscr{E}}_{\Omega_\mr{X}}(l)\otimes_{\C}\widehat{\mathscr{E}}_{\Omega_\mr{X}}(m) \to \widehat{\mathscr{E}}_{\Omega_\mr{X}}(l+m)$, given by $$(f\star g)_{l}(x,\xi) = \sum_{\substack{l=i+j-|\alpha|\\ \alpha \in \mathbf{N}}}\frac{1}{\alpha!}(\partial/\partial\xi_1)^{\alpha_1}\cdots (\partial/\partial\xi_n)^{\alpha_n}f_i(x,\xi)\cdot (\partial/\partial x_1)^{\alpha_1}\cdots (\partial/\partial x_n)^{\alpha_n}g_j(x,\xi).$$ In particular $\widehat{\mathscr{E}}_{\Omega_\mr{X}}$ and $\widehat{\mathscr{E}}_{\Omega_\mr{X}}(0)$ are sheaves of (non-commutative) $\C$-algebras. 
 \2
 \begin{remark}
 Notice that the total symbol of a differential operator is a polynomial in $\xi_1,\cdots, \xi_n$ and now we are allowing symbols which are general holomorphic functions rather than just polynomials.
 \end{remark}
 Let $\mr{\mathbf{P}}(\Omega_\mr{X})$ be the projectivised cotangent bundle of $\mr{X}$. The sheaf $\widehat{\mathscr{E}}_{\Omega_\mr{X}}$ is constant on the fibres of the projection $\pi:\Omega_\mr{X}\setminus \mr{X} \to \mr{\mathbf{P}}(\Omega_\mr{X})$ and we shall use the same notation for its pushforward to $\mr{\mathbf{P}}(\Omega_\mr{X})$. We shall eventually be interested in modules over the stack of formal microdifferential operators on a general contact manifold. The next properties of modules over $\Em$ are local and will readily extend to the stack version. \\
 Let $\Lambda$ be a Legendrian subvariety of $\mr{\mathbf{P}}(\Omega_\mr{X})$. There is a lattice in $\Em$ associated to $\Lambda$ defined as the subalgebra of $\Em|_\Lambda$ generated by $$\mr{I}_\Lambda=\{P\in \Em(1)|_\Lambda \; |\; \sigma_1(P)|_\Lambda=0\}.$$ We denote this subalgebra by $\widehat{\mathscr{E}}_{\Lambda/\Omega_\mr{X}}$.
 \2
 \begin{definition}
 \begin{enumerate}
 \item A coherent module $\mathscr{M}$ over $\Em$ is holonomic if its support is a Legendrian subvariety.
 \item A holonomic $\Em$-module $\mathscr{M}$ is regular holonomic if  locally there is a coherent $\Em(0)$-module $\mathscr{M}^0$ which generates $\mathscr{M}$ over $\Em$ and $\mathscr{M}^0/\Em(-1)\mathscr{M}^0$ is a $\mathscr{O}_\Lambda(0)$-module.
 \item A regular holonomic $\Em$-module $\mathscr{M}$ is simple along $\Lambda$ if $\mathscr{M}^0/\Em(-1)\mathscr{M}^0$ is an invertible $\mathscr{O}_\Lambda(0)$-module.
 \end{enumerate}
 \end{definition}
 \2
 \begin{remark}
 Regularity is equivalent to requiring that $\mathscr{M}^0$ be $\mr{I}_\Lambda$-invariant.
 \end{remark}
We recall briefly the microlocalisation of microdifferential modules, again referring to \cite{saf} for more details and references. Fix a Legendrian $\Lambda$ in $\mr{\mathbf{P}}(\Omega_\mr{X})$ and let $$\gamma :\Omega_\mr{X}\setminus \mr{X}\to \mathbf{P}(\Omega_\mr{X})$$ be its symplectisation. The canonical filtration on $\Em$ induces a filtration on the algebra $\widehat{\mathscr{E}}_{\Lambda/\Omega_\mr{X}}$. There is a canonical embedding $$\mr{gr}_\Lambda\Em\xhookrightarrow{} \gamma_*\mr{D}_{\widetilde{\Lambda}}^{\mr{K}_{\widetilde{\Lambda}|\mr{X}}^{1/2}}.$$ and a regular holonomic $\Em$-module $\mathscr{M}$. Then $\mathscr{M}$ carries a canonical $V$-filtration $V_{\Lambda}\mathscr{M}$ along the Legendrian $\Lambda$. We denote the associated graded by $\mr{gr}_\Lambda\mathscr{M}$.
\2
\begin{definition}
Let $\Lambda$ be a Legendrian in $\mr{\mathbf{P}}(\Omega_\mr{X})$ and let $\widetilde{\Lambda}$ be the homogeneous Lagrangian. The microlocalisation functor $$\mu_\Lambda: \mr{Mod}_\mr{rh}\left(\Em\right)\to \gamma_*\mr{Mod}_\mr{rh}\left(\mr{D}_{\widetilde{\Lambda}}^{\mr{K}_{\widetilde{\Lambda}|\mr{X}}^{1/2}}\right)$$ defined by $$\mathscr{M}\mapsto \mr{D}_{\widetilde{\Lambda}}^{\mr{K}_{\widetilde{\Lambda}|\mr{X}}^{1/2}}\otimes_{\gamma^{-1}{\mr{gr}_\Lambda\Em}}\gamma^{-1}\mr{gr}_\Lambda\mathscr{M}.$$
\end{definition}
In particular, we have the following lemma.
\2
\begin{lemma}
An $\Em$-module $\mathscr{M}$ is simple along $\Lambda$ if $\mu_\Lambda(\mathscr{M})$ is a line bundle.
\end{lemma}
 For the global version of the microdifferential modules, we shall need the following:
 \2
 \begin{definition}
 Let $\mr{X}$ be a complex manifold. The sheaf of half twisted microdifferential operators is defined as $$\Em^{\sqrt{v}}=\pi^{-1}\mr{K}_\mr{X}^{1/2}\otimes_{\pi^{-1}\mathscr{O}_\mr{X}}\Em\otimes_{\pi^{-1}\mathscr{O}_\mr{X}}\pi^{-1}\mr{K}_\mr{X}^{1/2}$$
 \end{definition}
 \2
 \begin{remark}
 We note that the above definitions and constructions extend without difficulty to the half twisted setting.
 
 In particular, we have a microlocalisation functor $$\mu_\Lambda: \mr{Mod}_\mr{rh}\left(\Em^{\sqrt{v}}\right)\to \gamma_*\mr{Mod}_\mr{rh}\left(\mr{D}_{\widetilde{\Lambda}}^{\sqrt{v}}\right).$$
 \end{remark}
 \subsection{Algebroid stacks}
 \begin{definition}
 Let $\mr{X}$ be a topological space. An $\mr{R}$-algebroid prestack on $\mr{X}$ is an $\mr{R}$-linear prestack $\mathscr{A}$ such that 
 \begin{enumerate}
 \item For any $x\in \mr{X}$, there is an open neighbourhood $x\in \mr{U}$ such that $\mathscr{A}(\mr{U})\not=\emptyset$.
 \item For any two objects $\sigma,\tau \in \mathscr{A}(\mr{U})$ and any $x\in \mr{U}$, there exists an open $x \in \mr{V}\subset \mr{U}$ such that $\sigma|_\mr{V}\cong \tau|_\mr{V}$.
\end{enumerate}
An $\mr{R}$-algebroid stack is an $\mr{R}$-algebroid prestack which is a stack.
\end{definition}
\2
\begin{example}Fix a topological space $\mr{X}$.
 Let $\mathscr{A}$ be a sheaf of $\mr{R}$-algebras on $\mr{X}$. The prestack $\mr{U} \to \mathscr{A}(\mr{U})^{+}$, where $\mathscr{A}(\mr{U})^{+}$ is the $\mr{R}$-linear category with one object whose endomorphisms are given by $\mathscr{A}(\mr{U})$, is an $\mr{R}$-algebroid prestack. \\
The associated stack $\mathscr{A}^{+}$ to this prestack is an algebroid stack. It is given by $$\mathscr{A}^{+}:\mr{U}\mapsto \{\text{locally free rank 1 } \mathscr{A}|_\mr{U}-\text{modules}\}.$$
Conversely, suppose that $\mathscr{A}$ is an algebroid. If $\mathscr{A}(\mr{X})$ is non-empty, choose any $\tau \in \mathscr{A}(\mr{X})$. We have an equivalence $\mathscr{A} \simeq \shom(\tau,\tau)^{+}$ of $\mr{R}$-algebroid stacks.
\end{example}
Given an $\mr{R}$-algebroid $\mathscr{A}$ over $\mr{X}$, let $\mathscr{M}_\mr{R}$ be the stack of sheaves of $\mr{R}$-modules on $\mr X$, we define the $\mr{R}$-linear abelian category of modules over $\mathscr{A}$ by $$\mr{Mod}(\mathscr{A}) = \mr{Fct}(\mathscr{A},\mathscr{M}_\mr{R}).$$
\2
\begin{definition}Let $\mathscr{R}$ be a sheaf of commutative $\C$-algebras.
\begin{itemize}
 \item An $\mathscr{R}$-algebroid is a $\C$-algebroid $\mathscr{A}$ together with a morphism of sheaves of $\C$-algebras $\mathscr{R} \to \mathscr{E}nd(\mr{id}_{\mathscr{A}})$.
 \item An $\mathscr{R}$-algebroid $\mathscr{A}$ is called invertible if $\mathscr{R}|_\mr{U} \to \mathscr{E}nd(\tau)$ is an isomorphism for every open $\mr{U}\subset \mr{X}$ and any $\tau \in \mathscr{A}(\mr U)$.
 \end{itemize}
\end{definition}
\subsection{Quantisation of contact manifolds}
Given a (holomorphic) contact manifold $\mr{Y}$, we may choose an open covering $\{\mr{U}_i\}_{i\in I}$ together with contact embeddings $\rho_i:\mr{U}_i\to \mr{\mathbf{P}}\Omega_{\mr{Y}_i}$ into projectivised cotangent bundles. The sheaves of formal microdifferential operators on these projectivised cotangent bundles do not glue and one need to rigidify the problem.\\
Let $\mr{X}$ be a complex manifold. Recall that locally the formal adjoint of a microdifferential operator $\mr{P}=\sum p_i \in \Em(m)$ is given by $\mr{P}^*=\sum p^*_i,$ where $$p_i^*(x,	\xi)= \sum_{\substack{i=k-|\alpha|\\ \alpha \in \mathbf{N}}}\frac{(-1)^{|\alpha|}}{\alpha!}(\partial/\partial\xi_1)^{\alpha_1}\cdots (\partial/\partial\xi_n)^{\alpha_n}\cdot (\partial/\partial x_1)^{\alpha_1}\cdots (\partial/\partial x_n)^{\alpha_n}p_k(x,-\xi).$$
We get a map $$*:\Em \to \pi^{-1}\mr{K}_\mr{X}\otimes_{\pi^{-1}\mathscr{O}_\mr{X}}\Em\otimes_{\pi^{-1}\mathscr{O}_\mr{X}}\pi^{-1}\mr{K}_\mr{X}^{-1}.$$
Hence, upon passing to half twists by the canonical bundle, we get a canonical anti-involution:
$$*:\Em^{\sqrt{v}}\to \Em^{\sqrt{v}}.$$
\begin{theorem}[\cite{MR3276591}]
Let $\left(\mr{Y},\mathscr{O}_\mr{Y}(1)\right)$ be a (holomorphic) contact manifold. There is a canonical filtered $\C$-algebroid $\widehat{\mr{E}}_\mr{Y}$ such that for any open $\mr{U}\subset \mr{Y}$ together with a contact embedding $\mr{U}\xhookrightarrow{} \mr{\mathbf{P}}\Omega_{\mr{X}}$, we have $$\widehat{\mr{E}}_\mr{Y}|_\mr{U}\simeq \left(\Em^{\sqrt{v}}|_\mr{U}\right)^{+}.$$ Moreover, there is a trivialisation $$\mr{gr}\widehat{\mr{E}}_\mr{Y}\simeq \left(\bigoplus_{m\in \Z}\mathscr{O}_\mr{Y}(m)\right)^{+}.$$
\end{theorem}
\2
\begin{remark}
Regularity is local and invariant under the gluing maps (quantised contact transformations) and so extends to $\widehat{\mr{E}}_\mr{Y}$.
\end{remark}
\2
\begin{definition}
Let $\mr{Y}$ be a holomorphic contact manifold. A coherent $\widehat{\mr{E}}_\mr{Y}$-module $\mathscr{D}$ is good if for any relatively compact open $\mr{U}\subset \mr{Y}$, there exists a $\widehat{\mr{E}}_\mr{Y}(0)$-module $\mathscr{D}^{0}$ generating $\mathscr{D}$ over $\mr{U}$.
\end{definition}
\subsection{Classification of \texorpdfstring{$\widehat{\mr{E}}_\mr{Y}$-}-modules}

Here we recall the classification of regular holonomic $\widehat{\mr{E}}_\mr{Y}$-modules.
\2
\begin{theorem}[\cite{MR2331247}, \cite{saf}]\label{miclass}
Let $\mr{Y}$ be a holomorphic contact manifold and let $\Lambda$ be a smooth Legendrian in $\mr{Y}$. Let $\gamma : \tilde{\mr{Y}}\to \mr{Y}$ be the symplectisation and $\tilde{\Lambda}=\gamma^{-1}(\Lambda)$. There is a global microlocalisation functor $$\mu_{\Lambda}:\mr{Mod}_{\Lambda,\mr{rh}}\left(\widehat{\mr{E}}_\mr{Y}\right)\to \gamma_*\mr{Mod}_\mr{rh}\left(\mr{D}_{\widetilde{\Lambda}}^{\sqrt{v}}\right).$$ Moreover, it induces an equivalence $$\mu_{\Lambda}:\mr{Mod}_{\Lambda,\mr{rh}}\left(\widehat{\mr{E}}_\mr{Y}\right)\simeq \gamma_*\mr{LocSys}_{\tilde{\Lambda}}^{\sqrt{v}}$$ between regular holonomic $\widehat{\mr{E}}_\mr{Y}$-modules along $\Lambda$ and the twisted local systems on the lift $\tilde{\Lambda}.$
\end{theorem} 
\2
We are going to be interested in a particular subset of the left side in the above theorem for a collection of Legendrians. 
\2
\begin{definition}\label{microor}
Let $\Lambda$ be a Legendrian in $\mr{Y}$ with its symplectisation $\gamma:\tilde{\mr{Y}}\to \mr{Y}$. A microdifferential orientation of $\Lambda$ is a good, simple $\widehat{\mr{E}}_\mr{Y}$-module $\mathscr{D}_\Lambda$ along $\Lambda$ such that $\mu_\Lambda\left(\mathscr{D}_\Lambda\right)$ is a square root of the canonical bundle $\mr{K}_{\tilde{\Lambda}}$ of the homogeneous Lagrangian corresponding to $\Lambda$.
\end{definition}
\subsection{DQ-algebras}
We let $\mr{X}$ be a complex manifold. Let $\C\brak$ be the ring of formal power series in $\hbar$, and $\C\cbrak$ its field of fractions, i.e. the field of formal Laurent series. Define a sheaf of $\C\brak$-algebras: $$\mathscr{O}_\mr{X}[\![\hbar]\!] = \varprojlim \mathscr{O}_\mr{X} \otimes_{\C} \C[\![\hbar]\!]/\hbar^{n}.$$
\begin{definition}
 A star product on $\mathscr{O}_\mr{X}[\![\hbar]\!]$ is a $\C\brak$-bilinear associative multiplication $\star$ such that $$f\star g = \sum_{i\ge 0} \mr{P}_i(f,g)\hbar^{i}, \text{ where } f,g \in \mathscr{O}_\mr{X},$$ such that $\mr{P}_i$ are holomorphic bidifferential operators with $\mr{P}_{0}(f,g)=fg$ and $\mr{P}_{i}(f,1) = \mr{P}_i(1,f)=0$ for all $i \ge 1$. The pair $(\mathscr{O}_\mr{X}[\![\hbar]\!],\star)$ is called a star algebra.
\end{definition}
\2
\begin{definition}
 A deformation quantisation algebra (DQ-algebra) on a complex manifold $\mr{X}$ is a sheaf of $\C\brak$-algebras $\mathscr{A}_\mr{X}$ locally isomorphic to a star algebra as a $\C\brak$-algebra.
\end{definition}
\2
\begin{example}\label{Poissonexample}
 Let $\mathscr{A}_\mr{X}$ be a DQ-algebra on $\mr{X}$. Let $\pi : \mathscr{A}_\mr{X} \to \mathscr{A}_\mr{X}/\hbar\mathscr{A}_\mr{X} \cong \mathscr{O}_\mr{X}$. For any $f,g \in \mathscr{O}_\mr{X}$, choose lifts $\tilde{f}, \tilde{g}$ such that $\pi(\tilde{f})=f$ and $\pi(\tilde{g})=g$. Then define a bracket $$\{f,g\} = \pi(\hbar^{-1}(\tilde{f}\tilde{g}-\tilde{g}\tilde{f})).$$ This is independent of the choices made and defines a Poisson structure on $\mr{X}$.
\end{example}
\2
\begin{example}
 Let $t$ be the coordinate on $\C$ and $(t;\tau)$ - the symplectic coordinates on $\Omega_{\C}$. Let $\Omega_{\mr{X}\times \C,\tau\not=0}$ be the open subset of $\Omega_{\mr{X}\times \C}$ where $\tau\not=0$. We have a map $$\rho : \Omega_{\mr{X}\times \C,\tau\not=0} \to \Omega_{\mr{X}}, \quad (x,t;\xi,\tau) \mapsto (x,\tau^{-1}\xi).$$ Define the subsheaf of operators independent of $t$: $$\widehat{\mathscr{E}}_{\Omega_{\mr{X}\times \C},\widehat{t}}(0)= \{\mr{P} \in \widehat{\mathscr{E}}_{\Omega_{\mr{X}\times \C}}(0) \text{ such that } [\mr{P},\partial_{t}]=0\}.$$ Then, letting $\hbar$ act as $\tau^{-1}$, we define the canonical DQ-algebra on $\Omega_{\mr{X}}$ by $$\widehat{\mathscr{W}}_{\Omega_\mr{X}}(0) = \rho_{*}\widehat{\mathscr{E}}_{\Omega_{\mr{X}\times \C},\widehat{t}}(0)\subset \rho_{*}\widehat{\mathscr{E}}_{\Omega_{\mr{X}\times \C}}(0).$$ The $\hbar$-localisation of $\widehat{\mathscr{W}}_{\Omega_\mr{X}}(0)$ is denoted by $\widehat{\mathscr{W}}_{\Omega_\mr{X}}$.
\end{example}
\2
\begin{remark}
We call the DQ algebra in the example above the canonical deformation quantisation of the cotangent bundle of a complex manifold. 
\begin{enumerate}
\item We have $\widehat{\mathscr{W}}_{\Omega_{\mr{X}}}(0)/\hbar\cong \mathscr{O}_{\Omega_\mr{X}}$, so it's indeed deformation quantisation.
\item $\widehat{\mathscr{W}}_{\Omega_\mr{X}}|_\mr{X}\cong \mr{D}_\mr{X}\cbrak$, that is, restricting to the zero section $\mr{X}$, we get differential operators over formal Laurent power series.
\item $\rho_{*}\widehat{\mathscr{E}}_{\Omega_{\mr{X}\times \C}}(0)$ is a flat $\widehat{\mathscr{W}}_{\Omega_\mr{X}}(0)$-module.
\end{enumerate}
\end{remark}
\subsection{DQ algebroids}
\begin{definition}
 A deformation quantisation algebroid (DQ-algebroid) on $\mr{X}$ is a $\C\brak$-algebroid $\mathscr{A}$ such that, for any open $\mr{U}\subset \mr{X}$ and $\tau \in \mathscr{A}(\mr{U})$, the $\C\brak$-algebra $\mathscr{H}om(\tau,\tau)$ is a DQ-algebra on $\mr{U}$.
\end{definition}
\2
If $\mr{X}$ is a holomorphic symplectic variety, then the holomorphic Darboux theorem implies that locally we have canonical DQ-algebras associated with $\mr{X}$, but they won't generally glue to a global DQ-algebra.\\ Just as in the contact case, one needs to work with half twists: similarly to the example above, we define the half-twisted DQ algebra $$\widehat{\mathscr{W}}^{\sqrt{v}}_{\Omega_\mr{X}}(0)= \rho_{*}\widehat{\mathscr{E}}_{\Omega_{\mr{X}\times \C},\widehat{t}}^{\sqrt{v}}(0)\subset \rho_{*}\widehat{\mathscr{E}}^{\sqrt{v}}_{\Omega_{\mr{X}}\times \C}(0).$$ Then the anti-involution restricts to $\widehat{\mathscr{W}}^{\sqrt{v}}_{\Omega_\mr{X}}(0)$ and takes $\hbar$ to $-\hbar$. Polesello and Schapira \cite{10.1155/S1073792804132819} construct a canonical DQ-algebroid gluing these twisted DQ algebras along quantised symplectic transformations:
\2
\begin{theorem}[\cite{10.1155/S1073792804132819}]
Let $\mr{X}$ be holomorphic symplectic. There exists a canonical filtered DQ algebroid, endowed with anti-involution $*$, $\widehat{\mathscr{W}}_{\mr{X}}(0)$ such that for any open $\mr{U}\subset \mr{X}$ with symplectic embedding $\mr{U}\xhookrightarrow{}\Omega_\mr{V}$, we have $$\widehat{\mathscr{W}}_{\mr{X}}(0)|_\mr{U}\simeq \widehat{\mathscr{W}}^{\sqrt{v}}_{\Omega_\mr{V}}(0)|_\mr{U}.$$Moreover, if $\rho:\mr{Y}\to \mr{X}$ is a contactification, there is a canonical embedding $$\widehat{\mathscr{W}}_{\mr{X}}(0)\subset \rho_*\widehat{\mr{E}}_\mr{Y}(0).$$
\end{theorem}
\begin{remark}
 Any other DQ-algebroid $\mathscr{A}_{\mr{X}}$ on $\mr{X}$ will be equivalent to $\widehat{\mathscr{W}}_{\mr{X}}(0)\otimes_{\C\brak}\mathscr{L}$ for some invertible $\C\brak$-algebroid $\mathscr{L}$. Hence DQ-algebroids are classified by $\mr{H}^{2}\big(\mr{X},\C\brak^{*}\big)$. We shall be predominantly working with the canonical DQ algebroid $\Ww(0)$, however, due to this classification, most results remain true for any $\mathscr{A}_\mr{X}$.
\end{remark}
\2
\begin{example}
 \cref{Poissonexample} shows that any DQ-algebroid on $\mr{X}$ induces a Poisson structure on $\mr{X}$. Conversely, \cite{MR2062626} shows that in the $\mr{C}^{\infty}$ setting (also locally for algebraic varieties) any Poisson structure is induced by some DQ-algebroid. The global algebraic quantisation is due to Yekutieli \cite{MR2183259} and Van den Bergh \cite{MR2344349}; in the setting of complex manifolds, these have been obtained by Calaque et al. \cite{MR2364075}.
\end{example}
\2
\begin{remark}
 If $\mathscr{A}_\mr{X}$ is a DQ-algebroid, the local notions of being locally free, coherent, flat, etc. make sense for an $\mathscr{A}_\mr{X}$-module $\mathscr{D}$.
\end{remark}
\2
Let $\iota : \mr{Spec}(\C) \to \mr{Spec}(\C\brak)$ be the canonical inclusion, define a $\C$-algebroid $\iota^{*}\mathscr{W}_\mr{X}(0)$ by taking the stack associated with the prestack given on objects and morphisms, respectively, by $$\mathscr{W}_\mr{X}(0)(\mr{U}) \text{ and } \mr{Hom}_{\mathscr{W}_\mr{X}(0)(\mr{U})}(\sigma,\tau)/\hbar\mr{Hom}_{\mathscr{W}_\mr{X}(0)(\mr{U})}(\sigma,\tau).$$ The so defined $\C$-algebroid is an invertible $\mathscr{O}_\mr{X}$-algebroid and we have an equivalence of invertible Poisson algebroids $$\iota^*\Ww(0)\simeq \mathscr{O}_\mr{X}.$$ There is a functor of $\C$-algebroids $$\mathscr{W}_\mr{X}(0) \to \iota^* \mathscr{W}_\mr{X}(0).$$ In particular, we get a functor preversing boundedness and coherence $$\iota^* :\mr{\mathbf{D}}(\mathscr{W}_\mr{X}(0)) \to \mr{\mathbf{D}}(\iota^*\mathscr{W}_\mr{X}(0)) \quad \iota^*:\mathscr{D} \mapsto \C\otimes_{\C\brak} \mathscr{D}.$$
The $\hbar$-localisation of $\mathscr{W}_\mr{X}(0)$ is $\mr{loc}\left(\mathscr{W}_\mr{X}(0)\right) = \C\cbrak\otimes _{\C\brak}\mathscr{W}_\mr{X}(0)$. More generally, we have a functor $$\mr{loc}: \mr{\mathbf{D}}^{\mr b}(\mathscr{W}_\mr{X}(0)) \to \mr{\mathbf{D}}^{\mr b}\left(\mr{loc}\left(\mathscr{W}_\mr{X}(0)\right)\right).$$
\begin{example}
We denote $\mr{loc}(\Ww(0))$ by $\Ww$. This is a $\C\cbrak$-algebroid and will be fundamental for our applications.
\end{example}
\2
\begin{lemma}[Kashiwara-Schapira \cite{MR3012169}]
 If $\mathscr{D} \in \mr{\mathbf{D}}^{\mr b}_\mr{coh}(\mathscr{W}_\mr{X}(0))$, then $\mr{Supp}(\mathscr{D}) = \mr{Supp}(\iota^*\mathscr{D})$. In particular, $\mr{Supp}(\mathscr{D})$ is a closed analytic subset of $\mr{X}$.\\
 If $\mathscr{E} \in \mr{\mathbf{D}}^{\mr b}_\mr{coh}(\mathscr{W}_\mr{X})$, then $\mr{Supp}(\mathscr{E})$ is a closed coisotropic analytic subset of $\mr{X}$.
\end{lemma}
\2
\begin{theorem}[Kashiwara-Schapira \cite{MR3012169}]\label{perf}
 Let $\mr{X}$ be a complex manifold endowed with its canonical deformation quantisation $\mathscr{W}_\mr{X}(0)$. Let $$\mathscr{D}, \mathscr{E} \in \mr{\mathbf{D}}^{\mr b}_\mr{coh}(\mathscr{W}_\mr{X}(0))$$ and suppose that $\mr{Supp}(\mathscr{D})\cap \mr{Supp}(\mathscr{E})$ is compact. Then $\mr{RHom}_{\mathscr{W}_\mr{X}(0)}(\mathscr{D},\mathscr{E})$ is a perfect complex of $\C\brak$-modules.
\end{theorem}
\subsection{Deformations and the dualising complex}
Let $\mr{X}$ be a complex manifold endowed with the DQ algebroid $\mathscr{W}_\mr{X}(0)$. Kashiwara and Schapira \cite{MR3012169} defined a deformation $\mr{D}_\mr{X}^{\mathscr{W}}$ of the sheaf of differential operators $\mr{D}_\mr{X}$. It is a $\C\brak$-subalgebroid of $\mathscr{E}nd_{\C\brak}(\mathscr{W}_\mr{X}(0))$ and there is an equivalence $$\mr{D}_\mr{X}^{\mathscr{W}}\simeq\mr{D}_\mr{X}\brak.$$ This induces an equivalence of stacks of modules over these algebroids. Under this equivalence, $\mathscr{W}_\mr{X}(0)$, regarded as a $\mr{D}_\mr{X}^{\mathscr{W}}$-module, corresponds to $\mathscr{O}_\mr{X}\brak$. On the central fibres we get $$\iota^*\mr{D}_\mr{X}^{\mathscr{W}}\simeq \mr{D}_\mr{X}.$$
This deformation gives rise to a deformation of the canonical bundle of $\mr{X}$ as follows. The above implies that $$\ext^p_{\mr{D}^{\mathscr{W}}_\mr{X}}\left(\mathscr{W}_\mr{X}(0),\mr{D}_\mr{X}^{\mathscr{W}}\right)=0$$ for all $p\not=n\coloneqq\mr{dim}\mr{X}$. Indeed, under the equivalence $\mr{D}^{\mathscr{W}}_\mr{X}\simeq \mr{D}_\mr{X}\brak$, we have $\mathscr{W}_\mr{X}(0)\simeq \mathscr{O}_\mr{X}\brak$, so the claim follows from the standard calculation $$\mr{R}\shom_{\mr{D}_\mr{X}\brak}\left(\mathscr{O}_\mr{X}\brak,\mr{D}_\mr{X}\brak\right)\simeq \mr{K}_\mr{X}\brak[-n].$$
Then the deformation of the canonical bundle is defined as $$\mr{K}_\mr{X}^{\mathscr{W}}=\ext^n_{\mr{D}^{\mathscr{W}}_\mr{X}}\left(\mathscr{W}_\mr{X}(0),\mr{D}_\mr{X}^{\mathscr{W}}\right).$$ This is a bi-invertible $\mathscr{W}_\mr{X}(0)\otimes\mathscr{W}_\mr{X}(0)^{\mr{op}}$-module such that $\iota^*\mr{K}_\mr{X}^{\mathscr{W}}\simeq \mr{K}_\mr{X}$. 
\2
\begin{definition}
We define the $\mathscr{W}_\mr{X}(0)$-dualising complex of $\mr{X}$ as $$\omega^{\mathscr{W}}_\mr{X}\coloneqq \mr{R}\shom_{\mr{D}^{\mathscr{W}}_\mr{X}}\left(\mathscr{W}_\mr{X}(0),\mr{D}^{\mathscr{W}}_\mr{X}\right)[2n]\simeq\mr{K}_\mr{X}^{\mathscr{W}}[n].$$
\end{definition}
By \cite{MR3012169}, it defines a Serre functor $$\Db_{\mr{coh}}(\mathscr{W}_\mr{X}(0))\to \Db_{\mr{coh}}(\mathscr{W}_\mr{X}(0)) \;\; \mathscr{D}\mapsto \mr{K}_\mr{X}^{\mathscr{W}}[n]\otimes_{\mathscr{W}_\mr{X}(0)}\mathscr{D}.$$ 
\begin{proposition}\label{cyprop}
Let $\mr{X}$ be holomorphic symplectic. Then we have an isomorphism of $\mathscr{W}_\mr{X}(0)\otimes\mathscr{W}_\mr{X}(0)^{\mr op}$-modules $\mr{K}_\mr{X}^{\mathscr{W}}\cong \hbar^{\mr{dim}\mr{X}/2}\mathscr{W}_\mr{X}(0)$.
\end{proposition}
\subsection{Holonomic DQ modules}
\begin{definition}
 Let $\mr{X}$ be complex manifold endowed with the canonical DQ-algebroid $\mathscr{W}_\mr{X}(0)$, and let $\mr{Y}$ be a smooth submanifold of $\mr{X}$. A coherent $\mathscr{W}_\mr{X}(0)$-module $\mathscr{D}$ supported on $\mr{Y}$ is called simple if $\iota^*\mathscr{D}$ is concentrated in degree $0$ and $\mr{H}^{0}(\iota^*\mathscr{D})$ is an invertible $\mathscr{O}_{\mr{Y}}$-module.
\end{definition}
\2
\begin{definition}
 Let $\mr{X}$ be a holomorphic symplectic variety equipped with the canonical DQ-algebroid $\mathscr{W}_\mr{X}(0)$. Recall $\mr{loc}(\Ww(0))=\Ww$.
 \begin{enumerate}
  \item A $\Ww$-module is called holonomic if it is coherent and its support is a Lagrangian subvariety of $\mr{X}$.
  \item A $\widehat{\mathscr{W}}_\mr{X}(0)$-module is called holonomic if it is coherent, $\hbar$-torsion free and its $\hbar$-localisation is holonomic.
  \item Let $\mr{L}$ be a smooth Lagrangian. A $\widehat{\mathscr{W}}_\mr{X}$-module $\mathscr{D}$ is called simple holonomic if there exists locally a $\widehat{\mathscr{W}}_\mr{X}(0)$-module $\widetilde{\mathscr{D}}$, simple along $\mr{L}$, which generates it, i.e. $\mr{loc}\left(\widetilde{\mathscr{D}}\right) \simeq \mathscr{D}$.
 \end{enumerate}
\end{definition}
\2
\begin{definition}
Let $\mr{X}$ be holomorphic symplectic. A coherent $\widehat{\mathscr{W}}_\mr{X}$-module $\mathscr{D}$ is good if for any relatively compact open $\mr{U}\subset \mr{X}$, there exists a $\widehat{\mathscr{W}}_\mr{X}(0)$-module $\mathscr{D}^{0}$ generating $\mathscr{D}$ over $\mr{U}$.
\end{definition}
\2
\begin{remark}
Good modules are particularly well-behaved whenever their support is compact. In this case, they are globally generated by a $\widehat{\mathscr{W}}_\mr{X}(0)$-module. 
\end{remark}
As in the contact case, see \cref{microor}, we have a notion of quantised orientations.
\2
\begin{definition}\label{dqor}
Let $\mr{L}$ be an orientable Lagrangian submanifold of $\mr{X}$ and $(\rho:\mr{Y}\to \mr{X}, \Lambda_\mr{L})$ be its contactification. A quantised orientation of $\mr{L}$ is a microdifferential orientation of $\Lambda_\mr{L}$, considered as a $\Ww$-module via the forgetful functor $\rho_*\mr{Mod}_\mr{rh}\left(\widehat{\mr{E}}_\mr{Y}\right)
\to \mr{Mod}_\mr{rh}\left(\Ww\right)$.\end{definition} 
\2
\begin{remark}
Observe that if $\mathscr{D}_\mr{L}$ is a quantised orientation of $\mr{L}$, then it is a good, simple $\widehat{\mathscr{W}}_\mr{X}$-module along $\mr{L}$ such that $\mu_{\Lambda_\mr{L}}(\mathscr{D}_\mr{L})|_{\mr{L}\times 1}$ is, as a coherent sheaf, a square-root of $\mr{K}_\mr{L}$.
\end{remark}
\subsection{Perverse sheaves via deformation quantisation}
Let us recall a fundamental result of Kashiwara and Schapira. For a holomorphic symplectic manifold $\mr{X}$, they associate a perverse sheaf on $\mr{X}$ to any pair of holonomic $\Ww$-modules. Formally, their result goes as follows:
\2
\begin{theorem}[Kashiwara-Schapira \cite{10.2307/40068123}]\label{perv}
 Let $\mr X$ be a holomorphic symplectic variety of dimension $2n$, equipped with the canonical DQ-algebroid $\widehat{\mathscr{W}}_\mr{X}(0)$. Suppose that $\mathscr{D}$ and $\mathscr{E}$ are two holonomic $\widehat{\mathscr{W}}_\mr{X}$-modules. Then the complex $\mr{R}\shom_{\Ww}(\mathscr{D},\mathscr{E})[n]$ is a perverse sheaf.
\end{theorem}
We are going to pair this theorem with an existence result of D'Agnolo and Schapira regarding quantised orientations of Lagrangian submanifolds. Indeed, by the classification \cref{miclass}, we get:
\2
\begin{theorem}[D'Agnolo-Schapira \cite{MR2331247}]\label{existsimplehol}
 Let $\mr{X}$ be a holomorphic symplectic manifold and let $i : \mr{L} \xhookrightarrow{} \mr{X}$ be a spin Lagrangian submanifold. Then, for any choice of a square root $\mr{K}_\mr{L}^{1/2}$ and any $\lambda \in \C$, there exists a unique, up to isomorphism, quantised orientation $\mathscr{D}^{\lambda}_\mr{L}$ corresponding to $\mr{K}_\mr{L}^{1/2}$ equipped with a monodromy automorphism $\mr{exp}(2\pi i \lambda)$.
\end{theorem}
\2
\begin{remark}
 More generally, without assuming the orientability of $\mr{L}$, we can consider the short exact sequence $$1 \to \C^{*} \to \mathscr{O}_{\mr{L}}^{*} \xrightarrow{\mr{dlog}} \mr{d}\mathscr{O}_{\mr{L}}\to 0.$$ It induces a long exact sequence in cohomology which includes $$\mr{H}^{1}(\mr{L},\C^{*}) \to \mr{H}^1(\mr{L},\mathscr{O}_{\mr{L}}^{*}) \xrightarrow{\alpha} \mr{H}^{1}(\mr{L},\mr{d}\mathscr{O}_{\mr{L}}) \xrightarrow{\delta} \mr{H}^{2}(\mr{L},\C^{*}).$$
Recall that invertible $\C$-algebroids on $\mr{X}$ are classified by $\mr{H}^2(\mr{X},\C^*)$. Let $\C_{\mr{K}_\mr{L}^{1/2}}$ be the $\C$-algebroid associated to the class $\delta(\frac{1}{2}\alpha(\mr{c}_1(\mr{K}_\mr{L})))$.\\
Then the classification \cref{miclass}, we see that the theorem above remains true for $\widehat{\mathscr{W}}_\mr{X}\otimes \C_{\mr{K}_\mr{L}^{1/2}}$-modules, i.e. in general we get twisted $\widehat{\mathscr{W}}_\mr{X}$-modules.\\
Notice that $\C_{\mr{K}_\mr{L}^{1/2}}$ is trivial if and only if there exists a line bundle $\mathscr{L}$ such that $\mr{K}_\mr{L}^{\vee}\otimes \mathscr{L}^{\otimes 2}$ admits a flat connection, hence $\mathscr{L}$ lifts to a $\widehat{\mathscr{W}}_\mr{X}$-module via microlocalisation.
\end{remark}

\subsection{Perverse sheaves on \texorpdfstring{$\mr{d}$}--critical loci}

Let $\mr{X}$ be a complex manifold together with a function holomorphic $f$ and consider the intersection
\[
\begin{tikzcd}
 \mr{crit}(f) \arrow[d] \arrow[r]& \Gamma_{\mr{d}f} \arrow[d]
 \\
\mr{X} \arrow[r]
  & \Omega_{\mr{X}},
\end{tikzcd}
\]
 where $\Gamma_{\mr{d}f} \subset \Omega_{\mr{X}}$ is the Lagrangian, given by the graph of $\mr{d}f \in \Gamma(\mr{X},\Omega_{\mr{X}})$ and $\mr{X}$ sits in $\Omega_{\mr{X}}$ as the zero section. \\
Recall that on $\mr{crit}(f) \subset X$, we have a naturally defined perverse sheaf of vanishing cycles 
\begin{equation}\label{canperv}
\mathscr{P}_{\mr{X},f}=\bigoplus_{c\in f(\mr{crit}(f))}\phi_{f-c}\left(\C_{\mr{X}}[\mr{dim}\,\mr{X}]\right),
\end{equation} that is the image of $\C_{\mr{X}}[\mr{dim}\,\mr{X}]$ under the vanishing cycles functor over the critical values of $f$.\\
In this section we explain briefly how to globalise the above remarks with applications to Lagrangian intersection as a primary goal.
\2
\begin{definition}
Let $\mr{X}/\C$ be a complex analytic space and suppose given an embedding of $\mr{X}$ into a complex manifold $\mr{S}$ with ideal sheaf $\mathscr{I}$, then we define the complex of derived $1$-jets $$\mathbf{J}^{1}_\mr{X} \coloneqq \mathscr{O}_{\mr{S}}/\mathscr{I}^{2} \to \Omega_{\mr{S}}|_{\mr{X}},$$ in degrees $-1$ and $0$.
\end{definition}
The next claim shows it is independent of the embedding.
\2
\begin{proposition}
 The complex of derived $1$-jets is naturally quasi-isomorphic to the complex $$\mr{Cone}\left(\mr{d}_\mr{dR}:\mathscr{O}_\mr{X} \to \mathbf{L}_\mr{X}\right),$$ where $\mathbf{L}_\mr{X} \coloneqq \mathscr{I}/\mathscr{I}^{2} \to \Omega_{\mr{S}}|_\mr{X}$ is the truncated cotangent complex.
\end{proposition} 
The cohomology sheaf $\mathscr{S}_{\mr{X}} = \mathscr{H}^{-1}(\mathbf{J}^{1}_\mr{X})$ was used by Joyce to introduce $\mr{d}$-critical loci.
\2
\begin{definition}(Joyce \cite{MR3399099})
 A structure of a $\mr{d}$-critical locus on a complex analytic space $\mr{X}$ is a choice of $s \in \Gamma(\mr{X},\mathscr{S}_\mr{X})$ such that for any $x \in \mr{X}$ there exists an open $\mr{U}$, containing $x$, and a closed embedding of $\mr{U}$ into a smooth $\mr{S}$ together with a function $f$ on $\mr{S}$ such that $s|_{\mr{U}} = f$ in $\Gamma(\mr{U},\mathscr{O}_{\mr{S}}/\mathscr{I}^{2})$ and $\mr{U} = \mr{crit}(f) \subset \mr{S}$. 
\end{definition}
\2
\begin{remark}
The triple $(\mr{U},\mr{S},f)$ is called a (critical) chart for the $\mr{d}$-critical locus $\mr{X}$.
\end{remark}
Let $(\mr{X},s)$ be a $\mr{d}$-critical locus. An important object associated to $\mr{X}$ is its virtual canonical bundle which is used to define orientability for $\mr{d}$-critical loci.\\
Suppose we have a critical chart $(\mr{U},\mr{S},f)$, we can  naively consider the canonical bundle $\mr{K}_\mr{S}|_{\mr{U}_\mr{red}}$ and ask these line bundles glue for a covering by critical charts. The answer is no, but their squares $\mr{K}_\mr{S}^{\otimes 2}|_{\mr{U}_\mr{red}}$ glue to a line bundle on $\mr{X}_\mr{red}$. Indeed if $\mr{X}$ is of the form $\mr{crit}(f)$, then the obstruction is a $\pm1$-cocycle, hence the claim. Formally, we have:
\2
\begin{proposition}[Joyce \cite{MR3399099}]
 Let $(\mr{X},s)$ be a $\mr{d}$-critical locus. There exists a unique line bundle $\mr{K}_{(\mr{X},s)}$ on $\mr{X}_\mr{red}$ such that for any critical chart $(\mr{U},\mr{S},f)$ we have an isomorphism $$\lambda_{(\mr{U},\mr{S},f)}:\mr{K}_{(\mr{X},s)}|_{\mr{U}_\mr{red}} \cong \mr{K}_{\mr{S}}^{\otimes 2}|_{\mr{U}_\mr{red}}.$$ For any étale morphism $\varphi : (\mr{U},\mr{S},f) \to (\mr{V},\mr{T},g)$ of critical charts, that is, $\varphi : \mr{S} \to \mr{T}$ is étale, $\varphi|_{\mr{U}} : \mr{U} \xhookrightarrow{} \mr{V}$ is the canonical inclusion and $f = g\circ \varphi$, we have $$\lambda_{(\mr{U},\mr{S},f)} = \mr{det}(\mr{d}\varphi)^{\otimes 2}|_{\mr{U}_\mr{red}}\circ \lambda_{(\mr{V},\mr{T},g)}|_{\mr{U}_{\mr{red}}}.$$
\end{proposition}
\2
\begin{definition}
Let $(\mr{X},s)$ be a $\mr{d}$-critical locus. An orientation of $(\mr{X},s)$ is a choice of a square-root $\mr{K}_{(\mr{X},s)}^{1/2}$ of its canonical bundle $\mr{K}_{(\mr{X},s)}$.
\end{definition}
\2
\begin{example}
 If $\mr{X}$ is smooth, then $\mr{K}_{(\mr{X},0)} = \mr{K}_{\mr{X}}^{\otimes 2}$. There is an extra $\mr{K}_\mr{X}$ factor because, as a derived scheme, the critical locus $\mr{crit}(0:\mr{X} \to \C)$ is the shifted cotangent bundle $\Omega_\mr{X}[1]$.
\end{example}
As with the canonical bundle of a $\mr{d}$-critical locus $(\mr{X},s)$, the canonical perverse sheaves $\mathscr{P}_{\mr{S},f}$ associated to critical charts $(\mr{U,S},f)$ given by \eqref{canperv} do not glue to a global perverse sheaf on $(\mr{X},s)$. The good news is the obstructions are easy to control. Namely, for an embedding of critical charts $(\mr{U,S},f)\xhookrightarrow{} (\mr{V,T},g)$, the associated perverse sheaves $\mathscr{P}_{\mr{S},f}$ and $\mathscr{P}_{\mr{T},g}$ differ by a $2$-torsion local system which is controlled by a choice of orientation of $(\mr{X},s)$.
\2
\begin{theorem}[Brav et al. \cite{bbdjs}]\label{pervdcrit}
 Let $(\mr{X},s)$ be a $\mr{d}$-critical locus equipped with an orientation $\mr{K}_{(\mr{X},s)}^{1/2}$. Then there exists a perverse sheaf $\mathscr{P}_{(\mr{X},s)}$ on $\mr{X}$ such that if $(\mr{U},\mr{S},f)$ is a chart, then we have a natural isomorphism $$\mathscr{P}_{(\mr{X},s)}|_{\mr{U}} \simeq \mathscr{P}_{\mr{S},f}\otimes_{\C}\mathfrak{or}_{\mr{X/S}},$$ where $\mathfrak{or}_{\mr{X/S}}$ is the local system associated to $\mr{K}_{(\mr{X},s)}^{-1/2}|_{\mr{U_{red}}}\otimes \mr{K}_{\mr{S}}|_{\mr{U_{red}}}$. 
\end{theorem}
\2
\begin{remark}
The perverse sheaf $\mathscr{P}_{(\mr{X},s)}$ is Verdier self-dual and comes equipped with a monodromy automorphism $\mr{T}_{\mr{X},s}:\mathscr{P}_{(\mr{X},s)}\to \mathscr{P}_{(\mr{X},s)}$.\\
In the same paper \cite{bbdjs}, similar statement is proved for monodromic mixed Hodge modules which might be useful to extend our results here beyond the clean intersection case. It would also be interesting to understand how the mixed Hodge module structure could arise via deformation quantisation instead.
\end{remark}
\subsection{Lagrangian intersections as \texorpdfstring{$\mr{d}$}--critical loci}\label{localmodel}
We apply the above to the case of Lagrangian intersection. Let $\mr{X}$ be holomorphic symplectic and suppose $\mr{L}$ and $\mr{M}$ are two Lagrangian submanifolds. By the Lagrangian neighbourhood theorem any $\mr{x} \in \mr{L}$ has open neighbourhoods $\mr{U}$ in $\mr{X}$ and $\tilde{\mr{U}}$ in $\Omega_{\mr{L\cap U}}$ together with a symplectic isomorphism $\Phi:\mr{U}\to \tilde{\mr{U}}$ which identifies $\mr{L}\cap\mr{U}$ with the zero section in the cotangent bundle. We may assume that $\mr{M}$ is transverse to the polarisation given by the projection $\Omega_{\mr{L\cap U}}\to \mr{L\cap U}$. Hence we can write $$\Phi(\mr{M}\cap \mr{U})=\Gamma_s\cap \tilde{\mr{U}}$$ for a closed $1$-form. Shrinking $\mr{U}$ if necessary, we may assume $s=\mr{d}f$ is exact and so $$(\mr{L\cap M})\cap \mr{U}=\mr{crit}(f).$$ In the terminology of the previous section, $(\mr{L\cap M \cap U},\mr{L}\cap \mr{U},f)$ is a critical chart, called an $\mr{L}$-critical chart in \cite{Bussi:2014psa}. We may swap the roles of $\mr{L}$ and $\mr{M}$ and get $\mr{M}$ charts. Using the diagonal $\Delta \xhookrightarrow{} \mr{X}\times\mr{X}$, one gets $\mr{LM}$-charts and the result of Bussi is that these turn $\mr{L\cap M}$ into a $\mr{d}$-critical locus:
\2 
\begin{proposition}[Bussi \cite{Bussi:2014psa}]
 Let $\mr{X}$ be a holomorphic symplectic variety. Suppose given two Lagrangians $\mr{L}$ and $\mr{M}$ in $\mr{X}$. Then then intersection $\mr{L}\cap \mr{M}$ admits a structure of a $\mr{d}$-critical locus $(\mr{L}\cap \mr{M},s)$ with canonical bundle $\mr{K}_{(\mr{L}\cap\mr{M},s)}=\mr{K}_\mr{L}|_\mr{\mr{L}\cap\mr{M}_{red}}\otimes\mr{K}_\mr{M}|_{\mr{\mr{L}\cap\mr{M}_{red}}}$.
\end{proposition}
Pairing this with \cref{pervdcrit}, we get:
\2
\begin{corollary}
 Consider the $\mr{d}$-critical locus $(\mr{L}\cap \mr{M},s)$ and assume that $\mr{K}_\mr{L}|_\mr{\mr{L}\cap\mr{M}_{red}}\otimes\mr{K}_\mr{M}|_{\mr{\mr{L}\cap\mr{M}_{red}}}$ admits a square root. Then there exists a perverse sheaf $$\mathscr{P}_{\mr{L,M}} \simeq \mathscr{P}_{\mr{M,L}}$$ on $\mr{L}\cap \mr{M}$ with the properties described in \cref{pervdcrit}.
\end{corollary}
\2
\begin{lemma}\label{lemmasmoothcase}
  Let $\mr{L,M}$ be two Lagrangians intersecting cleanly, then there is an isomorphism $$\mr{K_{\mr{L}\cap\mr{M}}\otimes K_{\mr{L}\cap\mr{M}} \cong \left.K_{L}\right|_{\mr{L}\cap\mr{M}}\otimes \left.K_{M}\right|_{\mr{L}\cap\mr{M}}}.$$
 \end{lemma}
 \begin{proof}
  Let $\mathbf{E}_{\mr{LM}}\to \mathbf{L}_\mr{L\cap M}$ be the symmetric obstruction theory on the intersection. Recall that $$\mathbf{E}_{\mr{LM}} = \left[\Omega_\mr{X}|_\mr{L\cap M} \xrightarrow{\mr{-res,res}} \Omega_\mr{L}|_\mr{L\cap M}\oplus\Omega_\mr{M}|_\mr{L\cap M}\right],$$ the map $\mathbf{E}_{\mr{LM}}\to \mathbf{L}_\mr{L\cap M}$ is defined via the quasi-isomorphism $$\mathbf{E}_{\mr{LM}} \simeq \left[\mathscr{I}_\mr{LX}/\mathscr{I}_\mr{LX}^2|_\mr{L\cap M} \to \Omega_\mr{M}|_\mr{L\cap M}\right],$$ and the symmetry comes from the holomorphic form on $\mr{X}$. Then, we shall compute the determinant of $\mathbf{E}_{\mr{LM}}$ in two ways. On the one hand, using that $\mr{X}$ has trivial canonical bundle, we get: $$\mr{det}\,\mathbf{E}_{\mr{LM}} = (\mr{det}\, \Omega_\mr{X}|_\mr{L\cap M})^\vee\otimes\mr{det}(\Omega_\mr{L}|_\mr{L\cap M}\oplus \Omega_\mr{M}|_\mr{L\cap M}) \cong \left.\mr{K}_\mr{L}\right|_{\mr{L}\cap\mr{M}}\otimes \left.\mr{K}_\mr{M}\right|_{\mr{L}\cap\mr{M}}.$$ Now, we can also calculate the determinant using the cohomology sheaves of the complex $\mathbf{E}_{\mr{LM}}$ and obtain \[\mr{det}\,\mathbf{E}_{\mr{LM}} = \big(\mr{det}\,\mathscr{H}^{-1}(\mathbf{E}_{\mr{LM}})\big)^\vee\otimes\mr{det}\,\mathscr{H}^0(\mathbf{E}_{\mr{LM}}) \cong \mr{K}_{\mr{L}\cap\mr{M}}\otimes \mr{K}_{\mr{L}\cap\mr{M}}.\qedhere\]
 \end{proof}
\begin{corollary}
 Let $\mr{L}\cap \mr{M}$ be smooth. Then $(\mr{L}\cap \mr{M},s)$ is oriented and for any choice of $\mr{K}_{(\mr{L}\cap \mr{M},s)}^{1/2}$ we have $\mathscr{P}_{\mr{L,M}} = \mathfrak{or}_{\mr{L\cap M}}[\mr{dim}\,\mr{X}]$, where $\mathfrak{or}_{\mr{L\cap M}}$ is the local system associated to $\mr{K}_{(\mr{L}\cap \mr{M},s)}^{-1/2} \otimes \mr{K}_{\mr{L}\cap \mr{M}}$.
\end{corollary}
\begin{proof}
 Indeed, our assumption means that $(\mr{L}\cap\mr{M},0)$ is the unique $\mr{d}$-critical structure on the intersection. Hence, \[\mathscr{P}_{\mr{L,M}} \cong \mathscr{P}_{\mr{L\cap M},0}\otimes_{\C}\mathfrak{or}_{\mr{L\cap M}} = \C_\mr{L\cap M}\left[\mr{dim}(\mr{L}\cap\mr{M})\right]\otimes_\C  \mathfrak{or}_{\mr{L\cap M}}=\mathfrak{or}_{\mr{L\cap M}}\left[\mr{dim}(\mr{L}\cap\mr{M})\right].\qedhere \]
\end{proof}
\subsection{Comparison of the two perverse sheaves}
Given two orientable Lagrangian submanifolds $\mr{L}$ and $\mr{M}$ in $\mr{X}$, we have two perverse sheaves on their intersection $\mr{L}\cap \mr{M}$ arising via deformation quantisation on the one hand and, on the other, via the $\mr{d}$-critical locus structure. The following theorem due to Gunningham and Safronov \cite{saf} compares the two and will be of importance for our applications.
\2
\begin{theorem}[Gunningham-Safronov \cite{saf}]\label{pervref}
 Let $\mr X$ be a holomorphic symplectic manifold of dimension $2n$, equipped with the canonical DQ-algebroid $\widehat{\mathscr{W}}_\mr{X}(0)$. Suppose that $\mr{L}$ and $\mr{M}$ are Lagrangian submanifolds equipped with orientation data $\mr{K}_{\mr{L}}^{1/2}$ and $\mr{K}_\mr{M}^{1/2}$. Let $\mathscr{D}^{\lambda}_{\mr{L}}$ and $\mathscr{D}^{\mu}_{\mr{M}}$ be two simple holonomic $\widehat{\mathscr{W}}_\mr{X}$-modules, supported on $\mr{L}$ and $\mr{M}$, respectively, as in \cref{existsimplehol}. Then we have an isomorphism of perverse sheaves $$\mr{R}\shom_{\widehat{\mathscr{W}}_\mr{X}}\big(\mathscr{D}^{\lambda}_{\mr{L}},\mathscr{D}^{\mu}_{\mr{M}}\big)[n] \xrightarrow{\sim} \C\cbrak \otimes_{\C} \mathscr{P}_{\mr{L,M}}.$$
\end{theorem}

\section{Coherent sheaves and Lagrangian intersections}\label{sec3}

\subsection{Malgrange-Serre resolutions of coherent sheaves}\label{msres}

Let $\mr{X}$ be a complex manifold and denote by $\mathscr{E}_\mr{X}^{p,q}$ the sheaf of smooth differential forms of type $(p,q)$; we denote by $\bar{\partial}$ the $(0,1)$ component of the exterior derivative. Recall that there is an exact sequence $$0 \to \mathscr{O}_\mr{X} \to \mathscr{E}_\mr{X}^{0,0} \xrightarrow{\bar{\partial}} \cdots \xrightarrow{\bar{\partial}} \mathscr{E}_\mr{X}^{0,n}\to 0,$$ where $n$ is the dimension of $\mr{X}$. We have the following important result, conjectured by Serre, and proved by Malgrange.
\2
\begin{theorem}[Malgrange \cite{Malgrange1962-1964}]
Let $\mr{X}$ be a complex manifold. The sheaves $\mathscr{E}_\mr{X}^{p,q}$ are flat $\mathscr{O}_\mr{X}$-modules.
\end{theorem}

For any coherent sheaf $\mathscr{F}$, we thus get an exact sequence \begin{equation}\label{soft}
 0\to \mathscr{F} \to \mathscr{F}\otimes_{\mathscr{O}_\mr{X}}\mathscr{E}_\mr{X}^{0,0} \xrightarrow{\bar{\partial}} \cdots \xrightarrow{\bar{\partial}} \mathscr{F}\otimes_{\mathscr{O}_\mr{X}}\mathscr{E}_\mr{X}^{0,n}\to 0
 \end{equation} which is a soft resolution of $\mathscr{F}$. Hence the cohomology of $\mathscr{F}$ can be identified with the cohomology of the global sections of the complex \labelcref{soft}, i.e. we have an isomorphism $$\mr{H}^{\bullet}(\mr{X},\mathscr{F}) \simeq \mr{H}^{\bullet}\Gamma(\mr{X},\mathscr{F}\otimes_{\mathscr{O}_\mr{X}}\mathscr{E}_\mr{X}^{0,p}).$$
The above resolution of $\mathscr{F}$ first appeared in the local duality results of Malgrange-Serre.
\2
\begin{definition}
Let $\mr{X}$ be a complex manifold. For a coherent sheaf $\F$ on $\mr{X}$, the Malgrange-Serre resolution of $\F$, denoted $\MS(\F)$, is given by the complex $$0\to \mathscr{F}\otimes_{\mathscr{O}_\mr{X}}\mathscr{E}_\mr{X}^{0,0}\xrightarrow{\bar{\partial}} \cdots \xrightarrow{\bar{\partial}} \mathscr{F}\otimes_{\mathscr{O}_\mr{X}}\mathscr{E}_\mr{X}^{0,n}\to 0.$$
\end{definition}

Thus, we have a complex of soft sheaves and we note that this defines an exact functor $$\MS:\mr{Coh}(\mr{X})\to \mr{C}^{b}(\mr{X})$$ and we use the same notation for its extension to $\mathbf{D}^{+}(\mr{X})$. By definition, for any $\mathscr{F} \in \mathbf{D}^{+}(\mr{X})$, we have natural quasi-isomorphisms $$\MS(\mathscr{O}_\mr{X})\otimes_{\mathscr{O}_\mr{X}}\mathscr{F}\simeq \MS(\mathscr{F}).$$
Let $\mr{X}$ and $\mr{Y}$ be complex manifolds and consider a map $f:\mr{X}\to \mr{Y}$. If $\mathscr{F}$ is an $\mathscr{O}_\mr{X}$-module  and $\mathscr{G}$ is an $\mathscr{O}_\mr{Y}$-module, then for any morphism $\alpha:\mathscr{G}\to f_*\mathscr{F}$, we have a pullback $$f^*(\alpha):\MS(\mathscr{G})\to f_*\MS(\mathscr{F})$$ which is a morphism of resolutions above $\alpha$.\\
We are going to be mainly interested in applications to sheaves of dg algebras and this is facilitated by the dg algebra structure on the graded sheaf $$\E_\mr{X}^{\bullet} \coloneqq \bigoplus_p\E_\mr{X}^{0,p},$$ given by the triple $(\E_\mr{X}^{\bullet},\bar{\partial},\wedge)$. Observe that for any $\mathscr{F}$, the resolution $\MS(\mathscr{F})$ is a dg $\mathscr{E}_\mr{X}^\bullet$-module.
\2
\begin{proposition}
Let $\mathscr{A}$ be dg algebra in $\Db(\mr{X})$. Then $\MS(\mathscr{A})$ is a dg algebra, i.e. $\MS$ is a multiplicative resolution.
\end{proposition}

\begin{proof}
Let $m_\mathscr{A}$ be the multiplication of $\mathscr{A}$. We have natural morphisms $$\MS(\mathscr{A})\otimes_{\mathscr{O}_\mr{X}}\MS(\mathscr{A}) \xrightarrow{\wedge} \MS(\mathscr{A}\otimes_{\mathscr{O}_\mr{X}}\mathscr{A})\xrightarrow{\MS(m_\mathscr{A})} \MS(\mathscr{A})$$ whose composition is the multiplication $m_{\MS(\mathscr{A})}$. Moreover, the following square is commutative \[
\begin{tikzcd}[column sep=large]
 \mathscr{A}\otimes_{\mathscr{O}_\mr{X}}\mathscr{A} \arrow[d] \arrow[r, "m_\mathscr{A}"]& \mathscr{A}  \arrow[d]
 \\
\MS(\mathscr{A})\otimes_{\mathscr{O}_\mr{X}}\MS(\mathscr{A})  \arrow[r,"m_{\MS(\mathscr{A})}"]
  & \MS(\mathscr{A}).
\end{tikzcd}
\]
\end{proof}

\subsection{Sheaves on Lagrangian intersections after Behrend-Fantechi}

Let $\mr{L}$ and $\mr{M}$ be submanifolds of $\mr{X}/\C$. 
\2
\begin{proposition}
Let $\mathscr{F}$ be locally free on $\mr{L}$ and $\mathscr{G}$ be coherent on $\mr{M}$. Then, we have a natural isomorphism for all $p \ge 0$: $$\ext^{p}_{\mathscr{O}_\mr{X}}(i_{*}\mathscr{F},j_{*}\mathscr{G}) \cong \ext^{p}_{\mathscr{O}_\mr{X}}(i_{*}\mathscr{O}_{\mr{L}},j_{*}\mathscr{O}_{\mr{M}})\otimes \mathscr{G}|_{L\cap M}\otimes\mathscr{F}^{\vee}|_{L\cap M}.$$
\end{proposition}
\begin{proof}
For any coherent $\mathscr{E}$ on $\mr{X}$, the composition defines a morphism $$\shom_{\mathscr{O}_\mr{X}}(i_*\mathscr{O}_\mr{L},\mathscr{E})\otimes_{\mathscr{O}_\mr{L}}\mathscr{F}^\vee\to \shom_{\mathscr{O}_\mr{X}}(i_*\mathscr{F},\mathscr{E}).$$ We see that this is an isomorphism by working on an affine open. Now for $p>0$, the result follows since both sides are effaceable $\delta$-functors in $\mathscr{E}$. Hence $$\ext^{p}_{\mathscr{O}_\mr{X}}(i_{*}\mathscr{F},j_{*}\mathscr{G}) \cong \ext^{p}_{\mathscr{O}_\mr{X}}(i_{*}\mathscr{O}_{\mr{L}},j_{*}\mathscr{G})\otimes_{\mathscr{O}_\mr{L}}\mathscr{F}^{\vee}.$$ A similar argument gives an isomorphism $$\ext^{p}_{\mathscr{O}_\mr{X}}(i_{*}\mathscr{O}_{\mr{L}},j_{*}\mathscr{G})\cong \ext^{p}_{\mathscr{O}_\mr{X}}(i_{*}\mathscr{O}_{\mr{L}},j_{*}\mathscr{O}_\mr{M})\otimes_{\mathscr{O}_\mr{M}}\mathscr{G}.$$ Composing with the previously established isomorphism, we get the result.
\end{proof}
 Let's recall the following computation, done for $\tor$ sheaves in \cite{MR2030054}. See also \cite{MR2536849} and \cite{MR4678893} for the $\ext$ case.
 \2
\begin{proposition}\label{intersectprop}
 Assuming $\mr{L}\cap \mr{M}$ smooth, we have $$\ext^{p}_{\mathscr{O}_\mr{X}}(i_{*}\mathscr{O}_{\mr{L}},j_{*}\mathscr{O}_{\mr{M}}) \cong \wedge^{p-c} \mathscr{N} \otimes \mr{det}\mathscr{N}_{\mr{\mr{L}\cap\mr{M}/M}},$$ where $c= \mr{rk}\mathscr{N}_{\mr{L\cap M/M}}$, $\mathscr{N} \coloneqq \left.\mathscr{T}_\mr{X}\right|_{\mr{L}\cap \mr{M}}/(\mr{\left.\mathscr{T}_{L}\right|_{L\cap M}+\left.\mathscr{T}_{M}\right|_{L\cap M}})$ is the excess normal bundle.
\end{proposition}
Let $\mr{X}$ be holomorphic symplectic and let $\mr{L,M}$ be Lagrangian submanifolds. Recall the local model of Lagrangian intersections from \cref{localmodel} : locally, $\mr{X}$ is a cotangent bundle $\Omega_\mr{Y}$. We may assume, without loss of generality, that $\mr{L}$ and $\mr{M}$ are given by the graphs of closed $1$-forms on $\mr{Y}$. Suppose $\mr{M}$ is the graph of the zero section. Then, upon shrinking $\mr{Y}$, we may write $\mr{L}=\Gamma_{\mr{d}f}$, that is the graph of the exact $1$-form $\mr{d}f$ for some function $f\in \mathscr{O}_\mr{Y}(\mr{Y})$.\\
The $1$-form $\mr{d}f$ defines a differential on $\Omega^\bullet_\mr{Y}=\Omega^\bullet_\mr{M}$. We are interested in the cohomology sheaves $\mathscr{H}^p\left(\Omega^\bullet_\mr{M},\wedge \mr{d}f\right)$. Denote by $\pi$ the canonical projection $\Omega_\mr{M} \to \mr{M}$ and let $\theta$ be the canonical $1$-form on $\Omega_\mr{M}$. Then the $1$-form $$s=\theta - \pi^*\mr{d}f$$ is a regular section of the horizontal cotangent bundle $\pi^*\Omega_{\mr{M}}$ whose vanishing locus is precisely the graph of $f$, i.e. $\mr{L}$. Hence the Koszul complex $$\mr{Kos}\left(\pi^*\Omega_{\mr{M}},s^\vee\right)$$ is a resolution of $i_*\mathscr{O}_\mr{L}$. Taking duals and restricting to $\mr{M}$, we get that $$\mathscr{H}^p\left(\Omega^\bullet_\mr{M},\wedge \mr{d}f\right) \cong \ext^{p}_{\mathscr{O}_\mr{X}}(i_{*}\mathscr{O}_{\mr{L}},j_{*}\mathscr{O}_{\mr{M}}).$$ This is the local model, and we have a complex $$\left(\ext^{\bullet}_{\mathscr{O}_\mr{X}}(i_{*}\mathscr{O}_{\mr{L}},j_{*}\mathscr{O}_{\mr{M}}),\mr{d}_\mr{dR}\right).$$ 
\begin{theorem}[Behrend-Fantechi \cite{MR2641169}]
Let $\mr{L}$ and $\mr{M}$ be spin Lagrangian submanifolds in $\mr{X}$ equipped with orientations $\mr{K}_\mr{L}^{1/2}$ and $\mr{K}_\mr{M}^{1/2}$. Then there is a unique $\C$-linear differential $$\mr{d}_\mr{BF} :\ext^{p}_{\mathscr{O}_\mr{X}}\big(i_{*}\mr{K_{L}^{1/2}},j_{*}\mr{K_{M}^{1/2}}\big)\to \ext^{p+1}_{\mathscr{O}_\mr{X}}\big(i_{*}\mr{K_{L}^{1/2}},j_{*}\mr{K_{M}^{1/2}}\big)$$ which is locally given by the de Rham differential. 
\end{theorem}
\2
\begin{remark}
\begin{enumerate}
\item The above holds more generally for half twisted local systems.
\item Locally the cohomology of the complex can be expressed as a $\tor_{-1}$-dual of the cohomology of $\Omega^\bullet_\mr{X}\otimes \mr{D}_\mr{X}$ equipped with the $\mathscr{O}_\mr{X}$-linear differential arising from $\mr{d}f$, hence the resulting complex is constructible \cite{MR1181207}.
\end{enumerate}
\end{remark}
Suppose $\mr{L}$ and $\mr{M}$ intersect cleanly. We have an exact sequence: $$\mr{0 \to \mathscr{T}_{\mr{L}\cap\mr{M}} \to \left.\mathscr{T}_{L}\right|_{\mr{L}\cap\mr{M}} \oplus \left.\mathscr{T}_{M}\right|_{\mr{L}\cap\mr{M}} \to \left.\mathscr{T}_{X}\right|_{\mr{L}\cap\mr{M}} \to \Omega_{\mr{L}\cap\mr{M}}} \to 0,$$ hence
 $$\ext^{p}_{\mathscr{O}_\mr{X}}(i_{*}\mathscr{O}_{\mr{L}},j_{*}\mathscr{O}_{\mr{M}}) \cong \Omega^{p-c}_{\mr{L}\cap\mr{M}} \otimes \mr{det}\mathscr{N}_{\mr{\mr{L}\cap\mr{M}/M}}.$$
 The adjunction formula yields an isomorphism $$\mr{det}\mathscr{N}_{\mr{\mr{L}\cap\mr{M}/M}} \cong \left.\mr{K}_{\mr{M}}^{\vee}\right|_{\mr{L}\cap\mr{M}}\otimes \mr{K}_{\mr{L}\cap \mr{M}},$$ hence by \cref{lemmasmoothcase} we obtain $$\mr{det\mathscr{N}_{\mr{L}\cap\mr{M}/L}\otimes det\mathscr{N}_{\mr{L}\cap\mr{M}/M}} \cong \mathscr{O}_{\mr{\mr{L}\cap\mr{M}}}.$$
 Suppose $\mr{K}^{1/2}_\mr{L}$ and $\mr{K}^{1/2}_\mr{M}$ are orientations for $\mr{L}$ and $\mr{M}$. Then, \begin{equation}\label{orind}\mr{or}\left(\mr{K}^{1/2}_\mr{L},\mr{K}^{1/2}_\mr{M}\right) \coloneqq \left(\mr{ \left.K_{L}^{1/2}\right|_{\mr{L}\cap\mr{M}}\otimes \left.K_{M}^{1/2}\right|_{\mr{L}\cap\mr{M}}}\right)^{\vee} \otimes \mr{K_{\mr{L}\cap\mr{M}}}
 \end{equation} is a $2$-torsion line bundle measuring the discrepancy between the canonical orientation on the $\mr{d}$-critical locus $\mr{L}\cap\mr{M}$ and the one induced by the orientations of $\mr{L}$ and $\mr{M}$. The above computations then yield $$\ext^{p}_{\mathscr{O}_\mr{X}}\big(i_{*}\mr{K_{L}^{1/2}},j_{*}\mr{K_{M}^{1/2}}\big) \cong \Omega^{p-c}_{\mr{\mr{L}\cap\mr{M}}} \otimes \mr{or}\left(\mr{K}^{1/2}_\mr{L},\mr{K}^{1/2}_\mr{M}\right).$$ This is precisely the de Rham complex of the flat line bundle $\mr{or}\left(\mr{K}^{1/2}_\mr{L},\mr{K}^{1/2}_\mr{M}\right)$. Hence in the smooth case, the Behrend-Fantechi differential is precisely the de Rham differential.
\subsection{Virtual de Rham complexes}\label{vdrsec}
Let $\mr{X}$ be holomorphic symplectic and let $\mr{L}$ and $\mr{M}$ be two spin Lagrangians in $\mr{X}$. Suppose $\rho : \mr{Y} \to \mr{X}$ is a contactification in a neighbourhood of $\mr{L}\cup \mr{M}$. Consider two orientation $\widehat{\mathscr{W}}_\mr{X}$-modules $\mathscr{D}_{\mr{L}}$ and $\mathscr{D}_{\mr{M}}$ along $\mr{L}$ and $\mr{M}$, respectively. Let $\Lambda_\mr{L}$ and $\Lambda_\mr{M}$ be the Legendrian lifts of the two Lagrangians. By definition, the line bundles $\mu_{\Lambda_\mr{L}}\left(\mathscr{D}_\mr{L}\right)|_{\mr{L}\times 1}$ and $\mu_{\Lambda_\mr{M}}\left(\mathscr{D}_\mr{M}\right)|_{\mr{M}\times 1}$ define orientations on $\mr{L}$ and $\mr{M}$.
\2
\begin{definition}\label{virdr}
The virtual de Rham complex of $\mathscr{D}_\mr{L}$ and $\mathscr{D}_\mr{M}$ is $$\mathscr{DR}^\mr{vir}\left(\mathscr{D}_\mr{L},\mathscr{D}_\mr{M}\right)=\left(\ext_{\mathscr{O}_\mr{X}}^\bullet\left(\mu_{\Lambda_\mr{L}}\left(\mathscr{D}_\mr{L}\right)|_{\mr{L}\times 1},\mu_{\Lambda_\mr{M}}\left(\mathscr{D}_\mr{M}\right)|_{\mr{M}\times 1}\right),\mr{d}_{\mr{BF}}\right).$$
\end{definition} 
\2
\begin{remark}
The virtual de Rham complex is a constructible complex. In the smooth case, it is a perverse sheaf up to shift. See below for details.
\end{remark}
Given $\mathscr{D}_\mr{L}$ and $\mathscr{D}_\mr{M}$ as above, we can consider the line bundle measuring the discrepancy between the canonical orientation on $\mr{L\cap M}$ and the one induced by the deformation quantisation: $$\mr{or}\left(\mu_{\Lambda_\mr{L}}\left(\mathscr{D}_\mr{L}\right)|_{\mr{L}\times 1},\mu_{\Lambda_\mr{M}}\left(\mathscr{D}_\mr{M}\right)|_{\mr{M}\times 1}\right).$$
\begin{corollary}
Suppose $\mr{L}\cap\mr{M}$ is smooth of codimension $c$ in $\mr{L}$. Then, we have $$\ext_{\mathscr{O}_\mr{X}}^p\left(\mu_{\Lambda_\mr{L}}\left(\mathscr{D}_\mr{L}\right)|_{\mr{L}\times 1},\mu_{\Lambda_\mr{M}}\left(\mathscr{D}_\mr{M}\right)|_{\mr{M}\times 1}\right)\cong \Omega^{p-c}_{\mr{L}\cap\mr{M}}\otimes \mr{or}\left(\mu_{\Lambda_\mr{L}}\left(\mathscr{D}_\mr{L}\right)|_{\mr{L}\times 1},\mu_{\Lambda_\mr{M}}\left(\mathscr{D}_\mr{M}\right)|_{\mr{M}\times 1}\right).$$
\end{corollary} 
 The local system corresponding to the bundle $\mr{or}\left(\mu_{\Lambda_\mr{L}}\left(\mathscr{D}_\mr{L}\right)|_{\mr{L}\times 1},\mu_{\Lambda_\mr{M}}\left(\mathscr{D}_\mr{M}\right)|_{\mr{M}\times 1}\right)$ is precisely the shift of Joyce's perverse sheaf $\mathscr{P}_{\mr{L}\cap\mr{M}}[c-n]$.\\
 By the Riemann-Hilbert correspondence the orientation discrepancy bundle with its flat connection gives a resolution $$\mathscr{P}_{\mr{L}\cap\mr{M}}[c-n] \to \Omega_{\mr{L}\cap \mr{M}}^{\bullet}\left(\mr{or}\left(\mu_{\Lambda_\mr{L}}\left(\mathscr{D}_\mr{L}\right)|_{\mr{L}\times 1},\mu_{\Lambda_\mr{M}}\left(\mathscr{D}_\mr{M}\right)|_{\mr{M}\times 1}\right)\right).$$ The complex on the right is precisely the shifted by $c$ virtual de Rham complex, hence we have:
  \begin{proposition}\label{pdr}
In the situation above, there exists a natural quasi-isomorphism $$ \mathscr{P}_{\mr{L}\cap\mr{M}}[-n]\simeq\mathscr{DR}^{\mr{vir}}\left(\mathscr{D}_{\mr{L}},\mathscr{D}_{\mr{M}}\right).$$
 \end{proposition}
 For general intersections, \cref{pdr} no longer holds. The two should nevertheless be related. The following is suggested in \cite{MR2641169} and is motivated by the local version of it in \cite{MR1181207}.
 \2
 \begin{conjecture}
 There exists a whose first page is $\mr{E}_1^\bullet = \mathscr{DR}^{\mr{vir}}\left(\mathscr{D}_{\mr{L}},\mathscr{D}_{\mr{M}}\right)$ which converges to $\mr{H}\mathscr{P}_{\mr{L}\cap\mr{M}}[-n]$.	
 \end{conjecture}  
\section{Differential graded categories and holomorphic symplectic manifolds}\label{sec4}

We introduce several dg categories associated to Lagrangian submanifolds in a holomorphic symplectic manifold. Our goal will then be to explore the relations between them and prove their formality in certain scenarios.
\subsection{The dg category of Lagrangian \texorpdfstring{$\mr{D}$}--branes}
We start with the dg category of $\mr{D}$-branes supported on Lagrangian submanifolds which is a full dg subcategory of $\Db_{\mr{dg}}(\mr{X})$. It is closely related to the Fukaya category via Kontsevich's homological mirror symmetry. Its objects are given by the orientations of the Lagrangians. Thinking of these as gauge fields wrapped on the Lagrangians, the open string spectrum then is given by the $\mr{Ext}$ groups of the line bundles. 
\2
\begin{definition}
Let $(\mr{X},\sigma)$ be a holomorphic symplectic manifold. We define $\mathcal{D}_\mr{Lag}(\mr{X},\sigma)$ to be the full dg subcategory of $\Db_\mr{dg}(\mr{X})$ spanned by the orientations of the (orientable) Lagrangian submanifolds in $\mr{X}$. 
\end{definition}
\2
\begin{itemize}
\item The objects of $\mathcal{D}_\mr{Lag}(\mr{X},\sigma)$ are choices of square-roots $\mr{K}_\mr{L}^{1/2}$ where $\mr{L}$ is an orientable Lagrangian submanifold in $\mr{X}$;
\item For a pair of orientable Lagrangian submanifolds $\mr{L}$ and $\mr{M}$ and two objects associated with these Lagrangians $\mr{K}_\mr{L}^{1/2}$ and $\mr{K}_\mr{M}^{1/2}$, the morphisms are the complexes $$\mathcal{D}_\mr{Lag}\left(\mr{K}_\mr{L}^{1/2},\mr{K}_\mr{M}^{1/2}\right)=\mr{RHom}_{\mathscr{O}_\mr{X}}\left(\mr{K}_\mr{L}^{1/2},\mr{K}_\mr{M}^{1/2}\right).$$
\end{itemize}
We will be mostly concerned with a local version of this category. Let $\mathfrak{L}$ be a collection of orientable Lagrangian submanifolds in $\mr{X}$. 
\2
\begin{definition}
We denote by $\mathcal{D}_\mr{Lag}(\mathfrak{L})$ the full subcategory of $\mathcal{D}_\mr{Lag}(\mr{X},\sigma)$ spanned by objects supported on the Lagrangian submanifolds in $\mathfrak{L}$.
\end{definition}
\subsection{Fukaya dg category via deformation quantisation}
Let $\mr{X}/\C$ be a holomorphic symplectic manifold which we are going to equip with its canonical deformation quantisation algebroid $\widehat{\mathscr{W}}_\mr{X}$. 
\2
 \begin{definition}
Given a holomorphic symplectic manifold $(\mr{X},\sigma)$, the deformation quantisation Fukaya category of $\mr{X}$ is the full dg subcategory $\mathcal{DQ}(\mr{X},\sigma)\subset\Db_\mr{dg}(\widehat{\mathscr{W}}_\mr{X})$ spanned by quantised orientations along the (orientable) Lagrangian submanifolds in $\mr{X}$.
\end{definition}
\2
\begin{itemize}
\item The objects in $\mathcal{DQ}(\mr{X},\sigma)$ are given by good, simple $\widehat{\mathscr{W}}_\mr{X}$-modules $\mathscr{D}_\mr{L}$ along spin Lagrangian submanifolds $\mr{L}$ which are in the image of the forgetful functor $\rho_*\mr{Mod}_\mr{rh}(\widehat{\mr{E}}_\mr{Y})\to \mr{Mod}_\mr{rh}(\widehat{\mathscr{W}}_\mr{X})$.
\item For a pair of Lagrangian submanifolds $\mr{L}$ and $\mr{M}$ and two objects associated with these Lagrangians $\mathscr{D}_\mr{L}$ and $\mathscr{D}_\mr{M}$, the morphism spaces are given by the complexes $$\mathcal{DQ}\left(\mathscr{D}_\mr{L},\mathscr{D}_\mr{M}\right)=\mr{RHom}_{\widehat{\mathscr{W}}_\mr{X}}\left(\mathscr{D}_\mr{L},\mathscr{D}_\mr{M}\right).$$
\end{itemize}

Again, we have a local version of this category. Let $\mathfrak{L}$ be a collection of orientable Lagrangian submanifolds in $\mr{X}$.
\2
\begin{definition}
 We denote by $\mathcal{DQ}_\mathfrak{L}(\mr{X},\sigma)$ the full subcategory of $\mathcal{DQ}(\mr{X},\sigma)$ spanned by objects supported in $\mathfrak{L}$.
 \end{definition}
 Suppose each $\mr{L}$ in $\mathfrak{L}$ is equipped with orientation data $\mr{K}_\mr{L}^{1/2}$. Then there is a canonical full subcategory $\mathcal{DQ}^{\mr{s}}_\mathfrak{L}(\mr{X},\sigma)$ of $\mathcal{DQ}_\mathfrak{L}(\mr{X},\sigma)$, containing a single object for each Lagrangian in $\mathfrak{L}$, defined as follows: Indeed, by \cite{MR2331247} and \cite{saf}, up to isomorphism, there exists a unique quantised orientation $\mathscr{D}^\lambda_\mr{L}$ such that there is an isomorphism $$\mu_{\Lambda_\mr{L}}\left(\mathscr{D}^\lambda_\mr{L}\right)|_{\mr{L}\times 1}\cong \mr{K}_\mr{L}^{1/2}\otimes\C_\lambda$$ of twisted $\mr{D}$-modules with monodromy $\mr{exp}(2\pi i \lambda)$. By \cite{MR4678893}, the morphism complexes are independent of $\lambda$, hence the dg category $\mathcal{DQ}^{\mr{s}}_\mathfrak{L}(\mr{X},\sigma)$ depends only on the (classical) orientations of the Lagrangians.
\2 
\begin{corollary}
Let $(\mr{X},\sigma)$ be a holomorphic symplectic manifold and let $\mathfrak{L}$ be a (countable) collection of orientable compact K\"{a}hler Lagrangian submanifolds with clean pairwise intersections. Then the differential graded category $\mathcal{DQ}^\mr{s}_\mathfrak{L}(\mr{X},\sigma)$ is formal.
\end{corollary}
Our next objective is to lift the $\C\cbrak$-linear dg category $\mathcal{DQ}_\mathfrak{L}(\mr{X},\sigma)$ to a $\C\brak$-linear dg category whose classical limit $\hbar \to 0$ will be of interest.
\2
 \begin{proposition}\label{lift}
Let $\mathfrak{L}$ be a collection of compact Lagrangian submanifolds in $(\mr{X},\sigma)$, then there exists a $\C\brak$-linear differential graded $\widetilde{\mathcal{DQ}}_\mathfrak{L}(\mr{X},\sigma)$ such that $$\mr{loc}\left(\widetilde{\mathcal{DQ}}_\mathfrak{L}(\mr{X},\sigma)\right)=\mathcal{DQ}_\mathfrak{L}(\mr{X},\sigma).$$ 
\end{proposition}
\begin{proof}
Any $\mathscr{D}$ in $\mathcal{DQ}_\mathfrak{L}(\mr{X},\sigma)$ is good and simple along a compact Lagrangian submanifold. Hence, by compactness, it has a global $\widehat{\mathscr{W}}_\mr{X}(0)$-lattice $\widetilde{\mathscr{D}}$. Thus, we see any $\mathscr{D}$ admits a lift along $$\mr{loc}:\Db(\widehat{\mathscr{W}}_\mr{X}(0))\to \Db(\widehat{\mathscr{W}}_\mr{X}).$$ The claim on morphisms follows immediately as $\mr{loc}$ commutes with $\mr{R}\shom$.
\end{proof}
\subsection{Virtual de Rham dg category}
We can use the virtual de Rham complexes introduced in \cref{virdr} to define another dg category associated to the orientable Lagrangian submanifolds of $\mr{X}$. 
\2
\begin{definition}
Let $(\mr{X},\sigma)$ be a holomorphic symplectic manifold. The virtual de Rham category of $\mr{X}$ is the dg category $\mathcal{DR}^{\mr vir}(\mr{X},\sigma)$ spanned by the orientation $\widehat{\mathscr{W}}_\mr{X}$-modules supported on the Lagrangian submanifolds in $\mr{X}$ whose morphism spaces are given by the virtual de Rham complexes $$\mr{DR}^{\mr{vir}}\left(\mathscr{D}_{\mr{L}_1},\mathscr{D}_{\mr{L}_2}\right)=\Gamma\left(\mr{X},\mr{MS}\left(\mathscr{DR}^{\mr vir}(\mathscr{D}_{\mr{L}_1},\mathscr{D}_{\mr{L}_2})\right)\right).$$
\end{definition}
\2
\begin{definition}
For a collection of Lagrangian submanifolds $\mathfrak{L}$, we let $\mathcal{DR}^{\mr{vir}}_\mathfrak{L}(\mr{X},\sigma)$ be the full dg subcategory of $\mathcal{DR}^{\mr{vir}}(\mr{X},\sigma)$ spanned by objects along the Lagrangian submanifolds in $\mathfrak{L}$.
\end{definition}
Given any three Lagrangian submanifolds $\mr{L}_i$ and orientation modules along them $\mathscr{D}_{\mr{L}_i}$, $i=1,2,3$, the multiplication $$\mathscr{DR}^{\mr{vir}}\left(\mathscr{D}_{\mr{L}_2},\mathscr{D}_{\mr{L}_3}\right)\otimes \mathscr{DR}^{\mr{vir}}\left(\mathscr{D}_{\mr{L}_1},\mathscr{D}_{\mr{L}_2}\right) \to \mathscr{DR}^{\mr{vir}}\left(\mathscr{D}_{\mr{L}_1},\mathscr{D}_{\mr{L}_3}\right),$$ induced by the Yoneda product, is compatible with the differential $\mr{d}_\mr{BF}$.\\
Since the Malgrange-Serre resolutions are multiplicative, we see that the pairing of complexes $$\MS\left(\mathscr{DR}^{\mr{vir}}\left(\mathscr{D}_{\mr{L}_2},\mathscr{D}_{\mr{L}_3}\right)\right)\otimes \MS\left(\mathscr{DR}^{\mr{vir}}\left(\mathscr{D}_{\mr{L}_1},\mathscr{D}_{\mr{L}_2}\right)\right) \to \MS\left(\mathscr{DR}^{\mr{vir}}\left(\mathscr{D}_{\mr{L}_1},\mathscr{D}_{\mr{L}_3}\right)\right)$$ is compatible with the splitting of the total differential $\mr{d}_\mr{BF}+\bar{\partial}$.\\
In the next section, we shall use this splitting to prove the formality of the virtual de Rham category in sufficiently nice situations. 

\subsection{Formality}

\begin{definition}A Solomon-Verbitsky collection $\mathfrak{L}$ is a countable collection of orientable compact K\"{a}hler Lagrangian submanifolds $\mr{L}_i$, $i\in \mr{I}$, in $\mr{X}$ with pairwise clean intersections.
\end{definition} 
\2
\begin{theorem}\label{drvir}
Let $\mr{X}/\C$ be a holomorphic symplectic manifold and let $\mathfrak{L}$ be a Solomon-Verbitsky collection. The virtual de Rham category $\mathcal{DR}^{\mr{vir}}_\mathfrak{L}$ of $\mathfrak{L}$ is formal.
\end{theorem}

\begin{proof}
We begin by defining two dg categories. Let $\mathcal{DR}_\mathfrak{L}^{\mr{vir},\bar{\partial}}$ with the same objects as $\mathcal{DR}^{\mr{vir}}_{\mathfrak{L}}$ and set $$\mathcal{DR}_\mathfrak{L}^{\mr{vir},\bar{\partial}}\left(\mathscr{D}_{\mr{L}_1},\mathscr{D}_{\mr{L}_2}\right) = \Gamma_{\bar{\partial}}\left(\mr{X},\MS\left(\mathscr{DR}^{\mr{vir}}\left(\mathscr{D}_{\mr{L}_1},\mathscr{D}_{\mr{L}_2}\right)\right)\right),$$ i.e. the subcomplex of the complex of global sections given by the $\bar{\partial}$-closed elements.\\
The second auxiliary category denoted $\mr{H}_{\bar{\partial}}\left(\mathcal{DR}^{\mr{vir}}_\mathfrak{L}\right)$ is the dg category with the same set of objects as $\mathcal{DR}^{\mr{vir}}_\mathfrak{L}$ and morphisms given by the quotients $$\mr{H}_{\bar{\partial}}(\mathcal{DR}^{\mr{vir}}_\mathfrak{L})\left(\mathscr{D}_{\mr{L}_1},\mathscr{D}_{\mr{L}_2}\right) = \mathcal{DR}_\mathfrak{L}^{\mr{vir},\bar{\partial}}\left(\mathscr{D}_{\mr{L}_1},\mathscr{D}_{\mr{L}_2}\right)/\bar{\partial}\left(\mathcal{DR}^{\mr{vir}}_\mathfrak{L}\left(\mathscr{D}_{\mr{L}_1},\mathscr{D}_{\mr{L}_2}\right)\right)$$ with differential induced by $\mr{d_{dR}}$.
There is a diagram of functors $$\mathcal{DR}^{\mr{vir}}_\mathfrak{L} \overset{\iota}\longleftarrow \mathcal{DR}_\mathfrak{L}^{\mr{vir},\bar{\partial}} \overset{\pi}\longrightarrow \mr{H}_{\bar{\partial}}(\mathcal{DR}^{\mr{vir}}_\mathfrak{L}).$$ The first is given by the natural inclusion, while the latter is the canonical projection. We are going to show that these induce quasi-isomorphisms and moreover $\mr{H}_{\bar{\partial}}(\mathcal{DR}^{\mr{vir}}_\mathfrak{L})$ is a graded category, i.e. $\mr{d_{dR}}$ induces the zero differential on the spaces of morphisms. Hence $\mathcal{DR}^{\mr{vir}}_\mathfrak{L}$ is formal.\\
The differential structure on $\mr{H}_{\bar{\partial}}	\left(\mathcal{DR}^{\mr{vir}}_\mathfrak{L}\right)$ is trivial: indeed pick any $\alpha$ with $\bar{\partial}(\alpha)=0$. The we can apply the $\mr{d_{dR}}\bar{\partial}$-lemma to $\mr{d_{dR}}(\alpha)$ hence getting $$\mr{d_{dR}}(\alpha)=\mr{d_{dR}}\bar{\partial}(\gamma)$$ which implies that $[\mr{d_{dR}}(\alpha)]=0$ in the quotient $$\mathcal{DR}_\mathfrak{L}^{\mr{vir},\bar{\partial}}\left(\mathscr{D}_{\mr{L}_1},\mathscr{D}_{\mr{L}_2}\right)/\bar{\partial}\left(\mathcal{DR}^{\mr{vir}}_\mathfrak{L}\left(\mathscr{D}_{\mr{L}_1},\mathscr{D}_{\mr{L}_2}\right)\right).$$
Since both $\iota$ and $\pi$ are identity on objects, they are essentially surjective on the underlying homotopy categories.\\
Let's begin by showing that $\iota$ induces a quasi-isomorphism between complexes of morphisms. It clearly induces a surjection on cohomology since any $\mr{d_{dR}}+\bar{\partial}$-closed element is also $\bar{\partial}$-closed for degree reasons. To show injectivity, suppose we have a $\bar{\partial}$-closed element $$\alpha= \mr{d_{dR}}(\beta)+\bar{\partial}(\beta).$$ Applying the $\mr{d_{dR}}\bar{\partial}$-lemma, we get $\alpha= \mr{d_{dR}}\bar{\partial}(\gamma)$ and clearly $\bar{\partial}(\gamma)$ is in the kernel of $\bar{\partial}$. \\
Surjectivity of $\pi$: suppose given $$[\alpha] \in \mathcal{DR}_\mathfrak{L}^{\mr{vir},\bar{\partial}}\left(\mathscr{D}_{\mr{L}_1},\mathscr{D}_{\mr{L}_2}\right)/\bar{\partial}\left(\mathcal{DR}^{\mr{vir}}_\mathfrak{L}\left(\mathscr{D}_{\mr{L}_1},\mathscr{D}_{\mr{L}_2}\right)\right),$$ since the differential is trivial, it is enough to find a $\mr{d_{dR}}$-closed representative of $\alpha$. To that end, apply the $\mr{d_{dR}}\bar{\partial}$-lemma to $\mr{d_{dR}}(\alpha)$ to find $\gamma$ such that $$\mr{d_{dR}}(\alpha)=\mr{d_{dR}}\bar{\partial}(\gamma),$$ so $\alpha'=\alpha-\bar{\partial}(\gamma)$ is in the class of $\alpha$ and is $\mr{d_{dR}}$-closed. Injectivity follows immediately from the $\mr{d_{dR}}\bar{\partial}$-lemma since any $\alpha$ with $\bar{\partial}(\alpha)=0$ and $\mr{d_{dR}}(\alpha)=0$ which projects to $0$ in $$\mathcal{DR}_\mathfrak{L}^{\mr{vir},\bar{\partial}}\left(\mathscr{D}_{\mr{L}_1},\mathscr{D}_{\mr{L}_2}\right)/\bar{\partial}\left(\mathcal{DR}^{\mr{vir}}_\mathfrak{L}\left(\mathscr{D}_{\mr{L}_1},\mathscr{D}_{\mr{L}_2}\right)\right)$$ must be $\mr{d_{dR}}$-exact.
\end{proof}
\2
\begin{remark}
We observe that the theorem remains true for any collection of Lagrangian submanifolds for which the $\mr{d}_\mr{BF}\bar{\partial}$-lemma holds.
\end{remark}
\2
\begin{proposition}\label{dqdr}
Let $\mathfrak{L}$ be a Solomon-Verbitsky collection. There is a quasi-isomorphism $$\mathcal{DQ}_\mathfrak{L}\cong \mr{Ind}_{\C\cbrak/\C}\left(\mathcal{DR}^{\mr{vir}}_\mathfrak{L}\right).$$
\end{proposition} 
\begin{proof}
It is enough to show that the quasi-isomorphisms given by \cref{pervref} and \cref{pdr} are multiplicative and this is easily checked locally using Koszul resolutions.
\end{proof}
\2
\begin{corollary}\label{lags}
Let $\mr{X}/\C$ be a holomorphic symplectic manifold. Let $\mathfrak{L}$ be a Solomon-Verbitsky collection. Then the holomorphic Solomon-Verbitsky category $\mathcal{DQ}_\mathfrak{L}$ is formal.
\end{corollary}
\begin{proof}
In view of the quasi-isomorphism $\mathcal{DQ}_\mathfrak{L}\cong \mr{Ind}_{\C\cbrak/\C}\left(\mathcal{DR}^{\mr{vir}}_\mathfrak{L}\right)$, this is an immediate consequence of the preceding theorem.
\end{proof}

Having established the formality of $\mathcal{DQ}_\mathfrak{L}$, we are now going to consider the coherent category $\mathcal{D}_\mr{Lag}(\mathfrak{L})$. Before proving our next result, we will need some preliminary lemmas. The first one gives a sufficient condition for the cohomology of a perfect complex over $\C\brak$ to be free. The second will be used to show that the Hochschild (co)homology of $\widetilde{\mathcal{DQ}}_\mathfrak{L}$ is torsion-free.

\2
 \begin{lemma}\label{freelemma}
 Let $\iota : \C \xhookrightarrow{} \C\brak$ be the inclusion of the central fibre and let $$\mr{C} \in \mr{Perf}\big(\mr{Spec}\big(\C\brak\big)\big).$$ Suppose that for all $i\in \mathbf{Z}$ we have $$\mr{dim}_\C\big(\mr{H}^{i}(\iota^{*}\mr{C})\big) = \mr{dim}_{\C\cbrak}\big(\mr{H}^{i}(\C\cbrak\otimes_{\C\brak} \mr{C})\big).$$
 Then the cohomology $\mr{H}(\mr{C})$ is free (of finite rank) over $\C\brak$.
 \end{lemma}
\begin{proof}
 Consider the exact triangle $$\mr{C} \xrightarrow{\hbar} \mr{C} \to \iota^{*}\mr{C} \xrightarrow{} \mr{C}[1].$$ It induces a long exact sequence in cohomology $$\mr{H}^{i}(\mr{C}) \xrightarrow{\hbar} \mr{H}^{i}(\mr{C}) \to \mr{H}^{i}(\iota^{*}\mr{C}) \to \mr{H}^{i+1}(\mr{C}) \xrightarrow{\hbar} \mr{H}^{i+1}(\mr{C}).$$ Hence there are exact sequences $$ 0 \to \C\otimes_{\C\brak}\mr{H}^{i}(\mr{C}) \to \mr{H}^{i}(\iota^{*}\mr{C}) \to \mr{Tor}_{1}^{\C\brak}\left(\C,\mr{H}^{i+1}(\mr{C})\right) \to 0.$$  In particular, we get that $$\mr{dim}_{\C}\left(\C\otimes_{\C\brak}\mr{H}^{i}(\mr{C})\right) \le \mr{dim}_\C\left(\mr{H}^{i}(\iota^{*}\mr{C})\right).$$
 Since $\mr{H}^{i}(\mr{C})$ is finitely generated, we may write it as $$\mr{H}^{i}(\mr{C}) = \C\brak^{d_i}\oplus \C\brak/\hbar^{k_1} \oplus \cdots \oplus \cdots \oplus\C\brak/\hbar^{k_{r_i}},$$ where $k_1,\cdots,k_{r_i},r_i \in \mathbf{N}$. Notice that $$\mr{dim}_{\C}\left(\C\otimes_{\C\brak}\mr{H}^{i}(\mr{C})\right) = d_i + r_i \text{ and } \mr{dim}_{\C\cbrak}\left(\C\cbrak\otimes_{\C\brak}\mr{H}^{i}(\mr{C})\right)=d_i.$$ It follows by flatness of $\C\cbrak$ that $$\mr{dim}_{\C\cbrak}\left(\C\cbrak\otimes_{\C\brak}\mr{H}^{i}(\mr{C})\right)= \mr{dim}_{\C\cbrak}\left(\mr{H}^{i}(\C\cbrak\otimes_{\C\brak} \mr{C})\right).$$ Hence, $r_i=0$ and $\mr{H}^{i}\left(\mr{C}\right)$ is free.
\end{proof}
\begin{lemma}
Let $\mathfrak{L}$ be a Solomon-Verbitsky collection and let $\widetilde{\mathcal{D}}\subset \widetilde{\mathcal{D}}_\mathfrak{L}$ be a finite full dg subcategory of $\widetilde{\mathcal{DQ}}_\mathfrak{L}$. Then it is formal and quasi-isomorphic to the trivial formal deformation of $\iota^*\widetilde{\mathcal{D}}$, where $\iota :\C\to \C\brak$ is the inclusion of the central fibre.
\end{lemma}
\begin{proof}
Let $\mathcal{D}=\mr{loc}\left(\widetilde{\mathcal{D}}\right)$ and $\mathcal{D}^\mr{vir}\subset \mathcal{DR}^{\mr{vir}}_\mathfrak{L}$ be the full subcategory spanned by the objects of $\mathcal{D}$. We know $\mathcal{D}$ is formal by \cref{lags}. The degeneration of the local-to-global $\mr{Ext}$ spectral sequences implies that $\mr{H}\widetilde{\mathcal{D}}$ is free of finite type over $\C\brak$. Hence we think of it as a (flat) formal deformation of $\mr{H}\iota^*\widetilde{\mathcal{D}}\cong \iota^*\mr{H}\widetilde{\mathcal{D}}$. Since $\mr{H}\widetilde{\mathcal{D}}$ is finite and its morphism spaces are free of finite rank over $\C\brak$, upper semi-continuity gives an inequality \begin{equation}\label{eq1}\mr{dim}_{\C\cbrak}\left(\mr{HH}^{p,q}\left(\mr{H}\mathcal{D}\right)\right)\le \mr{dim}_{\C}\left(\mr{HH}^{p,q}\left(\iota^*\mr{H}\widetilde{\mathcal{D}}\right)\right).\end{equation} On the other hand, the spectral sequence argument shows that there is a multiplicative filtration $\mr{F}$ on $\iota^*\mr{H}\widetilde{\mathcal{D}}$ such that $$\mr{gr}_\mr{F}\left(\iota^*\mr{H}\widetilde{\mathcal{D}}\right) \cong  \mr{H}\mathcal{D}^\mr{vir}.$$ The completed Rees deformation associated to $\mr{F}$ gives a deformation with central fibre $\mr{H}\mathcal{D}^{\mr{vir}}$. Thus
\begin{equation}\label{eq2}\mr{dim}_{\C\cbrak}\left(\mr{HH}^{p,q}\left(\mr{H}\mathcal{D}\right)\right)=\mr{dim}_{\C}\left(\mr{HH}^{p,q}\left(\mr{H}\mathcal{D}^\mr{vir}\right)\right)\ge \mr{dim}_{\C}\left(\mr{HH}^{p,q}\left(\iota^*\mr{H}\widetilde{\mathcal{D}}\right)\right).\end{equation}
Combining \eqref{eq1} and \eqref{eq2} with \cref{freelemma}, we get the compactly supported (of second kind) Hochschild cohomology $\mr{HH}_c^\bullet\left(\mr{H}\widetilde{\mathcal{D}}\right)$ is free over $\C\brak$. Now formality of $\mathcal{D}$ and \cref{a2dg} imply $\widetilde{\mathcal{D}}$ is formal.\\
Abbreviate the Hochschild cohomology of the graded categories as follows: $\mr{HH} \coloneqq \mr{HH}\big(\iota^*\mr{H}\widetilde{\mathcal{D}}\big)$, $\mr{HH}_{\hbar} \coloneqq \mr{HH}\big(\mr{H}\widetilde{\mathcal{D}}\big)$ and let $\mr{HH}_{(\hbar)}$ be the localisation of $\mr{HH}_\hbar$ at $\hbar$. The reduction $\mr{mod}\,\hbar$ gives a map $\mr{HH}_{\hbar} \to \mr{HH}$ and since $\mr{HH}_\hbar$ is free over $\C\brak$, taking a section of the reduction map, we get $\mr{HH}_\hbar \cong \mr{HH}\brak$.\\
  We have an isomorphism of graded categories $$\varphi_{\hbar} :\mr{H}\widetilde{\mathcal{D}}\otimes_{\C\brak}\C\cbrak = \mr{H}\mathcal{D}\cong \mr{Ind}_{\C\cbrak/\C}\left(\mr{H}\mathcal{D}^{\mr{vir}}\right).$$ Write $m_{\hbar}$ and $m_{\mr{dR}}$ for the compositions of $\mr{H}\widetilde{\mathcal{D}}$ and $\mr{H}\mathcal{D}^{\mr{vir}}$, respectively. Then by definition
  \begin{equation}\label{mult}
    m_{\hbar} = \varphi_{\hbar}^{-1}m_{\mr{dR}}(\varphi_{\hbar},\varphi_{\hbar}).
  \end{equation}
 Differentiating \eqref{mult} with respect to $\hbar$ gives
 \begin{equation}\label{mult2}
 m_{\hbar}' = -\mr{d}_{\hbar}(\varphi_{\hbar}^{-1}\varphi_{\hbar}'),
 \end{equation}
 where $\mr{d}_{\hbar}$ is the Hochschild differential of the graded category $\mr{H}\widetilde{\mathcal{D}}$, i.e. $$[m_\hbar']=0\text{ in }\mr{HH}^{2,0}_{(\hbar)}.$$ Since we have a free $\C\brak$-module, we may write $$m_{\hbar} = m + m_{r}\hbar^{r} + \cdots,\, r \ge 1,$$ hence we see that the left side of \eqref{mult2} is $$rm_{r}\hbar^{r-1}+(r+1)m_{r+1}\hbar^{r} + \cdots.$$ Hence we deduce that $$m_{r}+(r+1)/r\cdot m_{r+1}\hbar + (r+2)/r\cdot m_{r+2}\hbar^{2} +\cdots$$ is an $\hbar$-torsion class in $\mr{HH}^{2,0}_\hbar$, lifting the class $[m_r] \in \mr{HH}^{2,0}$, so it must vanish. Thus $[m_r]=0$ too. Letting $m_r =\mr{d}(\psi^{r})$, we define an automorphism $$\mr{H}\widetilde{\mathcal{D}}\to \mr{H}\widetilde{\mathcal{D}}$$ by acting as identity on objects and $\mr{id}-\psi^{r}\hbar^r$ on morphisms. It kills $m_r$. By induction we get $\psi^{i}$ for all $i \ge r$, the infinite composition  $$\psi_\hbar = \big(\big(\mr{id}-\psi^{r}\hbar^r\big)\circ \big(\mr{id}-\psi^{r+1}\hbar^{r+1}\big)\circ \cdots \big)$$ makes sense and we have $\psi_\hbar^{-1}m_\hbar(\psi_\hbar,\psi_\hbar) = m$, showing the triviality of $m_\hbar$.
\end{proof}
We immediately get the following corollary:
\2
\begin{corollary}\label{hfree}
The composition in $\widetilde{\mathcal{DQ}}_\mathfrak{L}$ is $\hbar$-free.
\end{corollary}
With this preliminary results in our toolkit, we are ready to prove the formality of the dg category of Lagrangian $\mr{D}$-branes in a Solomon-Verbitsky collection.
\2
\begin{theorem}\label{main1proof}
Let $\mr{X}/\C$ be a holomorphic symplectic manifold. Suppose $\mathfrak{L}$ is a Solomon-Verbitsky colllection. Then the dg category $\mathcal{D}_{\mr{Lag}}(\mathfrak{L})$ is formal.
\end{theorem}
\begin{proof}
The category $\mathcal{D}_{\mr{Lag}}(\mathfrak{L})$ can be realised as the central fibre of the dg category $\widetilde{\mathcal{DQ}}_\mathfrak{L}$. In addition, we have $$\mr{loc}\left(\widetilde{\mathcal{DQ}}_\mathfrak{L}\right)\cong\mathcal{DQ}_\mathfrak{L}.$$ By \cref{lags} it follows that $\mathcal{DQ}_\mathfrak{L}$ is formal. The strategy therefore is to apply \cref{a2cy} following the philosophy that "generically formal $\implies$ formal".\\
Recall that for any $\widetilde{\mathscr{D}}_\mr{L_i}$ in $\widetilde{\mathcal{DQ}}_\mathfrak{L}$, the space of morphisms $$\mr{RHom}_{\widehat{\mathscr{W}}_\mr{X}(0)}\left(\widetilde{\mathscr{D}}_\mr{L_1},\widetilde{\mathscr{D}}_\mr{L_2}\right)$$ is a perfect $\C\brak$-module and \cref{cyprop} implies that $\widetilde{\mathcal{DQ}}_\mathfrak{L}$ is weak proper $2n$-Calabi-Yau, and so its cohomology $\mr{H}\widetilde{\mathcal{DQ}}_\mathfrak{L}$ is Calabi-Yau too. It will now suffice to show that $\widetilde{\mathcal{DQ}}_\mathfrak{L}$ is flat and the Hochschild homology $\mr{HH}_\bullet\left(\mr{H}\widetilde{\mathcal{DQ}}_\mathfrak{L}\right)$ of the graded category $\mr{H}\widetilde{\mathcal{DQ}}_\mathfrak{L}$ is a torsion-free $\C\brak$-module. It is shown in \cite{MR4678893} that for any $\widetilde{\mathscr{D}}_\mr{L_i}$ in $\mr{H}\widetilde{\mathcal{DQ}}_\mathfrak{L}$, the spaces $$\mr{Ext}_{\widehat{\mathscr{W}}_\mr{X}(0)}\left(\widetilde{\mathscr{D}}_\mr{L_1},\widetilde{\mathscr{D}}_\mr{L_2}\right)$$ are finite free over $\C\brak$ which takes care of flatness. Moreover, by \cref{hfree}, the Hochschild differential of $$\mr{CC}_\bullet\left(\mr{H}\widetilde{\mathcal{DQ}}_\mathfrak{L}\right)$$ is $\hbar$-free, so its cohomology has no $\hbar$-torsion and it therefore torsion-free. Thus $\widetilde{\mathcal{DQ}}_\mathfrak{L}$ is formal by \cref{a2cy} and so is the central fibre $\iota^*\widetilde{\mathcal{DQ}}_\mathfrak{L}\cong \mathcal{D}_{\mr{Lag}}(\mathfrak{L})$.
\end{proof}
\begin{remark}
There is a more direct and somewhat self-contained version of the above proof. By the Calabi-Yau property, we have an isomorphism of $\C\brak$-modules $$\mr{HH}_\bullet^\vee\left(\mr{H}\widetilde{\mathcal{DQ}}_\mathfrak{L}\right)\simeq\mr{HH}^{\bullet+2n}\left(\mr{H}\widetilde{\mathcal{DQ}}_\mathfrak{L}\right).$$ We can calculate the left side via the spectral sequence $$\mr{E}_2^{p,q}=\mr{Ext}_{\C\brak}^p\left(\mr{HH}_q\left(\mr{H}\widetilde{\mathcal{DQ}}_\mathfrak{L}\right),\C\brak\right)\Rightarrow \mr{Ext}^{p+q}_{\C\brak}\left(\mr{CC}_\bullet\left(\mr{H}\widetilde{\mathcal{DQ}}_\mathfrak{L}\right),\C\brak\right).$$ Since $\C\brak$ has global dimension $1$, it degenerates on the second page. As $\mr{HH}_\bullet\left(\mr{H}\widetilde{\mathcal{DQ}}_\mathfrak{L}\right)$ is torsion-free, we see $\mr{HH}^\bullet\left(\mr{H}\widetilde{\mathcal{DQ}}_\mathfrak{L}\right)$ is the $\C\brak$-dual of $\mr{HH}_\bullet\left(\mr{H}\widetilde{\mathcal{DQ}}_\mathfrak{L}\right)$, hence is torsion-free.\\
We claim that the natural map 
\begin{equation}\label{hhcinj}\mr{HH}^\bullet\left(\mr{H}\widetilde{\mathcal{DQ}}_\mathfrak{L}\right)\otimes \C\cbrak \to \mr{HH}^\bullet(\mr{H}\mathcal{DQ}_\mathfrak{L})
\end{equation}is injective. Indeed, $\C\cbrak$ is flat over $\C\brak$ and torsion-free modules over $\C\brak$ are projective, so \cref{hhinj} implies the claim. Now, by \cref{a2kaledin}, we may work by induction on $n$, then injectivity of the map \eqref{hhcinj} and formality of $\mathcal{DQ}_\mathfrak{L}$ imply that the obstruction classes map to $0$ and hence are torsion in the $\C\brak$-module $\mr{HH}^2\left(\mr{H}\widetilde{\mathcal{DQ}}_\mathfrak{L}\right)$ which is torsion-free, so must vanish.
\end{remark}

 \bibliography{formalitybib}
 \bibliographystyle{alpha}
\end{document}